\documentclass[12pt,letterpaper]{article}

\usepackage[utf8]{inputenc}

\usepackage{amssymb,amsmath,amsthm,amsfonts,graphicx,wrapfig,authblk}
\usepackage[pdftex,colorlinks=true,linkcolor=blue,citecolor=blue,urlcolor=red,unicode=true,hyperfootnotes=false,bookmarksnumbered]{hyperref}
\usepackage[margin=1truein]{geometry}

\usepackage{enumerate}

\theoremstyle{plain}
\newtheorem{theorem}{Theorem}[section]
\newtheorem{lemma}[theorem]{Lemma}
\newtheorem{claim}[theorem]{Claim}

\theoremstyle{definition} % This does not use italics for the text.

\newtheorem{remark}[theorem]{Remark}

\makeatletter
\def\@gifnextchar#1#2#3{\let\@tempe#1\def\@tempa{#2}\def\@tempb{#3}%
  \futurelet\@tempc\@gifnch}

\def\@gifnch{\ifx\@tempc\@sptoken\let\@tempd\@tempb%
  \else\ifx\@tempc\@tempe\let\@tempd\@tempa\else\let\@tempd\@tempb\fi\fi\@tempd}

\def\SK@set#1{\left\{#1\right\}}
\def\SK@@set#1#2{\{#1\,:\,
    \begin{array}{@{}l@{}}#2\end{array}
\}}
\def\SK@mset#1{\left\{\!\!\left\{#1\right\}\!\!\right\}}
\def\SK@@mset#1#2{\{\!\!\{#1\,:\,
    \begin{array}{@{}l@{}}#2\end{array}
\}\!\!\}}
\def\BIG@set#1{\Big\{#1\Big\}}
\def\BIG@@set#1#2{\Big\{#1\:\Big|\:
    \begin{array}{@{}l@{}}#2\end{array}
\Big\}}
\newcommand{\Set}[1]{\@gifnextchar\bgroup{\SK@@set{#1}}{\SK@set{#1}}}
\newcommand{\Mset}[1]{\@gifnextchar\bgroup{\SK@@mset{#1}}{\SK@mset{#1}}}
\newcommand{\Bigset}[1]{\@gifnextchar\bgroup{\BIG@@set{#1}}{\BIG@set{#1}}}
\makeatother

\title{Global information from local observations of the noisy voter model on a graph}

\author{Itai Benjamini, %\footnote[1]{Weizmann Institute of Science, itai.benjamini@weizmann.ac.il}, 
Hagai Helman Tov, Maksim Zhukovskii%\thanks{Weizmann Institute of Science, zhukmax@gmail.com}
}
\affil{Weizmann Institute of Science}
\date{}

\begin{document}

\maketitle

\begin{abstract}
We observe the outcome of the discrete time noisy voter model at a single vertex of a graph. We show that certain pairs of graphs can be distinguished by the frequency of repetitions in the sequence of observations. We prove that this statistic is asymptotically normal and that it distinguishes between (asymptotically) almost all pairs of finite graphs. We conjecture that the noisy voter model distinguishes between any two graphs other than stars.
\end{abstract}

\section{Introduction}

In~\cite{BL} the following question was asked. Suppose that we can observe a random process on a graph only locally. Can a non-trivial global observation be distilled? In particular, suppose that we observe a random walk on a graph, but we can only observe it locally. What properties of the graph can be inferred from this? The latter question was partially answered in~\cite{BKL}. It was shown that, by observing the return time sequence of a random walk at the origin in a connected graph $G$, it is possible to determine an eigenvalue $\lambda$ (but not its multiplicity) of $G$ under some additional restrictions on $G$. In this paper, we observe the outcomes at a single vertex of a related process called noisy voter model. It turns out that it is in some sense more efficient than observing return times of a random walk (see Section~\ref{sc:last}). We even conjecture that any pair of graphs other than stars can be distinguished by such observations.  %As an example, they considered two trees presented in Figure~\ref{fg:fg} and showed that the distribution of the return time to the root is the same in both trees, but the eigenvalues have different multiplicities. It is not hard to show that NVM (in particular, $S_2$) distinguishes between these 2 graphs. In particular, if $\varepsilon=0.01$, then for the tree on the left $p_2\geq 0.926$, while for the tree on the right $p_2\leq 0.917$. A natural question to ask, does NVM distinguish between any pair of non-isomorphic graphs that are not stars and have equal eigenvalues?

The voter model was originally introduced to model interactions of two competing for territory populations~\cite{Clifford}, and some of its variants are also useful in genetics~\cite{Sawyer}. This model has been widely studied in the case of infinite graphs. One of the first rigorous study of its behavior on finite graphs was made in~\cite{Donnelly}. In that paper, the authors studied the original continuous time version. The discrete time version of the model was studied in, e.g.,~\cite{Hassin,Nakata}. In~\cite{Granovsky}, noise was introduced in the basic (continuous time) voter model, and a variant of the discrete time noisy voter model was studied in~\cite{Pymar}. 

As mentioned above, we address the following question: is it possible to reconstruct the graph or determine some global properties from an infinite sequence of observations of discrete time voter model at a given vertex $u$? Note that, when there is no noise, the voter model reaches a consensus on a finite graph. In particular, it means that with probability 1 there exists a time moment $t_0$ such that observations at moments $t\geq t_0$ are all equal (actually, if the graph is bipartite, then this is true for time moments with the same parity). Therefore, we assume that a positive noise is introduced in the model. For more details on particle systems see, e.g.,~\cite{Liggett}.\\

Consider a connected graph $G$ (not necessarily finite) and a real number $\varepsilon\in(0,1)$. For a vertex $u\in V(G)$ we denote by $N_G(u)$ the set of all neighbors of $u$ in $G$. Consider the discrete time {\it noisy voter model (NVM)} on $G$ with observations $(X_t(u)=X_t(G,u),\,t\in\mathbb{Z}_+)$ at $u\in V(G)$. The model is defined as follows. Initially, at time $t=0$, all vertices are assigned with independent Bernoulli random variables $X_0(u)\sim\mathrm{Bern}(1/2)$, $u\in V(G)$. Then, for every $t\in\mathbb{N}$, every $u\in V(G)$ (independently of all the other vertices), with probability $\varepsilon$, gets a random {\it opinion} $Y_t(u)\sim\mathrm{Bern}(1/2)$, or, with probability $1-\varepsilon$, chooses an opinion of one of its neighbors chosen uniformly at random. In other words, 
$$
X_t(u)=\xi_t(u)Y_t(u)+(1-\xi_t(u))\sum_{v\in N_G(u)}I(\eta_t(u)=v)X_{t-1}(v),
$$
where $\xi_t(u)\sim\mathrm{Bern}(\varepsilon)$ and $\eta_t(u)\sim U(N_G(u))$ are independent of each other and of all $X$'s and $Y$'s.

In this paper, we mainly focus on the following question. Given two rooted graphs $(G_1,u_1)$, $(G_2,u_2)$, can we distinguish between them by observing the output of NVM on $G_1$ and $G_2$ at $u_1$ and $u_2$ respectively? More formally, we say that $(G_1,u_1)$, $(G_2,u_2)$ are {\it NVM-distinguishable} (or, simply, distinguishable) if there exist a sequence of functions $S_t:\mathbb{R}^t\to\mathbb{R}$, $t\in\mathbb{Z}_+$, and a sequence of non-overlapping sets $\mathcal{S}_1(t),\mathcal{S}_2(t)\subset\mathbb{R}$, $t\in\mathbb{Z}_+$, such that
$$
{\sf P}\biggl[S_t(X_i(G_1,u_1),\,i\leq t)\in\mathcal{S}_1(t),\,S_t(X_i(G_2,u_2),\,i\leq t)\in\mathcal{S}_2(t)\biggr]\to 1\quad\text{as}\quad t\to\infty.
$$
When $S=(S_t,\, t\in\mathbb{Z}_+)$ is such a sequence, we also say that $(G_1,u_1)$, $(G_2,u_2)$ are {\it distinguishable by $S$}.\\

We answer this question positively for several graph families. The desired statistic that we use to distinguish is the average number of repetitions $X_i(u)=X_{i+d}(u)$ denoted by 
$$
S^{(d)}_t:=\frac{1}{t-d}\sum_{i=1}^{t-d}I(X_i(u)=X_{i+d}(u)),
$$
for a fixed parameter $d\in\mathbb{N}$. We also denote $S^{(d)}=(S^{(d)}_t,\, t\in\mathbb{Z}_+)$. In what follows, for $t\in\mathbb{Z}_+$, when $d$ is clear from the context, we omit the superscript and write simply $S_t$.

First of all, we prove that any two non-isomorphic complete graphs are distinguishable. 

\begin{theorem}
Let $u_1,u_2$ be vertices of $G_1=K_{n_1}$ and $G_2=K_{n_2}$. Then $(G_1,u_1)$ and $(G_2,u_2)$ are distinguishable either by $S^{(1)}$ or by $S^{(2)}$.
\label{thm:cliques}
\end{theorem}

The proof is given in Section~\ref{sc:cliques}. The same is true for cycles and integer lattices.

\begin{theorem}
Let $u_1,u_2$ be vertices of $G_1=C_{n_1}$ and $G_2=C_{n_2}$ respectively, $n_1\neq n_2$. Then $(G_1,u_1)$ and $(G_2,u_2)$ are distinguishable either by $S^{(1)}$ or by $S^{(2)}$.
\label{thm:cycles}
\end{theorem}

\begin{theorem}
Let $u_1,u_2$ be vertices of $G_1=\mathbb{Z}^{d_1}$ and $G_2=\mathbb{Z}^{d_2}$ respectively, $d_1\neq d_2$. Then $(G_1,u_1)$ and $(G_2,u_2)$ are distinguishable by $S^{(2)}$.
\label{thm:lattices}
\end{theorem}

The proofs are given in Sections~\ref{sc:cycles}~and~\ref{sc:lattices} respectively. We also distinguish perfect trees with equal arities but different heights as well as perfect trees with different arities but equal heights (other than stars). In what follows, we use the standard notion of a {\it rooted tree}, which is a tree $T$ with a distinguished vertex $v$ which is called a {\it root}, and the related standard terminology (children, leaves, etc). Let us recall that a {\it perfect $k$-ary tree} $(T,v)$ of height $h$ is a tree $T$ with root $v$ such that every non-leaf vertex has exactly $k$ children, and every leaf is at distance exactly $h$ from $v$.

\begin{theorem}
\begin{enumerate}
\item Let $k_1>k_2\geq 2$, $h\geq 2$ be integers, $(T_1,v_1)$, $(T_2,v_2)$ be perfect $k_1$-ary and $k_2$-ary trees respectively of height $h$. Then they are distinguishable by $S^{(2)}$.
\item Let $h_1> h_2\geq 1$, $k\geq 2$ be integers, $(T_1,v_1)$, $(T_2,v_2)$ be perfect $k$-ary trees of heights $h_1$ and $h_2$ respectively. Then they are distinguishable by $S^{(2)}$.
\end{enumerate}
\label{th:trees}
\end{theorem}

The proof is given in Section~\ref{sc:trees}. However, stars rooted at their centers are not distinguishable (see Theorem~\ref{th:bip} in Section~\ref{sc:bip}). We conjecture that stars are the only finite graphs that are not distinguishable. As evidence for the conjecture, we prove that two random graphs on $[n]:=\{1,\ldots,n\}$ chosen uniformly at random and independently are distinguishable with high probability (whp) for almost all $\varepsilon$.

\begin{theorem}
Let $G_n^1,G_n^2$ be two independent and uniformly distributed random graphs on $[n]$. Then whp for every $u_1,u_2\in[n]$, $S^{(2)}$ distinguishes between $(G_n^1,u_1)$ and $(G_n^2,u_2)$ for almost all $\varepsilon$.
\label{th:r_g}
\end{theorem}

The proof is given in Section~\ref{sc:random}.\\

In the next section, we study asymptotic properties of $S^{(d)}$ and explain why this statistic is so useful to distinguish graphs. In Sections~\ref{sc:random}--\ref{sc:bip} we distinguish between graphs from certain families using the properties of $S^{(d)}$ described in Section~\ref{sc:rep}. In particular,~we prove Theorems~\ref{thm:cliques}--\ref{th:r_g}. In Section~\ref{sc:last} we discuss certain interesting future directions.

\section{Repetitions}
\label{sc:rep}

The reason why $S^{(d)}_t$ is useful is that it converges in probability (see Section~\ref{sc:main_proof}) to a value that is uniquely determined by the limit probability (as $t\to\infty$) that $X_t(u)$ and $X_{t+d}(u)$ both receive an opinion of a vertex $v$, that shared its opinion at some moment $s\leq t$. Then, if these limits are different on $(G_1,u_1)$ and $(G_2,u_2)$, then we immediately get that $S^{(d)}$ distinguishes between $(G_1,u_1)$ and $(G_2,u_2)$.

More formally, let $\mathcal{D}_G$ be a directed graph on $V(G)\times\mathbb{Z}_+$ defined as follows: $((u,t),(v,t-1))\in E(\mathcal{D}_G)$, if $\xi_t(u)=0$ and $\eta_t(u)=v$. There are no other edges in $\mathcal{D}_G$. Fix a positive integer $d$. Let $\pi_t:=\pi_t(G,u)$ be the longest path in $\mathcal{D}_G$ that starts on $(u,t)$. Let $p_{t;d}:=p_{t;d}(G,u)$ be the probability that $\pi_t$ meets $\pi_{t+d}$. In other words, $p_{t;d}$ is the probability that there exists a vertex $v$ of $G$ such that $u$ receives its opinion $X_s(v)$, $s\leq t$, at time moments $t$ and $t+d$. Note that $p_{t;d}$ is not decreasing in $t$, and so there exists 
$$
p_{d}:=p_{d}(G,u)=\lim_{t\to\infty}p_{t;d}(G,u).
$$ 

\begin{theorem}
If $(G_1,u_1)$, $(G_2,u_2)$ are rooted graphs such that, for some $d\in\mathbb{N}$,%, for some integer $m\geq 2$ and positive integers $d_1<\ldots<d_{m-1}$, 
$$
p_d(G_1,u_1)\neq p_d(G_2,u_2),
%\label{eq:dif_lim_prob}
$$
then $(G_1,u_1)$ and $(G_2,u_2)$ are distinguishable by $S^{(d)}$.
\label{th:NVM_main}
\end{theorem}

In Section~\ref{sc:main_proof}, we prove this theorem. After that, in Section~\ref{sc:normal}, we show that $S^{(d)}_t$ (or, in some cases, when the behavior of $\mathrm{Var}S^{(d)}_t$ is atypical, its slight modification) is asymptotically normal.\\

Theorem~\ref{th:NVM_main} is used to prove Theorems~\ref{thm:cliques}--\ref{th:r_g}. However, as we show in Section~\ref{sc:bip}, this tool is not always sufficient to distinguish. In particular, let $G_1=K_{m_1,n_1}$ and $G_2=K_{m_2,n_2}$ be complete bipartite graphs with partitions $V(G_1)=V_1^1\sqcup V_1^2$ and $V(G_2)=V_2^1\sqcup V_2^2$ respectively, $|V_1^1|=n_1$, $|V_1^2|=m_1$, $|V_2^1|=n_2$, $|V_2^2|=m_2$, $u_1\in V_1^1$, $u_2\in V_2^1$. Then there exist arbitrarily large integers $m_1,n_1,m_2,n_2$ and $\varepsilon\in(0,1)$ (note that $\varepsilon$ can be chosen arbitrarily close to 0) such that $(G_1,u_1)$ and $(G_2,u_2)$ are NVM-distinguishable, but for every $d\in\mathbb{N}$, $p_d(G_1,u_1)=p_d(G_2,u_2)$. In order to overcome this, we introduce auxiliary statistics in Section~\ref{sc:other}.\\

\subsection{Proof of Theorem~\ref{th:NVM_main}}
\label{sc:main_proof}

%Let $m\geq 2$ be the minimum integer such that there exist positive integers $d_1<\ldots<d_{m-1}$ satisfying (\ref{eq:dif_lim_prob}). 

Let $(X_t=X_t(G,u),\,t\in\mathbb{Z}_+)$ be the sequence of observations of NVM on $G$ at $u$. %Consider
%$$
%S_{d}(X_1,\ldots,X_t)=\frac{1}{t-d_{m-1}}\sum_{i=1}^{t-d_{m-1}}I(X_i=X_{i+d_1}=X_{i+d_2}=\ldots=X_{i+d_{m-1}}).
%$$
Let us find $\lim\limits_{t\to\infty}{\sf P}(X_t=X_{t+d})$: 
\begin{align}
{\sf P}(X_t=X_{t+d})\,\,
&=\,\,{\sf P}(X_t=X_{t+d}|\pi_t\cap\pi_{t+d}\neq\varnothing){\sf P}(\pi_t\cap\pi_{t+d}\neq\varnothing)\notag\\
&\quad\quad\quad\quad\quad\quad\quad\quad+{\sf P}(X_t=X_{t+d}|\pi_t\cap\pi_{t+d}=\varnothing){\sf P}(\pi_t\cap\pi_{t+d}=\varnothing)\label{eq:prob_duplicate}\\
&=\,\, p_{t;d}+\frac{1}{2}(1-p_{t;d}).\notag
\end{align}
Therefore,
\begin{equation}
 \lim_{t\to\infty}{\sf P}(X_t=X_{t+d})=p_d+\frac{1}{2}(1-p_d)=\frac{1}{2}(1+p_d).
\label{eq:limit_expectation}
\end{equation}
The following lemma finishes the proof of Theorem~\ref{th:NVM_main}.\\

\begin{lemma}
For every $d\in\mathbb{N}$,
\begin{equation}
 S^{(d)}_t\stackrel{{\sf P}}\to \frac{1}{2}(1+p_d),\quad t\to\infty.
\label{eq:lll}
\end{equation}
\label{lm:lll}
\end{lemma}

\begin{proof} For every $t$, consider $\zeta_t=I(X_t=X_{t+d})$. From (\ref{eq:limit_expectation}), ${\sf E}\zeta_t\to\frac{1}{2}(1+p_d)$ as $t\to\infty$. Therefore, ${\sf E}S_t=\frac{1}{t-d}\sum_{s=1}^{t-d}{\sf E}\zeta_s\to\frac{1}{2}(1+p_d)$ as $t\to\infty$ as well. By Chebyshev's inequality, for every $\delta>0$ and $t$ large enough,
\begin{equation}
 {\sf P}\left(\left|S_t-\frac{1}{2}(1+p_d)\right|>\delta\right)\leq
 {\sf P}\left(|S_t-{\sf E}S_t|>\frac{1}{2}\delta\right)\leq\frac{4\mathrm{Var}S_t}{\delta^2}. 
\label{eq:Cheb}
\end{equation}
It remains to prove that $\mathrm{Var}S_t\to 0$ as $t\to\infty$. 

Let $\beta>0$. Let us choose $s_0$ such that, for all $s\geq s_0$,
$$
(1-\varepsilon)^s<\frac{\delta^2\beta}{12}\quad\text{ and }\quad{\sf E}\zeta_{s-1}\leq{\sf E}\zeta_{t+s+d}+\frac{\delta^2\beta}{12{\sf E}\zeta_t}\text{ for all }t
$$
(note that ${\sf E}\zeta_t\geq\frac{1}{2}>0$ for all $t$). Let us denote by $|\pi_t|$ the length of $\pi_t$. Then, clearly, for every $t$ and $s\geq s_0$,
\begin{align}
 {\sf E}\zeta_t\zeta_{t+s+d} \,\, & = \,\, {\sf P}(X_t=X_{t+d},X_{t+s+d}=X_{t+s+2d})\notag\\
 & = \,\,{\sf P}(X_t=X_{t+d},X_{t+s+d}=X_{t+s+2d},|\pi_{t+s+d}|<s)\notag\\
 &\quad\quad\quad\quad\quad\quad\quad\quad + {\sf P}(X_t=X_{t+d},X_{t+s+d}=X_{t+s+2d},|\pi_{t+s+d}|\geq s) \label{eq:cov_from_above}\\
 &\leq \,\,{\sf P}(X_t=X_{t+d}){\sf P}(X_{s-1}=X_{s-1+d})+{\sf P}(|\pi_{t+s+d}|\geq s)\notag\\
 & = \,\, (1-\varepsilon)^s+{\sf E}\zeta_t{\sf E}\zeta_{s-1} \,\, \leq \,\, {\sf E}\zeta_t{\sf E}\zeta_{t+s+d}+\frac{\delta^2\beta}{6}.\notag
\end{align}
Summing up, for all $t$ large enough,
$$
  \mathrm{Var}S_{t+d}=\frac{1}{t^2}\sum_{1\leq i,j\leq t}\mathrm{cov}(\zeta_i,\zeta_j)<\frac{\delta^2\beta}{6}+
  \frac{1}{t^2}\sum_{|i-j|\leq s_0}\mathrm{cov}(\zeta_i,\zeta_j)<\frac{\delta^2\beta}{6}+\frac{2s_0+1}{t}<\frac{\delta^2\beta}{4}.
$$
Combining this with (\ref{eq:Cheb}), we get the desired convergence in probability.
\end{proof}

\subsection{Expectation and variance}

As above, we set $\zeta_t=I(X_t=X_{t+d})$. Lemma~\ref{lm:lll} states that $\frac{\zeta_1+\ldots+\zeta_{t-d}}{t-d}\stackrel{\sf P}\to\frac{1}{2}(1+p_d)$. In order to optimise the denominator in this law of large numbers, let us bound from above $\mathrm{Var}S_t$. Applying the same argument as in (\ref{eq:cov_from_above}), we get that, for all $s\geq 1$,
\begin{multline*}
 {\sf E}\zeta_t\zeta_{t+s+d}\,\,=\,\,
  {\sf P}(X_{t}=X_{t+d}){\sf P}(X_{t+s+d}=X_{t+s+2d},|\pi_{t+s+d}|<s)\\
 +{\sf P}(X_{t}=X_{t+d},X_{t+s+d}=X_{t+s+2d},|\pi_{t+s+d}|\geq s).
\end{multline*}
Therefore,
\begin{equation}
 |\mathrm{cov}(\zeta_t,\zeta_{t+s+d})|\leq{\sf P}(|\pi_{t+s+d}|\geq s)=(1-\varepsilon)^s
\label{eq:cov_zeta_above}
\end{equation}
implying that
$$
 \mathrm{Var}(\zeta_1+\ldots+\zeta_t)=\sum_{j_1,j_2\in[t]}\mathrm{cov}(\zeta_{j_1},\zeta_{j_2})=\sum_{|j_1-j_2|\leq d}\mathrm{cov}(\zeta_{j_1},\zeta_{j_2})+
 \sum_{|j_1-j_2|>d}\delta(j_1,j_2),
$$
where $|\delta(j_1,j_2)|\leq(1-\varepsilon)^{|j_1-j_2|-d}$. Since $\mathrm{cov}(\zeta_{j_1},\zeta_{j_2})\leq 1$, we get
\begin{equation}
 \mathrm{Var}(\zeta_1+\ldots+\zeta_t)\leq(2d+1)t+2t\sum_{s=1}^{\infty}(1-\varepsilon)^s=O\left(t\right).
\label{eq:var_above_sublinear}
\end{equation}

May $\mathrm{Var}(\zeta_1+\ldots+\zeta_t)$ be sublinear? We do not have an answer. However, we prove that it is either linear in $t$ or $O(\ln t)$.\\

\begin{lemma} For every $d\in\mathbb{N}$,
\begin{enumerate}
\item ${\sf E}S^{(d)}_t=\frac{1}{2}(1+p_d)-O\left(\frac{1}{t}\right)$;
\item either $\mathrm{Var}(\zeta_1+\ldots+\zeta_t)=O(\ln t)$,\\ or $\mathrm{Var}(\zeta_1+\ldots+\zeta_t)=(\sigma+o(1))t$\,\, for some $\sigma\in(0,1)$.
\end{enumerate}
\label{lm:var_linear}
\end{lemma}

Note that we may easily replace $S^{(d)}$ with another statistic $\tilde S^{(d)}$ such that Lemma~\ref{lm:lll} and Theorem~\ref{th:NVM_main} hold for $\tilde  S^{(d)}$ as well but $\mathrm{Var}\tilde S^{(d)}_t=\Theta(1/t)$. Indeed, let $\kappa$ be so large that 
\begin{equation}
4(1-\varepsilon)^{\kappa-d}<(1-(1-\varepsilon)^{\kappa})(1+p_d)\left(1-\frac{1}{2}(1+p_d)\right)
\label{eq:kappa_choice}
\end{equation}
and 
$$
\tilde S^{(d)}_t=\frac{1}{t-d}\sum_{i=1}^{t-d}I(X_{\kappa i}=X_{\kappa i+d}).
$$
Note that 
\begin{equation}
{\sf P}(X_s=X_{s+d})\leq{\sf P}(X_{s+1}=X_{s+1+d})\leq{\sf P}(X_s=X_{s+d})+(1-\varepsilon)^{s+1}
\label{eq:bern_param_bounds_recur_bound}
\end{equation}
since, due to~(\ref{eq:prob_duplicate}),
\begin{align*}
 {\sf P}(X_{s+1}=X_{s+1+d}) \,\, 
 &=\,\,\frac{1}{2}(1+{\sf P}(\pi_{s+1}\cap\pi_{s+1+d}\neq\varnothing))\\
 &\leq\,\,\frac{1}{2}(1+{\sf P}(\pi_{s}\cap\pi_{s+d}\neq\varnothing)+{\sf P}(|\pi_{s+1}|=s+1))\\
 &=\,\,{\sf P}(X_s=X_{s+d})+\frac{1}{2}(1-\varepsilon)^{s+1}.
\end{align*}
Therefore, for every $s$, 
\begin{equation}
\max\left\{\frac{1}{2},\frac{1}{2}(1+p_d)-\frac{(1-\varepsilon)^s}{\varepsilon}\right\}\leq{\sf P}(X_s=X_{s+d})\leq \frac{1}{2}(1+p_d)
\label{eq:bern_param_bounds}
\end{equation}
implying that (since  ${\sf P}(X_s=X_{s+d})\geq\frac{1}{2}$ and $f(x)=x-x^2$ decreases on $[1/2,1]$)
$$
 \mathrm{Var}\zeta_s={\sf P}(X_s=X_{s+d})-[{\sf P}(X_s=X_{s+d})]^2\geq \frac{1}{2}(1+p_d)\left(1-\frac{1}{2}(1+p_d)\right).
$$
Then, arguing in a similar way to (\ref{eq:var_above_sublinear}), from (\ref{eq:cov_zeta_above}) and (\ref{eq:kappa_choice}), we get
\begin{align*}
 \mathrm{Var}(\zeta_{\kappa}+\zeta_{2\kappa}+\ldots+\zeta_{\kappa t}) \,\, & \geq \,\,
 \mathrm{Var}\zeta_{\kappa}+\ldots+\mathrm{Var}\zeta_{\kappa t}-2t\sum_{i=1}^{\infty}(1-\varepsilon)^{\kappa i-d}\\
 & \geq \,\,
 \frac{t}{2}(1+p_d)\left(1-\frac{1}{2}(1+p_d)\right)-\frac{2t(1-\varepsilon)^{\kappa-d}}{1-(1-\varepsilon)^{\kappa}} \,\,
 =\,\,\Omega(t).
\end{align*}

\begin{remark}
\label{rk:tilde_S}
In the same way as in the proof of Lemma~\ref{lm:var_linear} below, we may show that there exists $\lim_{t\to\infty}\frac{\mathrm{Var}(\zeta_{\kappa}+\zeta_{2\kappa}+\ldots+\zeta_{\kappa t})}{t}$.\\
\end{remark}

\subsubsection{Proof of Lemma~\ref{lm:var_linear}}

The asymptotics of the expectation immediately follows from (\ref{eq:bern_param_bounds}): 
$$
 \frac{1}{2}(1+p_d)-\frac{1}{\varepsilon^2(t-d)}\leq {\sf E}S_t=\frac{1}{t-d}\sum_{s=1}^{t-d}{\sf P}(X_s=X_{s+d})\leq\frac{1}{2}(1+p_d).
$$

Let us now switch to the variance. Assume that $\mathrm{Var}(\zeta_1+\ldots+\zeta_t)\neq O(\ln t)$. %Without loss of generality we may require $\frac{1}{\ln t}\mathrm{Var}(\zeta_1+\ldots+\zeta_t)\to\infty$ as $t\to\infty$. 
We need the following observation that immediately follows from (\ref{eq:cov_zeta_above}) and from the existence of $\lim_{t\to\infty}\mathrm{cov}(\zeta_t,\zeta_{t+s})$ for every fixed $s$.\\

\begin{claim} $\,$ 
\begin{enumerate}
\item There exists a finite
$$
 \lim_{s,t\to\infty}\mathrm{cov}(\zeta_t,\zeta_{t+1}+\ldots+\zeta_{t+s})=:c.%\leq d+\frac{1-\varepsilon}{\varepsilon}.
$$
%Moreover, there exists $\lim_{t\to\infty}\mathrm{cov}(\xi_t,\xi_{t+1}+\ldots+\xi_{t+s})=:c_s$.
\item For every $t$, there exists a finite
$$
 \lim_{s_1,s_2\to\infty}\mathrm{cov}(\zeta_{t+1}+\ldots+\zeta_{t+s_1},\zeta_{t+s_1+1}+\ldots+\zeta_{t+s_1+s_2})=:C,
$$
the convergence is uniform over all $t$, and $C$ does not depend on $t$.
%Moreover, there exists $\lim_{t\to\infty}\mathrm{cov}(\xi_{t-s_2}+\ldots+\xi_t,\xi_{t+1}+\ldots+\xi_{t+s_1})$.\\
\end{enumerate}
\label{cl:var_1}
\end{claim}

Another observation that we need is that the variance does not change much in response to translations.\\

\begin{claim} Let $t,s,\tau\in\mathbb{N}$. Then
$$
 |\mathrm{Var}(\zeta_{t+1}+\ldots+\zeta_{t+s})-\mathrm{Var}(\zeta_{t+\tau+1}+\ldots+\zeta_{t+s+\tau})|\leq
 8\frac{s}{\varepsilon}(1-\varepsilon)^{t}.
$$
\label{cl:var_2}
\end{claim}

\begin{proof} In the same way as in (\ref{eq:bern_param_bounds_recur_bound}), for every $t_1\leq t_2$,
\begin{align*}
 |\mathrm{cov}(\zeta_{t_1},\zeta_{t_2})-\mathrm{cov}(\zeta_{t_1+\tau},\zeta_{t_2+\tau})|\,\, &\leq\,\,
 |{\sf P}(\zeta_{t_1}=1,\zeta_{t_2}=1)-{\sf P}(\zeta_{t_1+\tau}=1,\zeta_{t_2+\tau}=1)|\\
 &\quad+|{\sf P}(\zeta_{t_1}=1){\sf P}(\zeta_{t_2}=1)-{\sf P}(\zeta_{t_1+\tau}=1){\sf P}(\zeta_{t_2+\tau}=1)|\\
&\leq\,\, (1-\varepsilon)^{t_1}+(1-\varepsilon)^{t_1}+(1-\varepsilon)^{t_2}+(1-\varepsilon)^{t_1+t_2}\\
&<\,\,4(1-\varepsilon)^{t_1}.
\end{align*}
This immediately implies the desired inequality. 
\end{proof}

We choose sufficiently small $\delta_1>0,\delta_2>0$. Let $T_0$ be so large that, for any $T\geq T_0$, 
\begin{equation}
 8\frac{T}{\varepsilon}(1-\varepsilon)^{\delta_1 T}<\delta_2,
\label{eq:cond_delta_T_2}
\end{equation}
\begin{equation}
 \frac{1}{\varepsilon^2(1-\varepsilon)^d}\sum_{j=1}^{\infty}j(1-\varepsilon)^{\frac{T}{2\delta_1^j}}<\delta_2.
\label{eq:cond_delta_T_3}
\end{equation}
Due to Claim~\ref{cl:var_1}, there exists $T_0'\geq T_0$ so large that
\begin{equation}
 |\mathrm{cov}(\zeta_t,\zeta_{t+1}+\ldots+\zeta_{t+s})-c|<\delta_2,\quad
 \text{for all }t,s\geq\lfloor\delta_1 T_0'\rfloor,
\label{eq:cond_delta_T_4}
\end{equation}
\begin{equation}
 |\mathrm{cov}(\zeta_{t+1}+\ldots+\zeta_{t+s_1},\zeta_{t+s_1+1}+\ldots+\zeta_{t+s_1+s_2})-C|<\delta_2,\quad
 \text{for all }s_1,s_2\geq\lfloor\delta_1 T_0'\rfloor\text{ and all }t.
\label{eq:cond_delta_T_5}
\end{equation}
By assumption, $\frac{1}{\ln t}\mathrm{Var}(\zeta_1+\ldots+\zeta_t)$ can be arbitrarily large. Therefore, there exists arbitrarily large $t$ such that 
$$
\max_{s\in[(1/\delta_1)^t,(1/\delta_1)^{t+1}]}\frac{1}{\ln s}\mathrm{Var}(\zeta_1+\ldots+\zeta_s)\geq\max_{s\in[(1/\delta_1)^{t-1},(1/\delta_1)^{t}]}\frac{1}{\ln s}\mathrm{Var}(\zeta_1+\ldots+\zeta_s),
$$
and
$$
\max_{s\in[(1/\delta_1)^t,(1/\delta_1)^{t+1}]}\frac{1}{\ln s}\mathrm{Var}(\zeta_1+\ldots+\zeta_s)\stackrel{t\to\infty}\longrightarrow\infty.
$$
Thus, %, due to Claim~\ref{cl:var_1},
 there exists $T\geq T_0'$ so large that %. Since $\frac{1}{\ln t}\mathrm{Var}(\xi_1+\ldots+\xi_t)\to\infty$ and due to Claim~\ref{cl:var_1}, we may assume that 
\begin{equation}
\mathrm{Var}(\zeta_1+\ldots+\zeta_T)-\mathrm{Var}(\zeta_1+\ldots+\zeta_{\lceil\delta_1 T\rceil-1})>5|C|+2|c|+11\delta_2+1
\label{eq:cond_delta_T_1}
\end{equation}
since the left hand side could be arbitrarily large.
%\begin{equation}
% 8\frac{T}{\varepsilon}(1-\varepsilon)^{\delta_1 T}<\delta_2,
%\label{eq:cond_delta_T_2}
%\end{equation}
%\begin{equation}
% \frac{1}{\varepsilon^2(1-\varepsilon)^d}\sum_{j=1}^{\infty}j(1-\varepsilon)^{\frac{T}{2\delta_1^j}}<\delta_2,
%\label{eq:cond_delta_T_3}
%\end{equation}
%\begin{equation}
% |\mathrm{cov}(\zeta_t,\zeta_{t+1}+\ldots+\zeta_{t+s})-c|<\delta_2,\quad
% \text{for all }t,s\geq\lfloor\delta_1 T\rfloor,
%\label{eq:cond_delta_T_4}
%\end{equation}
%\begin{equation}
% |\mathrm{cov}(\zeta_{t+1}+\ldots+\zeta_{t+s_1},\zeta_{t+s_1+1}+\ldots+\zeta_{t+s_1+s_2})-C|<\delta_2,\quad
% \text{for all }s_1,s_2\geq\lfloor\delta_1 T\rfloor\text{ and all }t.
%\label{eq:cond_delta_T_5}
%\end{equation}

Set 
$$
f^*(t)=\mathrm{Var}(\zeta_{\lceil\delta_1 t\rceil}+\ldots+\zeta_t),\quad
f_*(t)=\mathrm{Var}(\zeta_{\lfloor\delta_1 t\rfloor}+\ldots+\zeta_t).
$$

Therefore, $|f^*(t)-f_*(t )|\leq 1+2(|c|+\delta_2)$ for all $t\geq T$. So, from (\ref{eq:cond_delta_T_5}) and (\ref{eq:cond_delta_T_1}),
\begin{equation}
\min\{f_*(T),f^*(T)\}>3|C|+7\delta_2.
\label{eq:f_star_bounded_from_0}
\end{equation}

Assume first that $c\leq 0$. Then, for $t,s\geq T$, from Claim~\ref{cl:var_2} and bounds (\ref{eq:cond_delta_T_2}), (\ref{eq:cond_delta_T_4}), (\ref{eq:cond_delta_T_5}),
\begin{align*}
 f^*(t+s)\,\,
 &=\,\,\mathrm{Var}\biggl(\zeta_{\lceil\delta_1 (t+s)\rceil}+\ldots+\zeta_{\lceil\delta_1 (t+s)\rceil-\lceil\delta_1 t\rceil+t}+\zeta_{\lceil\delta_1 (t+s)\rceil-\lceil\delta_1 t\rceil+t+1}+\ldots+\zeta_{t+s}\biggr)\\
&\geq\,\,\mathrm{Var}(\zeta_{\lceil\delta_1 (t+s)\rceil}+\ldots+\zeta_{\lceil\delta_1 (t+s)\rceil-\lceil\delta_1 t\rceil+t})\\
&\quad\quad\quad\quad\quad\quad\quad\quad+\mathrm{Var}(\zeta_{\lceil\delta_1 (t+s)\rceil-\lceil\delta_1 t\rceil+t+1}+\ldots+\zeta_{t+s})+2C-2\delta_2\\
&\geq\,\,\mathrm{Var}(\zeta_{\lceil\delta_1 (t+s)\rceil}+\ldots+\zeta_{\lceil\delta_1 (t+s)\rceil-\lceil\delta_1 t\rceil+t})\\
&\quad\quad\quad\quad\quad\quad\quad\quad+\mathrm{Var}(\zeta_{t+\lceil \delta_1 s\rceil}+\ldots+\zeta_{t+s})+2C-4\delta_2\\
&\geq \,\, f^*(t)+f^*(s)+2C-4\delta_2-8\frac{t}{\varepsilon}(1-\varepsilon)^{\lceil\delta_1 t\rceil}-8\frac{s}{\varepsilon}(1-\varepsilon)^{\lceil\delta_1 s\rceil}\\
&>\,\,f^*(t)+f^*(s)+2C-6\delta_2.
\end{align*}
Letting $\tilde f^*(t)=f^*(t)+2C-6\delta_2$, we get that $\tilde f^*(t+s)\geq\tilde f^*(s)+\tilde f^*(t)$ for all $t,s\geq T$ and, due to (\ref{eq:f_star_bounded_from_0}), $\tilde f^*(T)>|C|+\delta_2$. This immediately implies that $\tilde f^*(t)\geq\frac{\tilde f^*(T)}{T}t+O(1)$, and the same is true for $f^*(t)$. \\

On the other hand, for $t,s\geq T$, from Claim~\ref{cl:var_2}  and bounds (\ref{eq:cond_delta_T_2}), (\ref{eq:cond_delta_T_4}), (\ref{eq:cond_delta_T_5}),
\begin{align*}
 f_*(t+s)\,\, 
 &\leq \,\, \mathrm{Var}(\zeta_{\lfloor\delta_1 (t+s)\rfloor}+\ldots+\zeta_{\lfloor\delta_1 (t+s)\rfloor-\lfloor\delta_1 t\rfloor+t})\\
 &\quad\quad\quad\quad\quad\quad\quad\quad+\mathrm{Var}(\zeta_{\lfloor\delta_1 (t+s)\rfloor-\lfloor\delta_1 t\rfloor+t+1}+\ldots+\zeta_{t+s})+2C+2\delta_2\\
&\leq\,\,\mathrm{Var}(\zeta_{\lfloor\delta_1 (t+s)\rfloor}+\ldots+\zeta_{\lfloor\delta_1 (t+s)\rfloor-\lfloor\delta_1 t\rfloor+t})\\
&\quad\quad\quad\quad\quad\quad\quad\quad+\mathrm{Var}(\zeta_{t+\lfloor \delta_1 s\rfloor}+\ldots+\zeta_{t+s})+2C+4-4c+6\delta_2\\
&\leq\,\,f_*(t)+f_*(s)+2C+4-4c+6\delta_2+8\frac{t}{\varepsilon}(1-\varepsilon)^{\lfloor\delta_1 t\rfloor}+8\frac{s}{\varepsilon}(1-\varepsilon)^{\lfloor\delta_1 s\rfloor}\\
&<\,\, f_*(t)+f_*(s)+2C+4-4c+8\delta_2.
\end{align*}
Letting $\tilde f_*(t)=f_*(t)+2C+4-4c+8\delta_2$, we get $\tilde f_*(t+s)\leq\tilde f_*(s)+\tilde f_*(t)$ for all $t,s\geq T$. This immediately implies that $\tilde f_*(t)\leq\frac{\tilde f_*(T)}{T}t+O(1)$, and the same is true for $f_*(t)$.\\

Let $t\geq T$. Choose $t_1<\ldots<t_{\ell}<t_{\ell+1}=t+1$ such that $t_{\ell}=\lceil \delta_1 t\rceil$, $t_{\ell-1}=\lceil \delta_1 (t_{\ell}-1)\rceil$, $\ldots$, $t_1=\lceil\delta_1 (t_2-1)\rceil$, $t_1\geq T$ but $\delta_1 t_1< T$. Then since $\ell=O(\log t)$, from (\ref{eq:cov_zeta_above}), (\ref{eq:cond_delta_T_3}) and (\ref{eq:cond_delta_T_5}), we get
\begin{align}
\mathrm{Var}(\xi_1+\ldots+\xi_t) \,\, & \geq \,\,  \mathrm{Var}(\xi_1+\ldots+\xi_{t_1-1})+\sum_{j=2}^{\ell+1}f^*(t_j-1)+2(C-\delta_2)\ell\notag \\
 &\quad\quad\quad\quad+2\sum_{2\leq i+1<j\leq \ell+1}\mathrm{cov}(\xi_{t_i}+\ldots+\xi_{t_{i+1}-1},\xi_{t_j}+\ldots+\xi_{t_{j+1}-1})\notag\\
&\geq\,\, f^*(t)+(C-\delta_2)\ell-\frac{2}{(1-\varepsilon)^d\varepsilon^2}\sum_{1\leq i<j\leq \ell}(1-\varepsilon)^{t_j-t_i}\label{eq:Var_liminf} \\
&\geq\,\, f^*(t)+O(\ln t)-\frac{2}{(1-\varepsilon)^d\varepsilon^2}\sum_{1\leq i<j\leq \ell}(1-\varepsilon)^{\frac{T}{2\delta_1^{j-1}}}\notag\\
&\geq\,\,\frac{\tilde f^*(T)}{T}t+O(\ln t)\,\,\geq\,\,\frac{\tilde f_*(T)+6c-16\delta_2-5}{T}t+O(\ln t).\notag
\end{align}
In particular, we get that 
\begin{equation}
\liminf_{t\to\infty}\frac{\mathrm{Var}(\zeta_1+\ldots+\zeta_t)}{t}\geq\frac{\tilde f^*(T)}{T}>0.
\label{eq:Var_liminf_positive}
\end{equation}
Also,  due to (\ref{eq:var_above_sublinear}), there exists $A>0$ such that $\frac{\tilde f_*(T)}{T}<A$ independently of the choice of $T$.

In the same way, define $\tilde t_1<\ldots<\tilde t_{\ell}<\tilde t_{\ell+1}=t+1$ by $\tilde t_{j-1}=\lfloor \delta_1 (\tilde t_j-1)\rfloor$, $\tilde t_1\geq T$, $\delta_1 \tilde t_1< T$. We get
\begin{align}
\mathrm{Var}(\xi_1+\ldots+\xi_t) \,\,
&\leq\,\, 
\mathrm{Var}(\xi_1+\ldots+\xi_{\tilde t_1-1})+\sum_{j=2}^{\ell+1}f_*(\tilde t_j-1)+O(\ln t)\notag \\
&\leq\,\,\sum_{j=2}^{\ell+1}\frac{\tilde f_*(T)}{T}(\tilde t_j-1)+O(\ln t)\label{eq:Var_limsup}\\
&\leq\,\, \frac{\tilde f_*(T)}{(1-\delta_1)T}t+O(\ln t).\notag
\end{align}
From (\ref{eq:Var_liminf}) and (\ref{eq:Var_limsup}), we conclude that, for $t$ large enough, $\frac{\mathrm{Var}(\xi_1+\ldots+\xi_t)}{t}$ belongs to an interval of length $\frac{A\delta_1}{1-\delta_1}-\frac{6c-16\delta_2-5}{T}$. Due to the arbitrariness of $\delta_1$ and $T$ and due to (\ref{eq:Var_liminf_positive}), we get that  $\lim_{t\to\infty}\frac{1}{t}\mathrm{Var}(\xi_1+\ldots+\xi_t)$ is finite and positive as needed.\\

The case $c\geq 0$ is symmetric and identical. The only difference is that, for $f^*(t+s)$, we need an upper bound, and, for $f_*(t+s)$, we need a lower bound. %$\Box$

\subsection{Asymptotic normality}
\label{sc:normal}

%It is a routine to verify that finding asymptotics of all moments of $S_t$ in the same way as we estimate $\mathrm{Var}S_t$ in Lemma~\ref{lm:lll} leads to asymptotical normality of $S_t$. \\

\begin{theorem}
Let $d\in\mathbb{N}$.
\begin{enumerate}
\item
 If  $\frac{\ln t}{t^2}=o(\mathrm{Var}S^{(d)}_t)$, then there exists a finite $\lim\limits_{t\to\infty}t\mathrm{Var}S_t^{(d)}$
and
$$
 \frac{S^{(d)}_t-{\sf E}S^{(d)}_t}{\sqrt{\mathrm{Var}S^{(d)}_t}}\stackrel{d}\to\eta\sim\mathcal{N}\left(0,1\right), \quad t\to\infty.
$$
\item There exists a finite $\lim\limits_{t\to\infty}t\mathrm{Var}\tilde S^{(d)}_t$ and
$$
 \frac{\tilde S^{(d)}_t-{\sf E}\tilde S^{(d)}_t}{\sqrt{\mathrm{Var}\tilde S^{(d)}_t}}\stackrel{d}\to\eta\sim\mathcal{N}\left(0,1\right), \quad t\to\infty.
$$
\end{enumerate}
\label{th:clt}
\end{theorem}

\begin{proof} Proofs for $S=S^{(d)}$ and $\tilde S=\tilde S^{(d)}$ are identical, so we prove only the first part. The existence of $\lim\limits_{t\to\infty}t\mathrm{Var}S_t$ is proven in Lemma~\ref{lm:var_linear} (and the existence of  $\lim\limits_{t\to\infty}t\mathrm{Var}\tilde S_t$ can be proven similarly --- see Remark~\ref{rk:tilde_S} before Lemma~\ref{lm:var_linear}).  \\

It is very well-known that convergences of moments imply the convergence in distribution (in particular, in the case of the convergence to a normal random variable), see, e.g.,~\cite[Theorem 30.2]{Billingsley}. Therefore, it is sufficient to prove that
$$
{\sf E}\left(\frac{S_t-{\sf E}S_t}{\sqrt{\mathrm{Var}S_t}}\right)^k=(k-1)!!I(k\text{ is even})+o(1).
$$

Clearly, ${\sf E}\left(S_t-{\sf E}S_t\right)=0$ and ${\sf E}\left(\frac{S_t-{\sf E}S_t}{\sqrt{\mathrm{Var}S_t}}\right)^2=1$. \\

Let $k\geq 3$. Let $\beta_s=\zeta_s-{\sf E}\zeta_s$. We have $(t-d)(S_t-{\sf E}S_t)=\beta_1+\ldots+\beta_{t-d}$.

%%%%%%%%%%%%%%%%

Let $\ell\in[k-1]$, $\ell_1+\ell_2=\ell$, $t^1_1<\ldots<t^1_{\ell_1}<t^2_1<\ldots<t^2_{\ell_2}$ be positive integers, $t_1^2-t^1_{\ell_1}=s+d$, $s\geq 1$. Applying the same argument as in (\ref{eq:cov_from_above}), we get that
\begin{multline*}
 {\sf E}\zeta_{t^1_1}\ldots \zeta_{t^1_{\ell_1}}\zeta_{t^2_1}\ldots\zeta_{t^2_{\ell_2}}
 \,\,=\,\,{\sf P}\left(X_{t_1^1}=X_{t_1^1+d},\ldots,X_{t_{\ell_1}^1}=X_{t_{\ell_1}^1+d}\right)\quad\quad\quad\quad\quad\quad\quad\quad\\
\quad\quad\quad\quad\quad\quad\quad\quad\quad\quad\quad \times{\sf P}\left(X_{t_1^2}=X_{t_1^2+d}\ldots,X_{t_{\ell_2}^2}=X_{t_{\ell_2}^2+d},|\pi_{t_1^2}|<s,\ldots,|\pi_{t_{\ell_2}^2}|<s\right)\\
 \,\,\,+ {\sf P}\left(\left\{X_{t_1^1}=X_{t_1^1+d},\ldots,X_{t_{\ell_1}^1}=X_{t_{\ell_1}^1+d},X_{t_1^2}=X_{t_1^2+d}\ldots,X_{t_{\ell_2}^2}=X_{t_{\ell_2}^2+d}\right\}\cap\{\exists j\,|\pi_{t_j^2}|\geq s\}\right).
\end{multline*}
Therefore,
\begin{align*}
 {\sf E}\zeta_{t^1_1}\ldots \zeta_{t^1_{\ell_1}}\left({\sf E}\zeta_{t^2_1}\ldots\zeta_{t^2_{\ell_2}}-{\sf P}\left(\exists j\,|\pi_{t_j^2}|\geq s\right)\right) \,\,
 &\leq\,\,
 {\sf E}\zeta_{t^1_1}\ldots \zeta_{t^1_{\ell_1}}\zeta_{t^2_1}\ldots\zeta_{t^2_{\ell_2}} \\
 &\leq\,\,
 {\sf E}\zeta_{t^1_1}\ldots \zeta_{t^1_{\ell_1}}{\sf E}\zeta_{t^2_1}\ldots\zeta_{t^2_{\ell_2}}+{\sf P}\left(\exists j\,|\pi_{t_j^2}|\geq s\right)
\end{align*}
implying that 
$$
 \left|{\sf E}\zeta_{t^1_1}\ldots \zeta_{t^1_{\ell_1}}\zeta_{t^2_1}\ldots\zeta_{t^2_{\ell_2}}-{\sf E}\zeta_{t^1_1}\ldots \zeta_{t^1_{\ell_1}}{\sf E}\zeta_{t^2_1}\ldots\zeta_{t^2_{\ell_2}}\right|\leq k(1-\varepsilon)^s.
$$
By the linearity of expectation, we get
\begin{equation}
 |{\sf E}\beta_{t^1_1}\ldots\beta_{t^1_{\ell_1}}\beta_{t^2_1}\ldots\beta_{t^2_{\ell_2}}-{\sf E}\beta_{t^1_1}\ldots\beta_{t^1_{\ell_1}}{\sf E}\beta_{t^2_1}\ldots\beta_{t^2_{\ell_2}}|\leq 2^{\ell}k(1-\varepsilon)^s.
\label{eq:central_moments_indep}
\end{equation}
From (\ref{eq:central_moments_indep}) and the triangle inequality, if $\ell_1+\ldots+\ell_j=\ell$, $t^1_1<\ldots<t^1_{\ell_1}<\ldots<t^j_1<\ldots<t^j_{\ell_j}$, $t_1^{i+1}-t^i_{\ell_i}=s_i+d$, then
\begin{equation}
\left |{\sf E}\prod_{i=1}^j\beta_{t^i_1}\ldots\beta_{t^i_{\ell_i}}-\prod_{i=1}^j{\sf E}\beta_{t^i_1}\ldots\beta_{t^i_{\ell_i}}\right|
\leq 
2^{\ell}k\sum_{i=1}^j(1-\varepsilon)^{s_i}\leq 2^{k}k^2(1-\varepsilon)^{\min s_i}.
\label{eq:central_moments_indep_multy}
\end{equation}

Choose $s_0=\Theta(\log t)$ in a way so that $2^k k^2 (1-\varepsilon)^{s_0}<t^{-k}$. 

If at least one element of a tuple $(t_1,\ldots,t_k)\in[t]^k$ (say, $t_1$) is at distance at least $d+s_0$ of the closest among the other elements of this $k$-tuple, then (\ref{eq:central_moments_indep_multy}) and the equality ${\sf E}\beta_{t_1}=0$ imply that $|{\sf E}\beta_{t_1}\ldots \beta_{t_k}|<t^{-k}$. %This immediately implies that $\mathrm{Var}S_d=\Theta(t^{-1})$. Indeed,
%$$
% \mathrm{Var}S_{t+d}=\frac{1}{t^2}{\sf E}(\beta_1+\ldots+\beta_t)^2=\frac{1}{t^2}\sum_{j_1,j_2\in[t]}\mathrm{cov}(\zeta_{j_1},\zeta_{j_2})=\frac{1}{t^2}\left(\sum_{|j_1-j_2|\leq d}\mathrm{cov}(\zeta_{j_1},\zeta_{j_2})+ \sum_{|j_1-j_2|\leq d}\delta(j_1,j_2)\right),
%$$
%where $|\delta_{j_1,j_2}|\leq\frac{(1-\varepsilon)^{|j_1-j_2|}}{\varepsilon}$.

For $\ell_1,\ldots,\ell_j\geq 2$ such that $\ell_1+\ldots+\ell_j=k$, let $\mathcal{I}(\ell_1,\ldots,\ell_j)$ be the set of tuples $(t_1,\ldots,t_k)\in[t]^k$ such that $\ell_1,\ldots,\ell_j$ are exactly the lengths of inclusion-maximum tuples $t_{j_1}\leq\ldots \leq t_{j_{\ell}}$ in $(t_1,\ldots,t_k)$ such that all the distances between neighbors in these tuples are at most $d+s_0-1$. Note that 
$$
|\mathcal{I}(\ell_1,\ldots,\ell_j)|\leq k!\prod_{h=1}^j\left[t(d+s_0)^{\ell_j-1}\right]=k!t^j(d+s_0)^{k-j}.
$$

If $k$ is odd, then, since all $\ell_i$ are at least 2, and $\ell_1+\ldots+\ell_j=k$, we get that at least one $\ell_i$ is at least $3$, and so $j\leq\frac{k-1}{2}$. Therefore, since $s_0=\Theta(\log t)$, we get that
$$
{\sf E}(\beta_1+\ldots+\beta_t)^k
\leq\sum_{\ell_1,\ldots,\ell_j}|\mathcal{I}(\ell_1,\ldots,\ell_j)|+O(1)
=o(t^{(k-1)/2}(\log t)^k).
$$

Now assume that $k\geq 4$ is even. Since $|\mathrm{cov}(\beta_{t_1},\beta_{t_2})|\leq 1$ for all $t_1,t_2\in[t]$, we get from (\ref{eq:central_moments_indep_multy}) that
\begin{align*}
 {\sf E}(\beta_1+\ldots+\beta_t)^k\,\,
 &=\,\,\sum_{(t_1,\ldots,t_k)\in\mathcal{I}(2,\ldots,2)}{\sf E}\beta_{t_1}\ldots\beta_{t_k}+o(t^{k/2})\\
 &=\,\,\frac{k!}{2^{k/2}(k/2)!}\sum_{(t_1,t_2),\ldots,(t_{k-1},t_k)\in[t]^2}
 \prod_{i=1}^{k/2}\mathrm{cov}(\beta_{t_{2i-1}},\beta_{t_{2i}})+o(t^{k/2})\\
 &=\,\,(k-1)!!\left[\mathrm{Var}(\beta_1+\ldots+\beta_t)\right]^{k/2}+o(t^{k/2}).
\end{align*}

Finally, since $\mathrm{Var}(\beta_1+\ldots+\beta_t)=\Theta(t)$, we get that, for odd $k$,
$$
{\sf E}\left(\frac{S_t-{\sf E}S_t}{\sqrt{\mathrm{Var}S_t}}\right)^k=
\frac{{\sf E}\left(\beta_1+\ldots+\beta_{t-d}\right)^k}{(\mathrm{Var}\left(\beta_1+\ldots+\beta_{t-d}\right))^{k/2}}=o(t^{-1/2}(\log t)^k),
$$
and, for even $k$,
$$
{\sf E}\left(\frac{S_t-{\sf E}S_t}{\sqrt{\mathrm{Var}S_t}}\right)^k=
\frac{{\sf E}\left(\beta_1+\ldots+\beta_{t-d}\right)^k}{(\mathrm{Var}\left(\beta_1+\ldots+\beta_{t-d}\right))^{k/2}}=(k-1)!!+o(1).
$$
\end{proof}

\subsection{Other statistics}
\label{sc:other}

In Section~\ref{sc:bip}, we show that there are arbitrary large integers $n_1>n_2>m_2>m_1$ such that, for $(G_1=K_{m_1,n_1},u_1)$ and $(G_2=K_{m_2,n_2},u_2)$, where $u_1$ belongs to the part of size $n_1$ of $G_1$ and $u_2$ belongs to the part of size $n_2$ of $G_2$, $p_d(G_1,u_1)=p_d(G_2,u_2)$ for all $d\in\mathbb{N}$, and so Theorem~\ref{th:NVM_main} is not applicable. To show that these rooted graphs are in fact NVM-distinguishable, we introduce here another statistic, that counts 4-tuples with equal opinions: let, for a rooted graph $(G,u)$ and positive integers $d_1<d_2<d_3$,
$$
S_t^{(d_1,d_2,d_3)}:=\frac{1}{t-d_3}\sum_{i=1}^{t-d_3}I(X_i=X_{i+d_1}=X_{i+d_2}=X_{i+d_3}).
$$
Let $p_{t;d_1,d_2}$ and $p_{t;d_1,d_2,d_3}$ be the probability that $\pi_t$ meets both $\pi_{t+d_1}$ and $\pi_{t+d_2}$, and the probability that $\pi_t$ meets all $\pi_{t+d_1}$, $\pi_{t+d_2}$, $\pi_{t+d_3}$ respectively. Also, let $\tilde p_{t;d_1,d_2,d_3}$ be the probability that $\pi_t$ meets $\pi_{t+d_1}$, $\pi_{t+d_2}$ meets $\pi_{t+d_3}$, but $\pi_{t+d_1}$ does not meet $\pi_{t+d_2}$. Set 
$$
p_{d_1,d_2}=\lim_{t\to\infty}p_{t;d_1,d_2},\quad 
p_{d_1,d_2,d_3}=\lim_{t\to\infty}p_{t;d_1,d_2,d_3},\quad 
\tilde p_{d_1,d_2,d_3}=\lim_{t\to\infty}\tilde p_{t;d_1,d_2,d_3}.
$$

In the same way as in (\ref{eq:prob_duplicate}), we get
\begin{multline*}
 {\sf P}(X_t=X_{t+d_1}= X_{t+d_2}=X_{t+d_3})\,\,=\\ 
 =\,\,p_{t;d_1,d_2,d_3}+\frac{1}{2}(p_{t;d_1,d_2}+p_{t;d_1,d_3}+p_{t;d_2,d_3}+p_{t+d_1;d_2-d_1,d_3-d_1}-4p_{t;d_1,d_2,d_3})\quad\quad\\
 +\frac{1}{2}(\tilde p_{t;d_1,d_2,d_3}+\tilde p_{t;d_2,d_1,d_3}+\tilde p_{t;d_3,d_1,d_2})
 \quad\quad\quad\quad\quad\quad\quad\,\,\,\\
 \quad\quad\quad\quad\quad\quad\,\,\,\,\,\,
 +\frac{1}{4}\biggl[\sum_{i=1}^3\biggl(p_{t;d_i}-\sum_{j\neq i}\left(p_{t;d_i,d_j}-p_{t;d_1,d_2,d_3}\right)-p_{t;d_1,d_2,d_3}-\tilde p_{t;d_i,d_{j_1\neq i},d_{j_2\neq i,j_1}}\biggr)\\
\quad\quad\quad\quad\quad\quad\quad\quad +(p_{t+d_1;d_2-d_1}+\ldots)+\ldots\biggr]\\
 \quad\quad\quad\quad\quad\,\,\,+\frac{1}{8}\biggl(1-p_{t;d_1}-p_{t;d_2}-p_{t;d_3}-p_{t+d_1;d_2-d_1}-p_{t+d_1;d_3-d_1}-p_{t+d_2;d_3-d_2}\\
 \quad\quad\quad\quad\quad\quad\quad\quad\quad\quad\quad\quad\quad
 +2(p_{t;d_1,d_2}+p_{t;d_1,d_3}+p_{t;d_2,d_3}+p_{t+d_1;d_2-d_1,d_3-d_1})\\
 \quad\quad\quad\quad\quad\quad\quad\quad\quad\quad\quad\quad\quad
 + \tilde p_{t;d_1,d_2,d_3}+\tilde p_{t;d_2,d_1,d_3}+\tilde p_{t;d_3,d_1,d_2}-3p_{t;d_1,d_2,d_3}\biggr)\\
 =\,\,\frac{1}{8}\biggl(1+p_{t;d_1}+p_{t;d_2}+p_{t;d_3}+p_{t+d_1;d_2-d_1}+p_{t+d_1;d_3-d_1}+p_{t+d_2;d_3-d_2}\quad\quad\quad\quad\,\,\,\\
 +\tilde p_{t;d_1,d_2,d_3}+\tilde p_{t;d_2,d_1,d_3}+\tilde p_{t;d_3,d_1,d_2}+p_{t;d_1,d_2,d_3}\biggr).
\end{multline*}
Set 
\begin{align*}
 q_{d_1,d_2,d_3}\,\,
 &=\,\,\lim_{t\to\infty}{\sf P}(X_t=X_{t+d_1}= X_{t+d_2}=X_{t+d_3})\\
 &=\,\, \frac{1}{8}\biggl(1+p_{d_1}+p_{d_2}+p_{d_3}+p_{d_2-d_1}+p_{d_3-d_1}+p_{d_3-d_2}\\ 
&\quad\quad\quad\quad
+\tilde p_{d_1,d_2,d_3}+\tilde p_{d_2,d_1,d_3}+\tilde p_{d_3,d_1,d_2}+p_{d_1,d_2,d_3}\biggr).
\end{align*}

The proof of the following theorem copies the proof of Theorem~\ref{th:NVM_main}. So, we omit it.\\

\begin{theorem}
If $(G_1,u_1)$, $(G_2,u_2)$ are rooted graphs such that, for some positive integers $d_1<d_2<d_3$,%, for some integer $m\geq 2$ and positive integers $d_1<\ldots<d_{m-1}$, 
$$
q_{d_1,d_2,d_3}(G_1,u_1)\neq q_{d_1,d_2,d_3}(G_2,u_2),
%\label{eq:dif_lim_prob}
$$
then $(G_1,u_1)$ and $(G_2,u_2)$ are distinguishable by  $S^{(d_1,d_2,d_3)}$.\\
\label{th:NVM_aux}
\end{theorem}

\begin{remark} Note that 3-tuples of time stamps are not sufficient since
\begin{align*}
{\sf P}(X_t=X_{t+d_1}=X_{t+d_2}) \,\, 
 & = \,\, p_{t;d_1,d_2} + 
\frac{1}{2}(p_{t;d_1}+p_{t;d_2}+p_{t+d_1;d_2-d_1}-3p_{t;d_1,d_2})\\
 & \quad\quad\quad\quad\,\,
+\frac{1}{4}(1-p_{t;d_1}-p_{t;d_2}-p_{t+d_1;d_2-d_1}+2p_{t;d_1,d_2})\\
 & = \,\, \frac{1}{4}(1+p_{t;d_1}+p_{t;d_2}+p_{t+d_1;d_2-d_1})
\end{align*}
does not depend on $p_{t;d_1,d_2}$.
\end{remark}

\section{Random graphs}
\label{sc:random}

As we mention in Introduction, we conjecture that the NVM distinguishes between any two connected graphs other than stars. In this section we prove Theorem~\ref{th:r_g}, i.e. we confirm the conjecture for asymptotically almost all pairs of graphs and almost all $\varepsilon$.

\subsection{Proof of Theorem~\ref{th:r_g}}

Let $\beta=1-\varepsilon.$ First, let us note that, for an arbitrary rooted graph $(G,u)$, the probability $p_2=p_2(G,u)$ equals $\sum_{j=2}^{\infty}a_j(G,u)\beta^j$, where $a_j(G,u)\in[0,1]$ is the limit probability (as $t\to\infty$) that $\pi_t(G,u)$ meets $\pi_{t+2}(G,u)$ for the first time at a vertex which is at distance exactly $\frac{j-2}{2}$ from $(u,t)$ in $\mathcal{D}_G$. The latter function $p_2=p_2(\beta)$ is analytic on $(-1,1)$. Indeed, all $a_j(G,u)$ are non-negative and at most 1. Therefore, for any compact set $K\subset(-1,1)$, any $\beta\in K$, and any $k\in\mathbb{N}$, we have that 
$$
\left|\frac{\partial^k}{\partial\beta^k}p_2\right|\leq\sum_{j=k}^{\infty}{j\choose k}k!|\beta|^{j-k}\leq\left.\frac{\partial^k}{\partial x^k}\left(\frac{1}{1-x}\right)\right|_{x=|\beta|}=\frac{(k-1)!}{(1-|\beta|)^k}\leq(k-1)!C^k,
$$
where $C=\frac{1}{1-\sup |K|}$.  Therefore, for two rooted graphs $(G_1,u_1)$ and $(G_2,u_2)$, either $p_2(G_1,u_1)=p_2(G_2,u_2)$ for all $\beta\in(0,1)$ (equivalently, $a_j(G_1,u_1)=a_j(G_2,u_2)$ for all $j$), or the zero set of $p_2(G_1,u_1)-p_2(G_2,u_2)$ has a zero measure (see, e.g.~\cite{Mityagin}). Therefore, due to Theorem~\ref{th:NVM_main}, it is sufficient to prove that whp, for any $u_1,u_2\in[n]$, $a_2(G_n^1,u_1)\neq a_2(G_n^2,u_2)$. 

Recall that 
$$
a_2(G,u)={\sf P}(\pi_{t+2}\text{ meets }\pi_t\text{ for the first time at }(v,t)\text{ for some }v),
$$ 
and this $v$ may only coincide with $u$ since $(u,t)\in\pi_t$. Thus, $a_2(G,u)$ is exactly the limit (as $t\to \infty)$ of the probability that $\pi_{t+2}$ at time $t$ returns to $u$. Therefore, 
$$
 a_2(G,u)=\frac{1}{\mathrm{deg}_G u}\sum_{v\in N_G(u)}\frac{1}{\mathrm{deg}_G v}.
$$
It remains to prove that whp there are no $u_1,u_2\in[n]$ such that 
\begin{equation}
\frac{1}{\mathrm{deg}_{G_n^1} u_1}\sum_{v\in N_{G_n^1}(u_1)}\frac{1}{\mathrm{deg}_{G_n^1} v}=\frac{1}{\mathrm{deg}_{G_n^2} u_2}\sum_{v\in N_{G_n^2}(u_2)}\frac{1}{\mathrm{deg}_{G_n^2} v}.
\label{eq:inverse_degrees_equal}
\end{equation}

Our proof strategy is as follows. Fix vertices $u_1,u_2$ of $G_n^1,G_n^2$ respectively. % If $d_1:=\mathrm{deg}_{G_n^1}u_1\neq\mathrm{deg}_{G_n^2}u_2=:d_2$, then
 We find a %large (bigger than square root of any degree in $G_n^1$ and $G_n^2$) 
prime number $p$ %=\Theta\left(\sqrt{n\ln n}\right)$
 such that 
\begin{itemize}
\item $d_1:=\mathrm{deg}_{G_n^1}u_1$, $d_2:=\mathrm{deg}_{G_n^2}u_2$ are not divisible by $p$;
\item there is a unique element $d^*$ of the set $\{\mathrm{deg}_{G_n^1} v,\,v\in N_{G_n^1}(u_1)\}$ divisible by $p$;
\item any element of $\{\mathrm{deg}_{G_n^2} v,\,v\in N_{G_n^2}(u_2)\}$ %the degree sequence of $G_n^2$ 
 other than $d^*$ is not divisible by $p$;
\item $\mu_1d_2-\mu_2d_1$ is not divisible by $p$, where $\mu_j$ is the number of vertices $v\in N_{G_n^j}(u_j)$ such that $\mathrm{deg}_{G_n^j} v=d^*$.
\end{itemize}
% sets $\mathcal{P}_1$ and $\mathcal{P}_2$ of numbers among $\mathrm{deg}_{G_n^1} v$, $v\in N_{G_n^1}(u_1)$, and among $\mathrm{deg}_{G_n^2} v$, $v\in N_{G_n^2}(u_2)$, that are divisible by $p$, are equal. We also require $d_2-d_1$ to be not divisible by $p$. 
If this is really possible, then, for some integer numbers $x_j,y_j$, $j\in\{1,2\}$, non-divisible by $p$, and for $j\in\{1,2\}$,
$$
\frac{1}{\mathrm{deg}_{G_n^j} u_j}\sum_{v\in N_{G_n^j}(u_j)}\frac{1}{\mathrm{deg}_{G_n^j} v}=\frac{1}{pd_j}\frac{\mu_j}{d^*/p}+\frac{x_j}{y_j}.
$$
Then (\ref{eq:inverse_degrees_equal}) implies that $\mu_1d_2-\mu_2d_1$ is divisible by $p$ --- a contradiction. Lemma~\ref{cl:degrees_existence} stated below implies that whp, for every $u_1,u_2\in[n]$, the desired $p$ exists. \\

Let $G_n$ be a uniformly distributed random graph on $[n]$.
Let $\delta>0$ be small enough. By the prime number theorem, for $n$ large enough, the set $\mathcal{P}=\mathcal{P}(n)$ of all prime numbers in $\left[(\sqrt{2}-2\delta)\sqrt{n\ln n},(\sqrt{2}-\delta)\sqrt{n\ln n}\right]$ that do not divide $n$ has cardinality $\Theta\left(\sqrt{\frac{n}{\ln n}}\right)$ (there is at most 1 prime number in this interval that divides $n$). For each $p\in\mathcal{P}$, find the unique $d^*=d^*(p)\in\left[\frac{n}{2}-\frac{p}{2},\frac{n}{2}+\frac{p}{2}\right]$ divisible by $p$.\\

\begin{lemma}
Let $C$ be large enough, $k=\lfloor C\ln n\rfloor$. Fix distinct $p_1,\ldots,p_k\in\mathcal{P}$ and positive integers $\mu_1,\ldots,\mu_k$. Then whp
%Fix $i\in[n]$, and let $U$ be the set of neighbors of $i$ in $G_n$. Let $C$ be large enough, $k=\lfloor C\ln n\rfloor$.
\begin{enumerate}
\item for every $u\in[n]$, %With probability $1-o(1/n)$, 
there exists at least $\ln^2 n$ numbers $p\in\mathcal{P}$ such that $d^*(p)$ is a unique element of $\{\mathrm{deg}_{G_n}v,\,v\in N_{G_n}(u)\}$ divisible by $p$;
%\item For any different $p_1,p_2,p_3\in\mathcal{P}$, with probability $1-o(1/n)$, there exists $j\in\{1,2,3\}$ such that if $d$ belongs to $\{\mathrm{deg}_{G_n}v,\,v\in U\}$ and is divisible by $p_j$ then $d=d^*(p_j)$.
\item %For any different  $p_1,\ldots,p_k\in\mathcal{P}$ and $\mu_1,\ldots,\mu_k$, with probability $1-o(1/n)$, 
for every $u\in[n]$, there exists $i\in[k]$ such that 
\begin{itemize}
\item the number of vertices in $N_{G_n}(u)$ with degree $d^*(p_i)$ does not equal to $\mu_i$,
\item there is no vertex in $N_{G_n}(u)$ with a degree $d\neq d^*(p_i)$ divisible by $p_i$;
\end{itemize}
\item %With probability $1-o(1/n)$, 
for every integer $d$ the number of vertices $u\in[n]$ of degree $d$ is less than $2\sqrt{n}$.
\end{enumerate}
%\begin{enumerate}
%\item If there exists $p=p(n)\in\mathcal{P}$ such that, for all $n$ large enough, $|pq-n/2|\leq\frac{1-\varepsilon}{2}\sqrt{n\ln n}$ for some integer $q$, then whp there exists a unique $d$ (actually $d=qp$) divisible by $p$ in the degree sequence of $G_n$.
%\item If, for all $n$ large enough, $\mathcal{P}$ does not meet $\left[\frac{n}{2}-\frac{1-\varepsilon}{2}\sqrt{n\ln n},\frac{n}{2}+\frac{1-\varepsilon}{2}\sqrt{n\ln n}\right]$, then whp there exists $p\in\mathcal{P}$ and a unique $d$ in the degree sequence of $G_n$ such that $d\in ...$ and $d$ is divisible by $p$.
%\end{enumerate}
\label{cl:degrees_existence}
\end{lemma}

Let us show that Lemma~\ref{cl:degrees_existence} indeed implies the existence of an appropriate prime $p$ and thus completes the proof of Theorem~\ref{th:r_g}. %Note that whp all vertices of $G_n$ have degrees equal to $n/2(1+o(1))$. 
 The first part of the lemma allows to find $\ln^2 n$ prime numbers $p\in\mathcal{P}$ such that $d^*(p)$ is a unique element of $\{\mathrm{deg}_{G_n^1}v,\,v\in N_{G_n^1}(u_1)\}$ divisible by $p$. Let $\mathcal{P}^*$ be those of these selected prime numbers that are not divisors of $d_1$ and $d_2$ (there are at most 2 since the considered prime numbers are larger than $\sqrt{n}$). Take distinct $p_1,\ldots,p_{k}\in\mathcal{P}^*$. For $i\in[k]$, let $\mu_1(i)\geq 1$ be the number of vertices $v\in N_{G_n^1}(u_1)$ with $\mathrm{deg}_{G_n^1}v=d^*(p_i)$. Let $\tilde\mu_2(i)$ be the minimum positive integer such that $\mu_1(i) d_2-\tilde\mu_2(i) d_1$ is divisible by $p_i$ (such a number exists since $d_1$ is not divisible by $p_i$). By the second part of the lemma we can guarantee the existence of $i\in[k]$ such that the number of vertices in $N_{G_n^2}(u_2)$ (denoted by $\mu_2(i)$) with degree $d^*(p_i)$ does not equal to $\tilde\mu_2(i)$, and for any other $d$ divisible by $p_i$ there is no vertex in $N_{G_n^2}(u_2)$ with degree $d$.

Let us show that $p=p_i$ is the desired prime number. The first three points are immediate. It remains to prove that $\mu_1(i)d_2-\mu_2(i)d_1$ is not divisible by $p$. From the third part of the lemma it follows that we may assume that $\mu_2(i)<2\sqrt{n}$. Since $p$ is a prime number, and $d_1$ is not divisible by $p$, we get that there is at most one number $\mu<2\sqrt{n}$ such that $\mu_1(i)d_2-\mu d_1$ is divisible by $p$, that immediately implies the desired assertion. This finishes the proof of Theorem~\ref{th:r_g}.\\

We prove Lemma~\ref{cl:degrees_existence} in Section~\ref{sc:lm_degrees_proof}. We will use an approximation of the degree sequence of the random graph by independent random variables given in the next section.

\subsection{Almost independence of the degree sequence}

We need the following result of McKay and Wormald~\cite{MW} stating that the degrees of $G_n$ are almost independent.\\

\begin{theorem}[McKay, Wormald~\cite{MW}]
Let $\xi$ be a normal random variable with mean $1/2$ and variance $1/(4n(n-1)),$ truncated to $(0,1)$. For each $x\in(0,1)$, let $\eta_1(x),\ldots,\eta_n(x)$ be independent binomial random variables with parameters $n-1$ and $x$, and let $\eta_1(x),\ldots,\eta_n(x)$ be independent of $\xi$. Let $\mathbf{d}_1,\ldots,\mathbf{d}_n$ be the degrees of vertices $1,\ldots,n$ in $G_n$ respectively. Fix a real $\alpha>0$ and a set $A_n\subset\{0,1,\ldots,n-1\}^n$ for each $n$. Then
$$
 {\sf P}[(\mathbf{d}_1,\ldots,\mathbf{d}_n)\in A_n]=(1+o(1)){\sf P}\left[(\eta_1(\xi),\ldots,\eta_n(\xi))\in A_n\left|\sum_{i=1}^n\eta_i(\xi)\text{ is even}\right.\right]+o(n^{-\alpha}).
$$
\label{thm:Wormald_McKay}
\end{theorem}

\subsection{Proof of Lemma~\ref{cl:degrees_existence}} 
\label{sc:lm_degrees_proof}

Fix $u\in[n]$. Let $U:=N_{G_n}(u)$, $N:=|U|=\mathrm{deg}_{G_n}(u)$. We are going to use Theorem~\ref{thm:Wormald_McKay} to explain why degrees of vertices in $U$ can be replaced with independent random variables. This transference is standard. However, for convenience we give here the whole argument. 

In order to transfer the probabilities of events associated with the whole degree sequence of $G_n$, we consider the random variables $\xi$ and $\eta_1(x),\ldots,\eta_{n}(x)$, $x\in(0,1)$, as in Theorem~\ref{thm:Wormald_McKay}. Let $I_n:=\left[\frac{1}{2}-2\frac{\ln n}{n},\frac{1}{2}+2\frac{\ln n}{n}\right]$,  $f$ be the density of $\xi$ and $\Phi(t)=\int_{-\infty}^t\frac{1}{\sqrt{2\pi}}e^{-s^2/2}ds$ be the cumulative distribution function of a standard normal random variable. Then
\begin{align}
 {\sf P}(\xi\notin I_n) \,\, & =\,\,\int\limits_0^{1/2-2\ln n/n} f(x)dx+\int\limits_{1/2+2\ln n/n}^1 f(x)dx\notag\\
 & =\,\,\frac{1}{1-2\Phi(-\sqrt{n(n-1)})}\left[\int\limits_{-\sqrt{n(n-1)}}^{-4\ln n\sqrt{1-1/n}}+\int\limits_{4\ln n\sqrt{1-1/n}}^{\sqrt{n(n-1)}}\right]\frac{1}{\sqrt{2\pi}}e^{-x^2/2}dx
 \label{eq:xi_concentration}\\
 &=\,\,o\left(\frac{1}{n^2}\right)\notag,
\end{align}
since $1-\Phi(t)\sim\frac{1}{\sqrt{2\pi}t}e^{-t^2/2}$ as $t\to\infty$; see, e.g.,~\cite{Feller}. Since $\sum_{i=1}^{n}\eta_i(x)$ has binomial distributions with parameters $n(n-1)$ and $x$, it is very well known (see, e.g.,~\cite{Herschkorn}) that, for $x\in I_n$,
$$
 {\sf P}\left(\sum_{i=1}^{n}\eta_i(x)\text{ is even}\right)=\frac{1}{2}+\frac{1}{2}(1-2x)^{n(n-1)}=\frac{1}{2}+o(1).
$$

In order to transfer the probabilities of events associated with the degrees of vertices in the neighborhood of a given vertex $u$, we let $G^{\prime}:=G_n\setminus\{u\}$ to be obtained from $G_n$ by removing $u$ with all edges adjacent to $u$. Notice that the distribution of edges in $G^{\prime}$ does not depend on $U$. Clearly, for every $v\in U$, $\mathrm{deg}_{G^{\prime}}(v)=\mathrm{deg}_{G_n}(v)-1$ has binomial distribution with parameters $n-2$ and $1/2$. By the Chernoff bound, $|N-n/2|>\sqrt{2n\ln n}$ with probability at most $n^{-2+o(1)}$. Fix an integer $\nu$ such that $|\nu-n/2|\leq \sqrt{2n\ln n}$. 

%For $\nu$ numbers $d_1,\ldots,d_{\nu}$, consider the event $\mathcal{E}_n(d_1,\ldots,d_{\nu})$ saying that there are no prime numbers among them satisfying (\ref{eq:rare_degrees_range}). 

All the below probabilities are conditioned on $\{N=\nu\}$. An application of Theorem~\ref{thm:Wormald_McKay} to $\mathrm{deg}_{G_n}(v),$ $v\in U$, is now legal. Without loss of generality, let $U=[N]$ and $u=n$. Since we apply it to $G^{\prime}$, $n$ in the notations of the statement of Theorem~\ref{thm:Wormald_McKay} becomes $n-1$. Let random variables $\xi$ and $\tilde\eta_1(x),\ldots,\tilde \eta_{n-1}(x)$, $x\in(0,1)$, be as in Theorem~\ref{thm:Wormald_McKay}. In the same way as above we get that, for $x\in I_n$,
$$
 {\sf P}\left(\sum_{i=1}^{n-1}\tilde\eta_i(x)\text{ is even}\right)=\frac{1}{2}+\frac{1}{2}(1-2x)^{(n-1)(n-2)}=\frac{1}{2}+o(1).
$$

%\newpage

Let $x\in I_n$, $\eta_i=\eta_i(x)$ for $i\in[n]$ and $\tilde\eta_i=\tilde\eta_i(x)$ for $i\in[n-1]$.

Fix $p\in\mathcal{P}$. Set $\mathcal{J}=\left[\frac{n}{2}-\frac{p}{2},\frac{n}{2}+\frac{p}{2}\right]$, $\mathcal{I}=\left[\frac{n}{2}-\frac{\sqrt{2}}{2}\sqrt{n\ln n},\frac{n}{2}+\frac{\sqrt{2}}{2}\sqrt{n\ln n}\right]$. Note that, by the de Moivre--Laplace limit theorem whp all $\eta_i(x)$ and $\tilde\eta_i(x)$ belong to $\mathcal{I}$. Due to the choice of $\mathcal{P}$, in $\mathcal{I}\setminus\mathcal{J}$ there is at most one number divisible by $p$. Assume that $d'\neq d^*$ is the number (if there are two, then take the smallest) divisible by $p$ and closest to $\mathcal{J}$. \\

1. For positive integers $y_1,\ldots,y_n$, consider the set $\mathcal{P}^*(y_1,\ldots,y_{\nu})$ of all $p\in\mathcal{P}$ such that there exists $i\in[\nu]$ with  $y_i=d^*(p)-1$, and the set $\mathcal{P}'(y_1,\ldots,y_n)$ of all $p$ such that there exist $i\in[n]$ and $j\in[n]$ with $y_i=d^*(p)$ and $y_j=d'(p)$. Consider the events $\mathcal{E}_n^*(y_1,\ldots,y_{\nu})$ and $\mathcal{E}_n'(y_1,\ldots,y_n)$ saying that $|\mathcal{P}^*(y_1,\ldots,y_{\nu})|< n^{\delta}$ and $|\mathcal{P}'(y_1,\ldots,y_n)|\geq 1$ respectively.

% Let $\mathcal{P}^*$ be the set of all $p\in\mathcal{P}$ such that there exists $i\in[N]$ with  $\eta_i=d^*(p)$. Let $\mathcal{P}'\subset\mathcal{P}^*$ be the set of all $p\in\mathcal{P}$ such that there exist $i\in[n]$ and $j\in[N]$ with $\eta_i=d^*(p)$ and $\eta_j=d'(p)$.

We shall prove that ${\sf P}(\mathcal{E}_n^*(\mathbf{d}_1-1,\ldots,\mathbf{d}_{\nu}-1))=o(1/n)$, and that ${\sf P}(\mathcal{E}_n'(\mathbf{d}_1,\ldots,\mathbf{d}_n))=o(1)$, and this immediately implies the first part of Lemma~\ref{cl:degrees_existence}. We first prove this for $\mathcal{E}_n^*(\tilde\eta_1,\ldots,\tilde\eta_{\nu})$ and $\mathcal{E}_n'(\eta_1,\ldots,\eta_n)$.

 For $i\in[\nu]$, %let $X^*_i$ be the indicator random variable of the event that $\eta_i=d^*(p)$, and let $X'_i$ be the indicator random variable of the event that $\eta_i=d'(p)$. Then, 
 by the de Moivre--Laplace limit theorem,
$$
{\sf P}(\tilde\eta_i=d^*(p)-1)=\Omega(n^{-3/2+\delta}).
$$
Therefore,
$$
{\sf P}(\exists p\in\mathcal{P}\,\,\tilde\eta_i=d^*(p)-1)=\Omega(|\mathcal{P}|n^{-3/2+\delta})=\Omega(n^{-1+\delta}\sqrt{\ln n}).
$$

We get that 
$$
{\sf E}|\mathcal{P}^*(\tilde\eta_1,\ldots,\tilde\eta_{\nu})|=\nu{\sf P}(\exists p\,\,\tilde\eta_1=d^*(p)-1)=\Omega(n^{\delta}\sqrt{\ln n}).
$$
By the Chernoff bound, ${\sf P}(|\mathcal{P}^*(\tilde\eta_1,\ldots,\tilde\eta_{\nu})|<n^{\delta})=o(1/n)$. In other words, ${\sf P}(\mathcal{E}_n^*(\tilde\eta_1,\ldots,\tilde\eta_{\nu}))=o(1/n)$ as needed.

It remains to show that whp $\mathcal{P}'(\eta_1,\ldots,\eta_n)$ is empty. If $p\in\mathcal{P}$ is such that 
$$
d^*(p)\in\mathcal{I}_{\delta}:=\left(\frac{n}{2}-\left(\frac{1}{\sqrt{2}}-2\delta\right)\sqrt{n\ln n},\frac{n}{2}+\left(\frac{1}{\sqrt{2}}-2\delta\right)\sqrt{n\ln n}\right),
$$
then $d'(p)\notin\mathcal{I}$ implying that whp there are no $i\in[n]$ with $\eta_i=d'(p)$. On the other hand, for distinct $i,j\in[n]$, integers $d^*\notin\mathcal{I}_{\delta}$ and $d'\notin\mathcal{J}$, we have that 
$$
{\sf P}(\eta_i=d^*,\eta_j=d')=O(n^{-3+5\delta}).
$$
Therefore,
$$
{\sf P}(\exists p\,\,\eta_i=d^*(p)\notin\mathcal{I}_{\delta},\eta_j=d'(p))=O(n^{-5/2+2\delta}\sqrt{\ln n}).
$$
By Markov's inequality, whp there are no $i,j\in[n]$ such that, for some $p\in\mathcal{P}$, $\eta_i=d^*(p)$ and $\eta_j=d'(p)$. In other words, ${\sf P}(\mathcal{E}_n'(\eta_1,\ldots,\eta_n))=o(1)$ as needed.\\

Now, let us transfer this result to the degree sequence. By Theorem~\ref{thm:Wormald_McKay},
\begin{align*}
 {\sf P}(\mathcal{E}_n^*(\mathbf{d}_1-1,\ldots,\mathbf{d}_{\nu}-1)) \,\, & 
 =\,\,(1+o(1)){\sf P}\left[\mathcal{E}_n^*(\tilde\eta_1(\xi),\ldots,\tilde\eta_{\nu}(\xi))\left|\sum_{i=1}^{n-1}\tilde\eta_i(\xi)\text{ is even}\right.\right]+o(n^{-2})\\
  &\leq\,\, (1+o(1))\frac{{\sf P}\left[\mathcal{E}_n^*(\tilde\eta_1(\xi),\ldots,\tilde\eta_{\nu}(\xi))\right]}{{\sf P}\left[\sum_{i=1}^{n-1}\tilde\eta_i(\xi)\text{ is even}\right]}+o(n^{-2})\\
  &\leq\,\, (2+o(1)){\sf P}\left[\mathcal{E}_n^*(\tilde\eta_1(\xi),\ldots,\tilde\eta_{\nu}(\xi))\right]+o(n^{-2}).
\end{align*}
Furthermore, by (\ref{eq:xi_concentration}),
\begin{align*}
 {\sf P}\left[\mathcal{E}_n^*(\tilde\eta_1(\xi),\ldots,\tilde\eta_{\nu}(\xi))\right] \,\, & \leq\,\,\int_{1/2-2\ln n/n}^{1/2+2\ln n/n}
 \left[\mathcal{E}_n^*(\tilde\eta_1(x),\ldots,\tilde\eta_{\nu}(x))\right]f(x)dx+o(n^{-2})\\
 &\leq\,\,
 \max_{|x-1/2|\leq 2\ln n/n}{\sf P}\left[\mathcal{E}_n^*(\tilde\eta_1(x),\ldots,\tilde\eta_{\nu}(x))\right]+o(n^{-2})\,\,=\,\,o(n^{-1}).
\end{align*}

In the same way, Theorem~\ref{thm:Wormald_McKay} implies that
$$
 {\sf P}\left[\mathcal{E}_n'(\mathbf{d}_1,\ldots,\mathbf{d}_n)\right]
 \leq (2+o(1))\max_{|x-1/2|\leq 2\ln n/n}{\sf P}\left[\mathcal{E}_n'(\eta_1(x),\ldots,\eta_n(x))\right]+o(n^{-2})=o(1).
$$ 
The first part of Lemma~\ref{cl:degrees_existence} is proven.\\

2. For positive integers $y_1,\ldots,y_{\nu}$ and $\ell\in[k]$, consider the event $\mathcal{A}_{\ell}(y_1,\ldots,y_{\nu})$ saying that the number of $i\in[\nu]$ such that $y_i=d^*(p_{\ell})-1$ equals $\mu_{\ell}$. Also, consider the event $\mathcal{B}_{\ell}(y_1,\ldots,y_{\nu})$ saying that there exists $i\in[\nu]$ such that $y_i=d'(p_{\ell})-1$. We shall prove that 
$$
{\sf P}(\forall\ell\in[\nu]\,\,\mathcal{A}_{\ell}(\mathbf{d}_1-1,\ldots,\mathbf{d}_{\nu}-1)\cup \mathcal{B}_{\ell}(\mathbf{d}_1-1,\ldots,\mathbf{d}_{\nu}-1))=o(1/n),
$$
and this immediately implies the second part of Lemma~\ref{cl:degrees_existence}. Set $\vec{\bold d}=(\mathbf{d}_1-1,\ldots,\mathbf{d}_{\nu}-1)$. We get
\begin{multline*}
{\sf P}\biggl(\bigcap_{\ell\in[k]}\,\,\mathcal{A}_{\ell}(\vec{\bold d})\cup \mathcal{B}_{\ell}(\vec{\bold d})\biggr)\,\,\leq\,\,
{\sf P}\biggl(\bigcup_{1\leq\ell_1<\ell_2\leq k}\mathcal{B}_{\ell_1}(\vec{\bold d})\cap\mathcal{B}_{\ell_2}(\vec{\bold d})\biggr)\\
+{\sf P}\biggl(\bigcap_{\ell\in[k]}\,\,\mathcal{A}_{\ell}(\vec{\bold d})\biggr)+\sum_{\ell\in[k]}{\sf P}\biggl(\bigcap_{\tilde\ell\in[k],\tilde\ell\neq\ell}\,\,\mathcal{A}_{\tilde \ell}(\vec{\bold d})\biggr).
\end{multline*}

By the union and by symmetry, it remains to prove that
$$
{\sf P}\biggl(\mathcal{B}_1(\vec{\bold d})\cap\mathcal{B}_2(\vec{\bold d})\biggr)=o\left(\frac{1}{n\ln^2n}\right),\quad
{\sf P}\biggl(\bigcap_{\ell\in[k-1]}\,\,\mathcal{A}_{\ell}(\vec{\bold d})\biggr)=o\left(\frac{1}{n\ln n}\right).
$$

By Theorem~\ref{thm:Wormald_McKay}, it is sufficient to prove this for $\mathbf{d}_1-1,\ldots,\mathbf{d}_{\nu}-1$ replaced by $\tilde\eta_1(\xi),\ldots,\tilde\eta_{\nu}(\xi)$ and $x\in I_n$. 

Since $d'(p_1)$ and $d'(p_2)$ do not belong to $\mathcal{J}$, we get by the de Moivre--Laplace limit theorem
$$
 {\sf P}(\tilde\eta_1(x)=d'(p_1)-1,\tilde\eta_2(x)=d'(p_2)-1)=O(n^{-3+4\delta}).
$$
By Markov's inequality,
\begin{multline*}
{\sf P}\biggl(\mathcal{B}_1(\tilde\eta_1(x),\ldots,\tilde\eta_{\nu}(x))\cap\mathcal{B}_2(\tilde\eta_1(x),\ldots,\tilde\eta_{\nu}(x))\biggr) =\\
={\sf P}\biggl(\exists i,j\in[\nu]\,\, \tilde\eta_1(x)=d'(p_1)-1,\tilde\eta_2(x)=d'(p_2)-1\biggr)
=o\left(\frac{1}{n\ln^2n}\right)
\end{multline*}
as desired.

Note that, since all $\mu_{\ell}$ are positive, for some non-negative $P_1,\ldots,P_k$ such that $P_1+\ldots+P_k=1$, we get
\begin{multline*}
{\sf P}\biggl(\bigcap_{\ell\in[k-1]}\,\,\mathcal{A}_{\ell}(\tilde\eta_1(x),\ldots,\tilde\eta_{\nu}(x))\biggr)\,\,
=\,\,{\nu\choose\mu_1,\ldots,\mu_{k-1}}P_1^{\mu_1}\ldots P_{k-1}^{\mu_{k-1}}P_{k}^{n-\mu_1-\ldots-\mu_{k-1}}\\
\leq\,\,\left(\frac{1+o(1)}{\sqrt{2\pi}}\right)^{k-1}\left[\left(\frac{P_1\nu}{\mu_1}\right)^{\frac{\mu_1}{\nu}}\ldots\left(\frac{P_{k-1}\nu}{\mu_{k-1}}\right)^{\frac{\mu_{k-1}}{\nu}}\left(\frac{P_{k}\nu}{\nu-\mu_1-\ldots-\mu_{k-1}}\right)^{1-\frac{\mu_1}{\nu}-\ldots-\frac{\mu_{k-1}}{\nu}}\right]^{\nu}\\
\leq\,\,\left(\frac{1+o(1))}{\sqrt{2\pi}}\right)^{k-1}\,\,=\,\,o\left(\frac{1}{n\ln n}\right)
\end{multline*}
as needed. Indeed, it is very well known that, for positive $t_1,\ldots,t_k$ such that $t_1+\ldots+t_k=1$,
$$
 t_1\ln\frac{P_1}{t_1}+\ldots+t_k\ln\frac{P_k}{t_k}\leq 0.
$$
In particular, the latter inequality follows from Jensen's inequality ${\sf E}f(Z)\geq f({\sf E}Z)$ where $f(z)=-\ln z$ and $Z$ has the following distribution: 
${\sf P}(Z=P_{\ell}/t_{\ell})=t_{\ell}$.\\

3. For a positive integer $d$ and integers $y_1,\ldots,y_n$ consider the event $\mathcal{E}_{d}(y_1,\ldots,y_n)$ saying that the number of $i\in[n]$ such that $y_i=d$ is greater than $2\sqrt{n}$. 
To prove the third part of Lemma~\ref{cl:degrees_existence}, it is sufficient to show that, for every $x\in I_n$, ${\sf P}(\exists d\,\,\mathcal{E}_d(\eta_1(x),\ldots,\eta_n(x)))=o(1)$ and apply Theorem~\ref{thm:Wormald_McKay}. By the union bound, it is enough to show that, for every $d\in[n-1]$,
$$
{\sf P}(\mathcal{E}_d(\eta_1(x),\ldots,\eta_n(x)))=o\left(\frac{1}{n}\right).
$$
Fix $d\in[n-1]$. Let $X$ count the number of $i\in[n]$ such that $\eta_i(x)=d$. By the de Moivre--Laplace limit theorem, $\pi_d:={\sf P}(\eta_i(x)=d)\leq\frac{1}{\sqrt{\pi n/2}}(1+o(1))$. Since $X$ has Binomial distribution with parameters $n$ and $\pi_d$, we get by the Chernoff bound that 
$$
{\sf P}(\mathcal{E}_d(\eta_1(x),\ldots,\eta_n(x)))={\sf P}(X>2\sqrt{n})\leq e^{-\frac{(2-\sqrt{2/\pi})^2}{2(\sqrt{2/\pi}+(2-\sqrt{2/\pi})/3)}\sqrt{n}}=o\left(\frac{1}{n}\right).
$$

% Set $d_0=0$. For $\mathcal{I}=\{i_1,\ldots,i_{\ell}\}\subset\mathbb{Z}^{\ell}$, where $0\leq i_1<\ldots<i_{\ell}\leq m-1$, set $p_{\mathcal{I}}:=p^{\ell+1}_{d_{i_2}-d_{i_1},\ldots,d_{i_{\ell}}-d_{i_1}}$. Then
%$$
% p=2\left(\frac{\varepsilon}{2}\right)^m+\sum_{\ell=2}^{m}\left(\frac{\varepsilon}{2}\right)^{m-\ell}\left(\sum_{V\in{\{0,\ldots,m-1\}\choose\ell}\sum_{\mathcal{I}_1\sqcup\ldots\mathcal{I}}p^{\ell+1}_{d_{i_2}-d_{i_1},\ldots,d_{i_{\ell}}-d_{i_1}}\right)
%$$

 %this probability is exactly
%It is sufficient to prove that $S^{m}_{d_1,\ldots,d_{m-1}}(X_1,\ldots,X_t)\stackrel{\sf P}\to p^m_{d_1,\ldots,d_{m-1}}(F,u)$ as $t\to\infty$.

\section{Graph families}

\subsection{Complete graphs}
\label{sc:cliques}

In this section, we prove Theorem~\ref{thm:cliques}. Let $G=K_n$ and $u$ be its vertex. Let us find $p_2$. 

Let us fix positive integers $1\leq k\leq t$ and a vertex $v$ of $G$.
Let $\mathcal{A}_{t,k}(v,u)$ be the event saying that $\pi_t$ meets $\pi_{t+2}$ for the first time at point $(v,t-k+1)$. 

First, consider $k=1$. Then $\mathcal{A}_{t,1}(v,u)=\varnothing$ if and only if $v\neq u$. Moreover, ${\sf P}(\mathcal{A}_{t,1}(u,u))=\frac{(1-\varepsilon)^2}{n-1}$. 

For $k=2$, we get that $\mathcal{A}_{t,2}(v,u)=\varnothing$ if and only if $v=u$. If $v\neq u$, then  
\begin{align*}
{\sf P}(\exists v\,\,(v,t-1)\text{ is the first common vertex of }\pi_t,\pi_{t+2}) \,\, 
&=\,\,{\sf P}(\sqcup_{v\neq u}\mathcal{A}_{t,2}(v,u))\\
&=\,\,\frac{(1-\varepsilon)^4(n-2)^2}{(n-1)^3},
\end{align*} 
since the second edge of $\pi_{t+2}$ might reach a vertex $w$ other than $u$ (it happens with probability $\frac{n-2}{n-1}$), the first edge of $\pi_t$ might reach a vertex $v$ other than $w$ (it happens with probability $\frac{n-2}{n-1}$ as well), and the third edge of $\pi_{t+2}$ might reach the same vertex $v$ (with probability $\frac{1}{n-1}$).

Finally, let $k\geq 3$. Then $\mathcal{A}_{t,k}(v,u)\neq\varnothing$ for all choices of $v$ and does not depend on this choice. Let us compute ${\sf P}(\mathcal{A}_{t,k}(u))$, where $\mathcal{A}_{t,k}(u)=\sqcup_{v}\mathcal{A}_{t,k}(v,u)$. Assume that the second edge of $\pi_{t+2}$ reaches a vertex $w\neq u$ (with probability $\frac{n-2}{n-1}$). Then the number of choices (to satisfy $\mathcal{A}_{t,k}(u)$) of the destinations of edges initiated at $(u,t)$ and $(w,t)$ equals $(n-2)(n-3)+2(n-1)-1$ (there are $(n-2)(n-3)$ choices of a pair of destinations other than $u$ and $w$, $n-1$ choices of an edge $((w,t),(z,t-1))$ when $((u,t),(w,t-1))\in E(\mathcal{D}_G)$, and $n-1$ choices of an edge $((u,t),(z,t-1))$ when $((w,t),(u,t-1))\in E(\mathcal{D}_G)$). In the same way, if, for some $i\in[k-3]$, $(w_1,t-i)$ and $(w_2,t-i)$ are the $(i+1)$th and the $(i+3)$th vertices of $\pi_t$ and $\pi_{t+2}$ respectively, then there are $(n-2)(n-3)+2(n-1)-1$ choices of the destinations of $(w_1,t-i)$ and $(w_2,t-i)$ in $E(\mathcal{D}_G)$. Finally, if $(w_1,t-k+2)$ and $(w_2,t-k+2)$ are the $(k-1)$th and the $(k+1)$th vertices of $\pi_t$ and $\pi_{t+2}$ respectively, then the probability that they meet at the next step is exactly $\frac{n-2}{(n-1)^2}$. We get
$$
 {\sf P}(\mathcal{A}_{t,k}(u))=(1-\varepsilon)^{2k}\frac{n-2}{n-1}\left(\frac{(n-2)(n-3)+2(n-1)-1}{(n-1)^2}\right)^{k-2}\frac{n-2}{(n-1)^2}.
$$
Summing up,
\begin{align}
 p_2\,\,
 &=\,\,\frac{(1-\varepsilon)^2}{n-1}+\sum_{k=2}^{\infty}\frac{(1-\varepsilon)^{2k}(n-2)^2}{(n-1)^3}\left(\frac{n^2-3n+3}{(n-1)^2}\right)^{k-2}\notag\\
 &=\,\,\frac{(1-\varepsilon)^2}{n-1}+\frac{(1-\varepsilon)^4(n-2)^2}{(n-1)((n-1)^2-(1-\varepsilon)^2(n^2-3n+3))}.\label{eq:cliques_p2}
\end{align}

In the same way, we get that 
$$
 p_1=\sum_{k=2}^{\infty}\frac{(1-\varepsilon)^{2k-1}(n-2)}{(n-1)^2}\left(\frac{n^2-3n+3}{(n-1)^2}\right)^{k-2}
$$
implying that $p_2=\frac{(1-\varepsilon)^2}{n-1}+(1-\varepsilon)\frac{n-2}{n-1}p_1$. Due to Theorem~\ref{th:NVM_main}, it remains to show that one of the two equalities $p_1(K_{n_1},u_1)=p_1(K_{n_2},u_2)$, $p_2(K_{n_1},u_1)=p_2(K_{n_2},u_2)$ is always false. Assume the contrary --- both inequalities hold for some $n_1\neq n_2$. Then, letting $p_1:=p_1(K_{n_1},u_1)$, we get 
$$
 p_1+\frac{1-\varepsilon-p_1}{n_1-1}=\frac{1-\varepsilon}{n_1-1}+\frac{n_1-2}{n_1-1}p_1=
 \frac{1-\varepsilon}{n_2-1}+\frac{n_2-2}{n_2-1}p_1=
 p_1+\frac{1-\varepsilon-p_1}{n_2-1}
$$
implying that $n_1=n_2$ --- a contradiction. This completes the proof of Theorem~\ref{thm:cliques}.

\begin{remark}
Let us show that $p_2$ is not sufficient to apply Theorem~\ref{th:NVM_main}. The function of $n$ in the right hand side of (\ref{eq:cliques_p2}) decreases for $n\geq 6$. Indeed, set $f(x)=\frac{1}{x}+\frac{(1-\varepsilon)^2(x-1)^2}{x(x^2-(1-\varepsilon)^2(x^2-x+1))}$. If $x\geq 5$, then 
\begin{equation}
 f'(x)=\frac{-x^4(2\varepsilon-\varepsilon^2)^2-x^2(x-1)(x-5)(1-\varepsilon)^2(2\varepsilon-\varepsilon^2)}{x^2(x^2-(1-\varepsilon)^2(x^2-x+1))}<0
 \label{eq:cliques_derivative}
\end{equation}
as needed. Therefore, for any vertex $u_1$ of $K_{n_1}$, and any $u_2$ of $K_{n_2}$, $6\leq n_1< n_2$, $p_2(K_{n_1},u_1)\neq p_2(K_{n_2},u_2)$. However, it is easy to see that $f'(x)$ in (\ref{eq:cliques_derivative}) is positive on $[2,4]$ for $\varepsilon$ close to 0 and negative on $[2,4]$ for $\varepsilon$ close to 1. Therefore, there exists $\varepsilon\in(0,1)$ such that $f(2)=f(4)$, i.e. $p_2(K_3,u_1)=p_2(K_5,u_2)$, and Theorem~\ref{th:NVM_main} can not be applied for $d=2$.
\end{remark}

% Note that every path of length $k+1$ (i.e. comprising $k+1$ edges) from $(u,t+2)$ to $(u_1,t-k+1)$ that does not meet points $(u_{i+1},t-k+i+1)$, $i\in[k-1]$, has probability $\frac{(1-\varepsilon)^{k+1}}{(n-1)^{k+1}}$, and there are $(n-2)^{k-1}(n-1)$ such paths. Therefore,
%$$
% {\sf P}\biggl(\mathcal{A}_{t,k}(u_1,u)\,\biggl|\,\mathcal{B}_t(u_1,\ldots,u_{k-1})\biggr)=\frac{(1-\varepsilon)^{k+1}}{(n-1)^k}(n-2)^{k-1}.
%$$
%Then
%\begin{align*}
% p_2 & =\lim_{t\to\infty}\sum_{k=1}^{t}\sum_{u_1,\ldots,u_{k-1}}\frac{(1-\varepsilon)^{k+1}}{(n-1)^k}(n-2)^{k-1}{\sf P}(\mathcal{B}_t(u_1,\ldots,u_{k-1}))\\
% &=\sum_{k=1}^{\infty}\sum_{u_1,\ldots,u_{k-1}}\frac{(1-\varepsilon)^{k+1}}{(n-1)^k}(n-2)^{k-1}\frac{(1-\varepsilon)^{k-1}}{(n-1)^{k-1}}=\sum_{k=1}^{\infty}\frac{(1-\varepsilon)^{2k}}{(n-1)^k}(n-2)^{k-1}\\
% &=\frac{(1-\varepsilon)^2}{n-1}\sum_{k=0}^{\infty}\frac{(1-\varepsilon)^{2k}(n-2)^k}{(n-1)^k}= \frac{(1-\varepsilon)^2}{(n-1)(n-1-(1-\varepsilon)^2(n-2))}
% \end{align*}
% decreases in $n$. Therefore, observations in any $u_1$ of $K_{n_1}$, $u_2$ of $K_{n_2}$, $n_1\neq n_2$, distinguish between $(K_{n_1},u_1)$ and $(K_{n_2},u_2)$. $\Box$

\subsection{Cycles}
\label{sc:cycles}

In this section, we prove Theorem~\ref{thm:cycles}. Let $G=C_n$ and $u$ be its vertex. Let us find $p_1$ and $p_2$. Note that, if $n$ is even, then $p_1=0$. Assume that $n$ is odd. One of the possibilities is that $\pi_t$ meets $\pi_{t+1}$ for the first time at a point $(v,s)$ in $\mathcal{D}_G$ so that $v$ is at distance $\lfloor n/2\rfloor$ from $u$, $s=t-\lfloor n/2\rfloor$, and $\pi_t$, $\pi_{t+1}$ follow from $u$ to $v$ in opposite directions. Therefore, $p_1\geq\left(\frac{1-\varepsilon}{2}\right)^n$. Due to Theorem~\ref{th:NVM_main}, it immediately implies that $S^{(1)}$ distinguishes between two cycles such that $n_1,n_2$ have different parities. 

In order to distinguish $C_{n_1},C_{n_2}$ such that $n_1,n_2$ have equal parities, we prove that $p_2(C_{n_1},u_1)\neq p_2(C_{n_2},u_2)$ and apply Theorem~\ref{th:NVM_main} as usual. For convenience, we assume that both $n_1,n_2$ are odd. Otherwise, the proof is essentially the same (just replace everywhere below $\pm 2n_1$ and $\pm 2n_2$ with $\pm n_1$ and $\pm n_2$ respectively).

Let us consider $p_2$ defined for $(G=C_n,u)$. Consider a simple symmetric random walk $R_s$, $s\in\mathbb{Z}_+$, on $\mathbb{Z}$. Let $\xi_s$, $s\in\mathbb{N}$, be a sequence of independent Bernoulli random variables with success probability $1-\varepsilon$, and let $\tau=\min\{s:\xi_s=0\}$. Hereinafter we let $R_s$ terminate as soon as $\xi_s=0$, and denote {\it the terminable} version of $R_s$ by $R_s^*$ (i.e. $R_s^*=R_s$ for all $s<\tau$). We stick to the convention that the value of $R_s^*$ is undefined whereas $s\geq\tau$. Consider two independent simple symmetric terminable random walks $R^*_s$ and $\tilde R^*_s$ on $\mathbb{Z}$. Note that ${\sf P}(\pi_t\cap\pi_{t+2}\neq\varnothing)$ equals the probability that, at some moment $s$ (which is less than both independent stopping times $\tau$ and $\tilde\tau$), $\tilde R^*_{s+2}-R^*_s$ is divisible by $n$. The values of $\mathcal{R}_s:=\tilde R^*_{s+2}-R^*_s$ are even, and $-2,0,2$ are all the possible values of $\mathcal{R}_{s+1}-\mathcal{R}_s$. 

Now, assume that $n_2>n_1$. It remains to prove that
\begin{equation*}
 {\sf P}(\exists s<\min\{\tau,\tilde\tau\}\,\,\, \mathcal{R}_s\text{ is divisible by }n_1)>
 {\sf P}(\exists s<\min\{\tau,\tilde\tau\}\,\,\, \mathcal{R}_s\text{ is divisible by }n_2).
%\label{eq:n1_n2}
\end{equation*}
If at some moment $s<\tau^*:=\min\{\tau,\tilde\tau\}$, $n_2$ divides $\mathcal{R}_s$, then, either $\mathcal{R}_s=0$, and then $n_1$ divides $\mathcal{R}_s$ as well, or $\mathcal{R}_s=2jn_2$ for some integer $j\neq 0$. In the latter case, by continuity, at some moment $s^*<s$, %$\mathcal{R}_{s^*}=\pm 2n_2$, and, at some moment $s^{**}<s^*$, 
 $\mathcal{R}_{s^{*}}=\pm 2n_1$. Therefore, the existence of $s<\tau^*$ such that $\mathcal{R}_s$ is divisible by $n_2$ implies the existence of $s<\tau^*$ such that $\mathcal{R}_s$ is divisible by $n_1$. It remains to show that there is an event that has positive probability and happens when $\mathcal{R}$ reaches a number divisible by $n_1$, but never reaches a number divisible by $n_2$. But this is obvious: on the one hand, with positive probability, there exists $s<\tau^*$ such that $\mathcal{R}_s=2n_1$ and, for all $t<s$, $\mathcal{R}_t$ is not divisible by $n_1$ (just set $s=n_1$ and let $\mathcal{R}_s$ move $s$ times upward). On the other hand, if $\mathcal{R}_s=2n_1$, then the probability that $\mathcal{R}_{\geq s}$ terminates before it reaches the value $0$ or $2n_2$ is at least $\varepsilon$. % \mathcal{R}_t=\pm 2n_2$ at some moment $t>s$ is at most $(1-\varepsilon)^2$. This immediately implies that

%Note also that if $\mathcal{R}_s=0$ and $n_1\text{ does not divide }\mathcal{R}_t$ for all $t<s$, then the same is true for $n_2$. Therefore,
%\begin{align*}
%{\sf P}(\exists s<\tau^*\,\,\, n_1\text{ divides }\mathcal{R}_s)\,\,
%&=\,\,{\sf P}(\exists s<\tau^*\,\,\,\mathcal{R}_s=0,\,\,\forall t<s\,\,2\leq|\mathcal{R}_t|\leq 2n_1-2)\\
%&\quad\quad +{\sf P}(\exists s<\tau^*\,\,\,\mathcal{R}_s=\pm2n_1,\,\, \forall t<s\,\,2\leq|\mathcal{R}_t|\leq 2n_1-2)\\
%&=\,\,{\sf P}(\exists s<\tau^*\,\,\,\mathcal{R}_s=0,\,\,\forall t<s\,\,2\leq|\mathcal{R}_t|\leq 2n_2-2)\\
%&\quad\quad +{\sf P}(\exists s<\tau^*\,\,\,\mathcal{R}_s=\pm2n_2,\,\, \forall t<s\,\,2\leq|\mathcal{R}_t|\leq 2n_2-2)\\
%&=\,\,{\sf P}(\exists s<\tau^*\,\,\, n_2\text{ divides }\mathcal{R}_s).
%\end{align*}
%With positive probability, there exists $s$ such that $\mathcal{R}_s=\pm2n_1$ and, for all $t<s$, $\mathcal{R}_t$ is not divisible by $n_1$. Finally, if $\mathcal{R}_s=\pm2n_1$, then the probability that $\mathcal{R}_t=\pm 2n_2$ at some moment $t>s$ is at most $(1-\varepsilon)^2$. This immediately implies that
%\begin{multline*}
% {\sf P}(\exists s<\tau^*\,\,\,\mathcal{R}_s=\pm2n_1,\,\, \forall t<s\,\,2\leq|\mathcal{R}_t|\leq 2n_1-2)>\\
% >{\sf P}(\exists s<\tau^*\,\,\,\mathcal{R}_s=\pm2n_2,\,\, \forall t<s\,\,2\leq|\mathcal{R}_t|\leq 2n_2-2),
%\end{multline*}
%(\ref{eq:n1_n2}) follows.

\subsection{Perfect trees}
\label{sc:trees}

In this section, we prove Theorem~\ref{th:trees}.

%As we mention in Introduction, NVM does not distinguish between stars observed at their centers. It is natural to ask, does NVM distinguish between two $k_1$-ary and $k_2$-ary perfect trees with the same height, but $k_1\neq k_2$. In this section, we answer positively for any height at least 2. In this section we ask positively

%A {\it perfect $k$-ary tree} $(T,v)$ of height $h$ is a tree $T$ with root $v$ such that every non-leaf vertex has exactly $k$ children, and every leaf is at distance exactly $h$ from $v$.\\

1. Let $(T,v)$ be a perfect $k$-ary tree of height $h$, $h\geq 2$ is fixed. It is sufficient to prove that $p_2=p_2(k)$ decreases with $k$ and apply Theorem~\ref{th:NVM_main}. Clearly, $p_2=(1-\varepsilon)^2\frac{1}{k+1}+(1-\varepsilon)^2\frac{k}{k+1}f(k)$, where $f(k)$ is the probability that two {\it terminable} independent random walks (at every step, a random walk terminates with probability $\varepsilon$) on $T$ originated at $v$ and at a vertex at distance 2 from $v$ meet. Since $f(k)<1$, it remains to prove that $f(k)$ decreases with $k$. For any two vertices $u,w$ and time moment $t$, let $P_t(T;u,w)$ be the probability that terminable random walks on $T$ originated at $u$ and $w$ meet before the time $t$.

Let $(T',v)$ be a perfect $(k+1)$-ary tree of the same height. Consider a monomorphism $\varphi:V(T)\to V(T')$ of rooted trees. Let us prove by induction on $t$, that, for any distinct $u,w\in V(T)$ at an even distance from each other (otherwise, two random walks initiated at $u$ and $w$ never meet), $P_t(T;u,w)> P_t(T';\varphi(u),\varphi(w))$ (this is clearly sufficient). This is obvious for $t=1$: if the vertices are at distance $2$, then $P_t(T;u,w)=\frac{(1-\varepsilon)^2}{(k+1)^2}$ and $P_t(T';\varphi(u),\varphi(w))=\frac{(1-\varepsilon)^2}{(k+2)^2}$; otherwise both probabilities equal to 0. Assume that the inequality holds for some $t$. %If $u=w$, then
%$$
%P_{t+1}(T';\varphi(u),\varphi(w))=P_{t+1}(T;u,w)=1.
%$$ 
%If $u$ and $w$ are at distance 1, then both probabilities equal to 0. Let the distance be at least 2. 

First, let $u,w$ be not leaves. Consider neighbors $u^{-},u^+$ of $u$ and neighbors $w^-,w^+$ of $w$ such that $u^-$ and $w^-$ belong to the path between $u$ and $w$ while $u^+$ and $w^+$ do not. Then
\begin{align*}
 P_{t+1}(T';\varphi(u),\varphi(w)) \,\,
 &=\,\, 
 \frac{k^2 P_t(T';\varphi(u^+),\varphi(w^+))}{(k+1)^2}+
  \frac{k P_t(T';\varphi(u^-),\varphi(w^+))}{(k+1)^2}\\
 &\quad\quad\quad\quad +\frac{kP_t(T';\varphi(u^+),\varphi(w^-))}{(k+1)^2}+
 \frac{P_t(T';\varphi(u^-),\varphi(w^-))}{(k+1)^2}\\
 &<\,\, \frac{k^2 P_t(T;u^+,w^+)}{(k+1)^2}+
 \frac{k P_t(T;u^-,w^+)}{(k+1)^2}\\
 &\quad\quad\quad\quad+\frac{k P_t(T;u^+,w^-)}{(k+1)^2}+
 \frac{P_t(T;u^-,w^-)}{(k+1)^2}\\
 &\leq\,\,\frac{(k-1)^2 P_t(T;u^+,w^+)}{k^2}+
 \frac{(k-1) P_t(T;u^-,w^+)}{k^2}\\
 &\quad\quad\quad\quad+\frac{(k-1) P_t(T;u^+,w^-)}{k^2}+
 \frac{P_t(T;u^-,w^-)}{k^2}\\
 &=\,\,P_{t+1}(T;u,w),
\end{align*}
where the last inequality is obtained by spreading out the extra part $\left(\frac{k}{k+1}\right)^2$-$\left(\frac{k-1}{k}\right)^2$ of probability $\left(\frac{k}{k+1}\right)^2$ across the remaining three probabilities and applying the claim stated below.\\

\begin{claim}
Let $u,u',w$ be vertices of $T$ such that the distance between $u$ and $w$ is even in $T$, and $u'$ belongs to the path joining $u$ with $w$ and is at distance 2 from $u$. Then $P_t(T;u,w)\leq P_t(T;u',w)$.\\
\label{cl:monotone_path_length}
\end{claim}

The proof of this claim is straightforward since without loss of generality, by symmetry, we may assume that the vertex $R_s(u')$ occupied at time $s$ by a random walk initiated at $u'$ stands on the path between the vertices $R_s(u)$ and $R_s(w)$ occupied at time $s$ by random walks initiated at $u$ and $w$ respectively until it meets one of the endpoints of the path. Indeed, if $R_s(u')$ is an inner vertex of the unique path between $R_s(u)$ and $R_s(w)$ but $R_{s+1}(u')$ does not belong to the path between $R_{s+1}(u)$ and $R_{s+1}(w)$, then, by symmetry of $T$, we may find a vertex $z$ from the path between $R_{s+1}(u)$ and $R_{s+1}(w)$ such that replacing $R_{s+1}(u')$ with $z$ does not change the probability $P_t(T;u',w)$.\\ %appears after the proof of the second part of Theorem~\ref{th:trees}.

In the same way, if $u$ is a leaf, and $w$ is not, then
\begin{align*}
 P_{t+1}(T';\varphi(u),\varphi(w))\,\, 
 & = \,\,
\frac{k}{(k+1)^2}P_t(T';\varphi(u^-),\varphi(w^+))+
 \frac{1}{(k+1)^2}P_t(T';\varphi(u^-),\varphi(w^-))\\
 &< \,\, \frac{k}{(k+1)^2}P_t(T;u^-,w^+)+
 \frac{1}{(k+1)^2}P_t(T;u^-,w^-)\\
 &< \,\, \frac{k-1}{k^2}P_t(T;u^-,w^+)+
 \frac{1}{k^2}P_t(T;u^-,w^-)\,\,=\,\,P_{t+1}(T;u,w),
\end{align*}
since both $\frac{k}{(k+1)^2}$ and $\frac{1}{(k+1)^2}$ decrease with $k$.  Finally, if both $u$ and $w$ are leaves, then
\begin{align*}
 P_{t+1}(T';\varphi(u),\varphi(w)) \,\, 
 & = \,\,
 \frac{1}{(k+1)^2}P_t(T';\varphi(u^-),\varphi(w^-))\\
 & < \,\,
 \frac{1}{(k+1)^2}P_t(T;u^-,w^-) \,\, < \,\,
 \frac{1}{k^2}P_t(T;u^-,w^-) \,\, = \,\, P_{t+1}(T;u,w).
\end{align*}

2.  Let $(T,v)$ be a perfect $k$-ary tree of height $h$, and $k\geq 2$ is fixed. It is sufficient to prove that $p_2=p_2(h)$ decreases with $h$ and apply Theorem~\ref{th:NVM_main}. Clearly, $p_2=(1-\varepsilon)^2\frac{1}{k+1}+(1-\varepsilon)^2\frac{k}{k+1}f(h)$, where $f(h)$ is the probability that two terminable independent random walks on $T$ originated at $v$ and at a vertex at distance 2 from $v$ meet. It remains to prove that $f(h)$ decreases. Let $P(T;u,v)$ be the probability that two terminable random walks on $T$ originated at $u$ and $v$ meet. 

Extend $(T,v)$ to a perfect $k$-ary tree $(T',v)$ of height $h+1$ (we have $T\subset T'$). For two random walks $\bold{R}^1=(R^1_t:=R_t(u),t\geq 0)$, $\bold{R}^2=(R^2_t:=R_t(v),t\geq 0)$ on $T'$ initiated at vertices $u\in V(T)$ and $v$, let $\sigma\in\mathbb{N}\cup\{\infty\}$ be the first moment when $\bold{R}^1$ and $\bold{R}^2$ meet. Consider {\it truncated} versions of these two random walks denoted by $\bold{\tilde R}^1=(\tilde R_t^1,t\geq 0)$, $\bold{\tilde R}^2=(\tilde R_t^2,t\geq 0)$ and defined on $T$ recursively in a way (described below) such that, for every $t\leq\sigma$, both $\tilde R_t^1,\tilde R_t^2$ belong to the path between $R_t^1$ and $R_t^2$. Let $\tau^1\in\mathbb{N}\cup\{\infty\}$ and $\tau^2\in\mathbb{N}\cup\{\infty\}$ be the first moments when $\bold{R}^1$ and $\bold{R}^2$ reach a leaf of $T'$ respectively. Let $\gamma_t$ be the path between $R^1_t$ and $R^2_t$. For every $t=0,1,2,\ldots,\sigma$,
\begin{itemize}
\item for $j\in\{1,2\}$, if $t<\tau^j$, then set $\xi^j_t:=0$; %if $t<\tau(v)$, then set $\xi^2_t:=0$;
\item for $j\in\{1,2\}$, if $t=\tau^j$, then set $\xi^j_t:=1$; %if $t=\tau(v)$, then set $\xi^2_t:=1$;
\item for $j\in\{1,2\}$, if $t>\tau^j$, and at least one of the following two conditions holds: $\tilde R_{t-1}^j\neq R_{t-1}^j$ or $R^j_{t}$ is a leaf of $T'$, then $\xi^j_{t}=1$;
\item for $j\in\{1,2\}$, if $t>\tau^j$, $\tilde R_{t-1}^j=R_{t-1}^j$ and $R_{t}^j$ is not a leaf of $T'$, then $\xi_{t}^j=0$;
%\item if $t\geq\tau(v)$, $\tilde R_t(v)\neq R_t(v)$ or $\tilde R_{t+1}(v)$ is a leaf of $T'$, then $\xi_{t+1}(w)=1$;
%\item if $t\geq\tau(v)$, $\tilde R_t(v)=R_t(v)$ and $\tilde R_{t+1}(v)$ is not a leaf of $T'$, then $\xi_{t+1}(w)=0$.
\item for $j\in\{1,2\}$, if $\xi^j_t=0$, then $\tilde R^j_t=R^j_t$;  
%\item if $\xi_t(v)=0$, then $\tilde R_t(v)=R_t(v)$;
\item for $j\in\{1,2\}$, if $\xi_t^j=1$, then $\tilde R_t^j$ is obtained from $\tilde R_{t-1}^j$ by following the edge of $\gamma_t$ going from $\tilde R_{t-1}^j$ to its child in $T$ with probability $\frac{k}{k+1}$ and the edge of $\gamma_t$ going from $\tilde R_{t-1}^j$ to its parent with probability $\frac{1}{k+1}$ (if $\tilde R_{t-1}^j$ is a leaf of $T$, then $\bold{\tilde R}^j$ follows the unique edge of $\gamma_t$ in $T$ at time $t$).
%\item if $\xi_t(v)=1$, then consider the path $\pi$ in $T'$ between $R_t(u)$ and $R_t(v)$; $\tilde R_t(v)$ is obtained from $\tilde R_{t-1}(v)$ by following the edge of $\pi$ going from $\tilde R_{t-1}(v)$ in the direction to $R_t(v)$ with probability $\frac{k}{k+1}$ and the edge of $\pi$ going from $\tilde R_{t-1}(v)$ in the direction to $R_t(u)$ with probability $\frac{1}{k+1}$  (if $\tilde R_{t-1}(v)$ is a leaf, then $\bold{\tilde R}$ follows the unique edge of $\pi$ in $T$ at time $t$); 
%\item if $t\geq\tau(u)$, $\tilde R_t(u)\neq R_t(u)$ or $\tilde R_{t+1}(u)$ is a leaf of $T'$, then $\xi_{t+1}(u)=1$;
%\item if $t\geq\tau(u)$, $\tilde R_t(u)=R_t(u)$ and $\tilde R_{t+1}(u)$ is not a leaf of $T'$, then $\xi_{t+1}(u)=0$;
%\item if $t\geq\tau(v)$, $\tilde R_t(v)\neq R_t(v)$ or $\tilde R_{t+1}(v)$ is a leaf of $T'$, then $\xi_{t+1}(w)=1$;
%\item if $t\geq\tau(v)$, $\tilde R_t(v)=R_t(v)$ and $\tilde R_{t+1}(v)$ is not a leaf of $T'$, then $\xi_{t+1}(w)=0$.
\end{itemize}
Since, for every $t$, both $\tilde R_t^1,\tilde R_t^2$ belong to $\gamma_t$, we get that the probability $\tilde P(T;u,v)$ that $\bold{\tilde R}^1$ and $\bold{\tilde R}^2$ meet is strictly greater than $P(T';u,v).$ It remains to notice that $P(T;u,v)=\tilde P(T;u,v)$ due to the following claim.\\

\begin{claim}
For every $t$ and for a $k$-vertex set $U\subset V(T)$ consider a distribution ${\sf Q}_t(U)$ on $U$. Let $\bold{R}^{\sf Q}(u)$ and $\bold{R}^{\sf Q}(v)$ be two random walks on $T$ initiated at $u$ and $v$ respectively defined as follows. For every $t\geq 0$, let $\gamma^{{\sf Q}}_t$ be the path that joins $R_t^{{\sf Q}}(u)$ with $R_t^{{\sf Q}}(v)$. Then, with probability $\frac{1}{k+1}$, $R_{t+1}^{{\sf Q}}(u)=z$, where $z$ is the vertex adjacent to $R_t^{{\sf Q}}(u)$ in $\gamma^{{\sf Q}}_t$ which is closer to $R_t^{{\sf Q}}(v)$, or with probability $\frac{k}{k+1}$, $\bold{R}^{{\sf Q}}(u)$ moves to the neighbor of $R_t^{{\sf Q}}(u)$ according to the distribution ${\sf Q}_t(U)$, where $U$ is the set of all neighbors of $R_t^{{\sf Q}}(u)$ in $T$ other than $z$ (if $R_t^{{\sf Q}}(v)$ is a leaf, then $\bold{R}^{{\sf Q}}(u)$ follows the only possible edge). The move of $\bold{R}^{{\sf Q}}(v)$ at time $t+1$ is defined in the same way. Then the probability that $\bold{R}^{\sf Q}(u)$ meets $\bold{R}^{\sf Q}(v)$ equals $P(T;u,v)$.\\
\label{cl:monotone_path_length}
\end{claim}

The proof is straightforward due to the symmetry of $T$. The proof of Theorem~\ref{th:trees} is completed.

\begin{remark}
It is possible to show that there exist distinct $k_1,k_2\geq 2$, distinct $h_1,h_2\geq 2$ and $\varepsilon>0$ such that perfect $k_1$-ary and $k_2$-ary trees of heights $h_1$ and $h_2$ respectively observed at their roots are not $S^{(2)}$-distinguishable. However, it is not hard to see that they are distinguishable for almost all $\varepsilon$ since the coefficients in front of $(1-\varepsilon)^2$ in $p_2$ are different.
\end{remark}

\section{Integer lattices}
\label{sc:lattices}

In this section, we observe NVM at a vertex of an infinite graph (namely, $\mathbb{Z}^d$). We show that observations of NVM may be used to reconstruct the dimension of $\mathbb{Z}^d$. 

Define an {\it $\varepsilon$-terminable} random walk $R_s$, $s\in\mathbb{Z}_+$, on $\mathbb{Z}^d$ as follows. It originates at some non-random point $R_0\in\mathbb{Z}^d$, and, for every positive $s$, $R_s$ either terminates with probability $\varepsilon$ or, with probability $1-\varepsilon$, moves to a neighbor of $R_{s-1}$ chosen uniformly at random. 

It is easy to see that 
\begin{equation}
p_2=\frac{(1-\varepsilon)^2}{2^d}+\frac{(1-\varepsilon)^2}{2^{2d}}\left(\sum_{\mathbf{r}\in\mathcal{R}_1}\mu(d,\mathbf{r})+2\sum_{\mathbf{r}\in\mathcal{R}_2}\mu(d,\mathbf{r})\right),
\label{eq:lattice_NVM_meet_prob}
\end{equation}
where 
\begin{itemize}
\item $\mathcal{R}_1=\mathcal{R}_1(d)$ is the set of all $\mathbf{r}=(r_1,\ldots,r_d)\in\mathbb{Z}^d$ such that, for some $i\in[d]$, $|r_i|=2$ and, for every other $j\in[d]$, $r_j=0$;
\item $\mathcal{R}_2=\mathcal{R}_2(d)$ is the set of all $\mathbf{r}=(r_1,\ldots,r_d)\in\mathbb{Z}^d$ such that, for some $i\neq j\in[d]$, $|r_i|=|r_j|=1$ and, for every other $\ell\in[d]$, $r_{\ell}=0$;
\item $\mu(d,\mathbf{r})$ is the probability that two independent $\varepsilon$-terminable random walks on $\mathbb{Z}^d$ that originate at points $0$ and $\mathbf{r}$ meet.
\end{itemize} 
Note that, if $d=1$, then $\mathcal{R}_2=\varnothing$, and so the respective summand vanishes.\\

%Below, we give a definition of terminable random walks and define $\mu(d,\mathbf{r})$. 

%Lemma~\ref{lm:lattices} that is stated and proven in this section immediately implies that NVM distinguishes between different $d$.\\

In what follows, we prove that $p_2$ decreases with $d$. It immediately implies Theorem~\ref{thm:lattices} due to Theorem~\ref{th:NVM_main}.\\

For every $d\in\mathbb{N}$, consider two independent $\varepsilon$-terminable random walks $R_s^1(d)$, $R_s^2(d)$, $s\in\mathbb{Z}_+$, on $\mathbb{Z}^d$ with origins at 0 and $\mathbf{r}=\mathbf{r}(d)$ respectively. Let $\mu(d,\mathbf{r})$ be the probability that they meet before they terminate.

Lemma~\ref{lm:lattices} stated below, the equality (\ref{eq:lattice_NVM_meet_prob}) and the fact that $\mu(d,\mathbf{r})$ is stable with respect to permutations of the coordinates of $\mathbf{r}$ imply the statement of Theorem~\ref{thm:lattices} since 
\begin{align*}
 \sum_{\mathbf{r}\in\mathcal{R}_1(d+1)}\mu(d+1,\mathbf{r}) \,\, 
 &= \,\,
 \sum_{\mathbf{r}\in\mathcal{R}_1(d)}\mu(d+1,(\mathbf{r},0))+\mu(d+1,(\mathbf{0},2))+
 \mu(d+1,(\mathbf{0},-2))\\
 &\leq \,\, 2\sum_{\mathbf{r}\in\mathcal{R}_1(d)}\mu(d,(\mathbf{r},0))\,\,\leq\,\, 2\sum_{\mathbf{r}\in\mathcal{R}_1(d)}\mu(d,\mathbf{r}),
\end{align*}
and, for $d\geq 2$,
\begin{multline*}
 \sum_{\mathbf{r}\in\mathcal{R}_2(d+1)}\mu(d+1,\mathbf{r})=
 \sum_{\mathbf{r}\in\mathcal{R}_2(d)}\mu(d+1,(\mathbf{r},0))+\sum_{|\mathbf{r}|_1=1}[\mu(d+1,(\mathbf{r},1))+\mu(d+1,(\mathbf{r},-1))].
\end{multline*}
%\begin{multline*}
% \sum_{\mathbf{r}(d+1)\in\mathcal{R}_2(d+1)}\mu(d+1,\mathbf{r})=\\
% \sum_{\mathbf{r}(d)\in\mathcal{R}_2(d)}\mu(d+1,(\mathbf{r},0))+\sum_{|\mathbf{r}(d)|_1=1}[\mu(d+1,(\mathbf{r},1))+\mu(d+1,(\mathbf{r},-1))].
%\end{multline*}
The cardinality $|\mathcal{R}_2(d)|$ is $\frac{d-1}{2}\geq\frac{1}{2}$ times bigger than the number of $(d+1)$-vectors $(\mathbf{r},1)$, $(\mathbf{r},-1)$ with $|\mathbf{r}|_1=1$. Therefore, by Lemma~\ref{lm:lattices} and by the symmetry of $\mu$, 
$$
\sum_{\mathbf{r}\in\mathcal{R}_2(d+1)}\mu(d+1,\mathbf{r})\leq 
3\sum_{\mathbf{r}\in\mathcal{R}_2(d)}\mu(d,\mathbf{r}).
$$
Summing up,
$$
p_2(d+1)\leq\frac{(1-\varepsilon)^2}{2^{d+1}}+\frac{(1-\varepsilon)^2}{2^{2d+2}}\left(2\sum_{\mathbf{r}\in\mathcal{R}_1(d)}\mu(d,\mathbf{r})+6\sum_{\mathbf{r}\in\mathcal{R}_2(d)}\mu(d,\mathbf{r})\right)<p_2(d),
$$
completing the proof of Theorem~\ref{thm:lattices}.
 
% \leq 2\sum_{|\mathbf{r}(d+1)|_1=2}\mu(d,(\mathbf{r},0))\leq 2\sum_{|\mathbf{r}(d)|_1=2}\mu(d,\mathbf{r}).\quad\Box
%\end{multline*}

%\begin{multline*}
% \sum_{|\mathbf{r}(d+1)|_1=2}\mu(d+1,\mathbf{r})=\\
% \sum_{|\mathbf{r}(d)|_1=2}\mu(d+1,(\mathbf{r},0))+\mu(d+1,(\mathbf{0},2))+
% \mu(d+1,(\mathbf{0},-2))+\sum_{|\mathbf{r}(d)|_1=1}[\mu(d+1,(\mathbf{r},1))+\mu(d+1,(\mathbf{r},-1))]\\
% \leq 2\sum_{|\mathbf{r}(d+1)|_1=2}\mu(d,(\mathbf{r},0))\leq 2\sum_{|\mathbf{r}(d)|_1=2}\mu(d,\mathbf{r}).\quad\Box
%\end{multline*}

\begin{lemma}
Consider $\mathbf{r}=(r_1,\ldots,r_d)\in\mathbb{Z}^d$ such that $|\mathbf{r}|_1=|r_1|+\ldots+|r_d|=2$. Then $\mu(d,\mathbf{r})\geq\mu(d+1,(\mathbf{r},0))$.\\
%\begin{itemize}
%\item $|r_1|+\ldots+|r_d|=2$, 
%\item for all $d\geq 2$ and $j\in[d]$, $r_j(d)=r_j(d+1)$.
%\end{itemize}
%Then $\mu(d,\mathbf{r})$ is strictly decreasing in $d$.\\
\label{lm:lattices}
\end{lemma}

\begin{proof} Fix $d\in\mathbb{N}$. Define a `combined' random walk $\tilde R(d)$ on $\mathbb{Z}^d$ in the following way: for $s\in\mathbb{Z}_+$, if both $R_s^1(d)$ and $R_s^2(d)$ do not terminate, then set
\begin{itemize}
\item $\tilde R_{2s}(d):=R_s^1(d)-R_s^2(d)$,
\item $\tilde R_{2s+1}(d):=R_{s}^1(d)-R_{s+1}^2(d)$.
\end{itemize}
Let $s$ be the minimum time when either $R^1(d)$ or $R^2(d)$ terminates. Then $\tilde R(d)$ terminates at time $2s-1$. Clearly, $\mu(d,\mathbf{r})$ is exactly the probability that $\tilde R(d)$ reaches 0 before it terminates. 

Let $\mathcal{B}_s(d)$ be the event saying that $\tilde R(d)$ terminates at time $2s-1$ and reaches 0 before the time $2s-1$. %Note that all $\mathcal{B}_s(d)$, $s<\ell$, are empty. 
We are going to show that there is a coupling $(\tilde R(d),\tilde R(d+1))$ such that $\mathcal{B}_s(d+1)\subset\mathcal{B}_s(d)$ for all $s$. It immediately implies that
\begin{equation}
 \mu(d,\mathbf{r})=\sum_s{\sf P}[\mathcal{B}_s(d)]\geq\sum_s{\sf P}[\mathcal{B}_s(d+1)]=\mu(d+1,(\mathbf{r},0)).
\label{eq:mu_monotone}
\end{equation}
%Moreover, we will show that 
%\begin{equation}
%{\sf P}[\mathcal{B}_3(d+1)]>{\sf P}[\mathcal{B}_3(d)]
%\label{eq:mu_monotone_strict}
%\end{equation}
%${\sf P}\left[\mathcal{B}_{\ell}(d+1)\setminus\left(\cup_{s>\ell}\mathcal{B}_s(d+1)\right)\right]<{\sf P}\left[\mathcal{B}_{\ell}(d)\setminus\left(\cup_{s>\ell}\mathcal{B}_s(d)\right)\right]$ 
%implying that the inequality in (\ref{eq:mu_monotone}) is actually strict. \\

Let $\omega=(\omega_t,\, t\leq 2s-2)$ be a trajectory in $\mathbb{Z}^{d+1}$ originated at $(\tilde R_0,0)$ of length $2s-2$.  %Notice that if $d>1$, then the $(d+1)$th coordinate of $\omega_0$ equals 0. 
Consider a trajectory $\varphi(\omega)$ in $\mathbb{Z}^d$ defined as follows. Let $t_1<\ldots<t_b$ be all the time moments when either the first or the last coordinate of $\omega$ changes. Replace them in a way such that the moments that relate to the first coordinate go first and keep the order of moments within the same coordinate. After that, merge these two coordinates. More formally, let $\tau_1<\ldots<\tau_a$ be the time moments when the first coordinate of $\omega$ is changed, and $\sigma_1<\ldots<\sigma_{b-a}$ be the time moments when the last coordinate is changed. Set the $j$th coordinate of $\varphi_t(\omega)$ (denoted by $\varphi_t^j(\omega)$) equal to the $j$th coordinate of $\omega_t$ (denoted by $\omega_t^j$) if $j\in\{2,\ldots,d\}$. Let $\varphi^1(\omega)$ change exactly at $t_1,\ldots,t_b$, and
$$
\varphi^1_{t_{\ell}}(\omega)=\omega^1_{\tau_{\ell}},\,\,\ell\in[a],\quad
\varphi^1_{t_{a+\ell}}(\omega)=\omega^1_{\tau_a}+\omega^{d+1}_{\sigma_{\ell}},\,\,\ell\in[b-a].
$$

Let us denote by $\mathcal{T}_s(d)$ the set of all trajectories in $\mathbb{Z}^d$ originated at $\tilde R_0$ of length $2s-2$. Note that $\varphi$ is a surjection from $\mathcal{T}_s(d+1)$ to $\mathcal{T}_s(d)$ such that, for every trajectory $\omega^*\in\mathcal{T}_s(d)$, %that preserves the distribution, i.e., for every trajectory $\omega^*\in\mathcal{T}_s(d)$, 
$$
{\sf P}\left[\tilde R_{\leq 2s-2}(d+1)\in\varphi^{-1}(\omega^*)\right]
=\frac{(1-\varepsilon)^{2s-2}2^b}{2^{(d+1)(2s-2)}}
\leq\frac{(1-\varepsilon)^{2s-2}}{2^{d(2s-2)}}%(1-(1-\varepsilon)^2)
={\sf P}\left[\tilde R_{\leq 2s-2}(d)=\omega^*\right],
$$
where $b\leq 2s-2$ is the number of times when the first coordinate of $\omega^*$ changes.
%Moreover,
%$$
% \bigcup_{\omega\in\mathcal{B}_s(d+1)}\left\{\tilde R_{\leq 2s-2}(d)=\varphi(\omega),\,\,\tilde R_{2s-2}(d)\text{ terminates}\right\}=\mathcal{B}_s(d)
%$$
%implying (\ref{eq:mu_monotone}) since

Letting $\mathcal{T}^0_s(d)\subset\mathcal{T}_s^d$ to be the set of those trajectories that meet $0$, we get that, for every $\omega\in\mathcal{T}^0_s(d+1)$, $\varphi(\omega)\in\mathcal{T}^0_s(d)$, but the opposite is not necessarily true. Therefore,
\begin{align*}
 {\sf P}(\mathcal{B}_s(d)) \,\, 
 & = \,\, \sum_{\omega^*\in\mathcal{T}_s^0(d)}
 {\sf P}\left(\tilde R_{\leq 2s-2}(d)=\omega^*,\,\,\tilde R_{2s-1}(d)\text{ terminates}\right)\\
 &= \,\, (1-(1-\varepsilon)^2)\sum_{\omega^*\in\mathcal{T}_s^0(d)}
 {\sf P}\left(\tilde R_{\leq 2s-2}(d)=\omega^*\right)\\
 &\geq \,\, (1-(1-\varepsilon)^2)\sum_{\omega^*\in\mathcal{T}_s^0(d)}
 {\sf P}\left[\tilde R_{\leq 2s-2}(d+1)\in\varphi^{-1}(\omega^*)\right]\\
 &\geq \,\, (1-(1-\varepsilon)^2)\sum_{\omega\in\mathcal{T}_s^0(d+1)}
 {\sf P}\left(\tilde R_{\leq 2s-2}(d+1)=\omega\right) \,\,= \,\,{\sf P}(\mathcal{B}_s(d+1))
\end{align*}
implying (\ref{eq:mu_monotone}).
\end{proof}

\section{Insufficiency of $p_d$: complete bipartite graphs}
\label{sc:bip}

\begin{theorem}
Let $u_1,u_2$ be vertices of parts of sizes $n_1$ and $n_2$ of  $G_1=K_{n_1,m_1}$ and $G_2=K_{n_2,m_2}$ respectively and $(n_1,m_1)\neq(n_2,m_2)$. Then 
\begin{enumerate}
\item if $n_1=n_2=1$, then $(G_1,u_1)$ and $(G_2,u_2)$ are not NVM-distinguishable;
\item if $n_1=n_2>1$ or $m_1=m_2$, then $(G_1,u_1)$ and $(G_2,u_2)$ are distinguishable by $S^{(2)}$;
\item if $(1-\varepsilon)^2$ is not a root of a polynomial of degree 3 with integer coefficients, and $\max\{n_1,n_2\}>1$, then $(G_1,u_1)$ and $(G_2,u_2)$ are distinguishable by $S^{(2)}$;
\item for every $k$ and every $\varepsilon'\in(0,1)$, there exist $\varepsilon\in(0,\varepsilon')$, distinct $n_1,n_2>k$ and distinct $m_1,m_2>k$ such that $(G_1,u_1)$ and $(G_2,u_2)$ are NVM-distinguishable but $p_d(G_1,u_1)=p_d(G_2,u_2)$ for any $d\in\mathbb{N}$.
\end{enumerate}
\label{th:bip}
\end{theorem}

\proof As usual, to prove parts 2 and 3 of the theorem, we will apply Theorem~\ref{th:NVM_main}. To derive NVM-distinguishability in the last part of the theorem, we are going to apply Theorem~\ref{th:NVM_aux}.

 Let $G=K_{n,m}$ have parts $V_1$ of size $n$ and $V_2$ of size $m$. Let $u$ be a vertex of $V_1$. Let us find $p_2$. Let us fix positive integers $1\leq k\leq t$ and a sequence of vertices $u_1,\ldots,u_k=u$ of $G$ such that, for every $i\in[k-1]$, $u_{i+1}$ and $u_i$ belong to different parts. Let $\mathcal{B}_t(u_1,\ldots,u_{k-1})$ be the event saying that egdes $((u_{i+1},t-k+i+1),(u_i,t-k+i)),$ $i\in[k-1]$, belong to $\mathcal{D}_G$. Let $\mathcal{A}_{t,k}(u_1,u)$ be the event saying that $\pi_t$ meets $\pi_{t+2}$ for the first time at point $(u_1,t-k+1)$. Note that every path of length $k+1$ from $(u,t+2)$ to $(u_1,t-k+1)$ that does not meet points $(u_{i+1},t-k+i+1)$, $i\in[k-1]$, has probability $\frac{(1-\varepsilon)^{k+1}}{n^{\lfloor(k+1)/2\rfloor}m^{\lceil(k+1)/2\rceil}}$, and there are $(n-1)^{\lceil(k+1)/2\rceil-1}(m-1)^{\lfloor(k+1)/2\rfloor-1}m$ such paths. Therefore,
$$
 {\sf P}\biggl(\mathcal{A}_{t,k}(u_1,u)\,\biggl|\,\mathcal{B}_t(u_1,\ldots,u_{k-1})\biggr)=(1-\varepsilon)^{k+1}\frac{(n-1)^{\lceil(k+1)/2\rceil-1}(m-1)^{\lfloor(k+1)/2\rfloor-1}}{n^{\lfloor(k+1)/2\rfloor}m^{\lceil(k+1)/2\rceil-1}}.
$$
Then
\begin{align*}
 p_2 \,\, 
 & = \,\, \lim_{t\to\infty}\sum_{k=1}^{t}\sum_{u_1,\ldots,u_{k-1}}(1-\varepsilon)^{k+1}\frac{(n-1)^{\lceil\frac{k+1}{2}\rceil-1}(m-1)^{\lfloor\frac{k+1}{2}\rfloor-1}}{n^{\lfloor\frac{k+1}{2}\rfloor}m^{\lceil\frac{k+1}{2}\rceil-1}}{\sf P}(\mathcal{B}_t(u_1,\ldots,u_{k-1}))\\
 & = \,\, \sum_{k=1}^{\infty}(1-\varepsilon)^{2k}
 \frac{(n-1)^{\lceil\frac{k+1}{2}\rceil-1}(m-1)^{\lfloor\frac{k+1}{2}\rfloor-1}}{n^{\lfloor\frac{k+1}{2}\rfloor}m^{\lceil\frac{k+1}{2}\rceil-1}}\\
 &= \,\, \frac{(1-\varepsilon)^2(m+(n-1)(1-\varepsilon)^2)}{nm-(n-1)(m-1)(1-\varepsilon)^4} 
 \,\,=: \,\, p(n,m,\varepsilon).
 \end{align*}
Note that, for every even $d\geq 4$, $p_d=(1-\varepsilon)^{d-2}p(n,m,\varepsilon)$, and, for every odd $d$, $p_d=0$.\\
 
First, if $n=1$, then $p_2=(1-\varepsilon)^2$ does not depend on $m$. Let us show that, if $n=1$, and $m_1\neq m_2$, then NVM does not distinguish between $(G_1=K_{1,m_1},u_1)$ and $(G_2=K_{1,m_2},u_2)$, where $u_j$ is the central vertex of $G_j$, $j\in\{1,2\}$. In other words, we need to show that, for every $t\in\mathbb{Z}_+$, the vectors $(X_i(G_j,u_j),i\leq t)$, $j\in\{1,2\}$, are equal in distribution. Let us prove it by induction on $t$. For $t=0$ it is straightforward by the definition of the model. Assume that, for some $t\in\mathbb{Z}_+$, the equality in distribution is true. For convenience, assume that $[m_j]$ is the set of leaves of $G_j$. For $\mathbf{x}=(x_1,\ldots,x_t,x_{t+1}=x)\in\{0,1\}^{t+1}$ and $j\in\{1,2\}$,
\begin{align*}
 {\sf P}\left(\bigcap\limits_{i\leq t+1}X_i(G_j,u_j)=x_i\right) \,\, 
 & = \,\,
 {\sf P}(X_{t+1}(G_j,u_j)=x|X_i(G_j,u_j)=x_i,\,i\leq t)\\
 & \quad\quad\quad\quad\quad\quad\quad\quad\quad\quad 
 \times{\sf P}(X_i(G_j,u_j)=x_i,\,i\leq t)\\
 & = \,\, \left(\frac{\varepsilon}{2}+\sum_{v=1}^{m_j}
 \frac{(1-\varepsilon){\sf P}(X_t(G_j,v)=x|X_i(G_j,u_j)=x_i,\,i\leq t)}{m_j}\right)\\
 & \quad\quad\quad\quad\quad\quad\quad\quad\quad\quad 
 \times{\sf P}(X_i(G_j,u_j)=x_i,\,i\leq t).
\end{align*}
Clearly, for every $v\in[m_j]$,
$$
 {\sf P}(X_t(G_j,v)=x_{t+1}|X_i(G_j,u_j)=x_i,\,i\leq t)=\frac{\varepsilon}{2}+(1-\varepsilon)I(x_{t+1}=x_{t-1}).
$$
It immediately follows that ${\sf P}(X_i(G_1,u_1)=x_i,\,i\leq t+1)={\sf P}(X_i(G_2,u_2)=x_i,\,i\leq t+1)$.\\

Second, assume that $n_1=n_2=n>1$ and $m_1\neq m_2$. Note that
\begin{align*}
 \frac{\partial p(n,m,\varepsilon)}{\partial m} \,\, 
 & = \,\, (1-\varepsilon)^2 \,\, \times \\
 &\times\,\, \frac{nm-(1-\varepsilon)^4(n-1)(m-1)-(n-(n-1)(1-\varepsilon)^4)(m+(n-1)(1-\varepsilon)^2)}{(nm-(n-1)(m-1)(1-\varepsilon)^4)^2}\\
 & = \,\, -(1-\varepsilon)^4\frac{(n-1)((n-1)(1-(1-\varepsilon)^4)+1-(1-\varepsilon)^2)}{(nm-(n-1)(m-1)(1-\varepsilon)^4)^2} \,\, < \,\, 0
 \end{align*}
 for all $m$. Therefore, the values of $p_2$ are all distinct.
 
In the same way, if $m_1=m_2=m$, then the statement immediately follows from the fact that  
$$
\frac{\partial p(n,m,\varepsilon)}{\partial n}=-(1-\varepsilon)^2\frac{m(1-(1-\varepsilon)^2)(m(1+(1-\varepsilon)^2)-(1-\varepsilon)^2)}{(nm-(n-1)(m-1)(1-\varepsilon)^4)^2}<0
$$ 
as well.\\
 
Now, let us show that, for every $k$ and every $\varepsilon'>0$, there exist $\varepsilon\in(0,\varepsilon')$, distinct $n_1,n_2>k$ and distinct $m_1,m_2>k$ such that $p_2(G_1,u_1)=p_2(G_2,u_2)$ for a vertex $u_1$ of the $n_1$-part of $K_{n_1,m_1}$ and a vertex $u_2$ of the $n_2$-part of $K_{n_2,m_2}$. 

Fix $\varepsilon'>0$. Choose a rational $\alpha\in(0,\varepsilon')$. Let us choose integer $n_1>n_2>m_2>m_1>k$ such that $n_2=(1-\alpha)n_1$, $m_2=(1-\alpha)n_2$, $m_1=(1-\alpha)m_2$. Note that $m_1n_1=m_2n_2$ and $\frac{m_2-m_1}{n_1-n_2}=(1-\alpha)^2$. It immediately implies that $p(n_1,m_1,\alpha)\sim p(n_2,m_2,\alpha)$ ($m_1\to\infty$). Then, for $\varepsilon_1\in(0,\alpha)$, $\varepsilon_2\in(\alpha,\varepsilon')$ and $n_1$ large enough, we get $p(n_1,m_1,\varepsilon_1)> p(n_2,m_2,\varepsilon_1)$ and $p(n_1,m_1,\varepsilon_2)< p(n_2,m_2,\varepsilon_2)$. Since $p(n_2,m_2,\varepsilon)-p(n_1,m_1,\varepsilon)$ is continuous in $\varepsilon$, we get that there exists an $\varepsilon\in(\varepsilon_1,\varepsilon_2)$ such that $p(n_1,m_1,\varepsilon)=p(n_2,m_2,\varepsilon)$ as needed.

Also it is obvious that, if $p(n_1,m_1,\varepsilon)=p(n_2,m_2,\varepsilon)$ for some $n_1\neq n_2$, $m_1\neq m_2$, then $(1-\varepsilon)^2$ is a root of a polynomial of degree 3 with integer coefficients. Indeed, $p(n_1,m_1,\varepsilon)=p(n_2,m_2,\varepsilon)$ if and only if $(1-\varepsilon)^2$ is a root of $Ax^3+Bx^2+Cx+D$, where
$$
 A=(n_1-1)(n_2-1)(m_1-m_2),\quad
 B=m_2(n_1-1)(m_1-1)-m_1(n_2-1)(m_2-1),
$$
$$ 
 C=(n_1-1)m_2n_2-(n_2-1)m_1n_1,\quad 
 D=m_1m_2n_2-m_2m_1n_1.
$$
If $\max\{n_1,n_2\}>1$, then this polynomial becomes trivial only when $n_1=n_2$ and $m_1=m_2$.\\

Finally, assume that $m_1=(1-\alpha)m_2=(1-\alpha)^2n_2=(1-\alpha)^3n_1$, where $\alpha\in(0,\varepsilon')$ is a rational number as above, and $n_1$ is sufficiently large. We are going to distinguish $(G_1=K_{n_1,m_1},u_1)$ and $(G_2=K_{n_2,m_2},u_2)$ via $S^{(2,4,6)}$ for infinitely many tuples $(m_1,n_1,m_2,n_2,\varepsilon)$ that violate the $S^{(2)}$-distinguishibility. Due to Theorem~\ref{th:NVM_aux}, it finishes the proof of Theorem~\ref{th:bip}.

Fix $u\in V_1$, where $V_1$ is the part of size $n$ of $G=K_{n,m}$, and $n,m$ are some positive integers. It is easy to see that $\tilde p_{2,4,6}=(p_2)^2(1+o(1))$, $\tilde p_{4,2,6}=(p_4)^2(1+o(1))$, $\tilde p_{6,2,4}=p_6p_2(1+o(1))$. Indeed, 
$$
\tilde p_{2,4,6}={\sf P}(\pi_{t+4}\cap\pi_{t+6}\neq\varnothing,\pi_{t+4}\cap\pi_{t+2}=\varnothing|\pi_t\cap\pi_{t+2}\neq\varnothing){\sf P}(\pi_t\cap\pi_{t+2}\neq\varnothing),
$$
and
$$
 p_2\left(1-O\left(\frac{1}{n}+\frac{1}{m}\right)\right)\leq{\sf P}(\pi_{t+4}\cap\pi_{t+6}\neq\varnothing,\pi_{t+4}\cap\pi_{t+2}=\varnothing|\pi_t\cap\pi_{t+2}\neq\varnothing)\leq p_2.
$$
The same arguments work for $\tilde p_{4,2,6}$ and $\tilde p_{6,2,4}$. It means that $\tilde p_{2,4,6}$, $\tilde p_{4,2,6}$ and $\tilde p_{6,2,4}$ are asymptotically equal on $(G_1,u_1)$ and $(G_2,u_2)$. We may only hope that the values of $p_{2,4,6}$ are asymptotically different (and this is indeed the case). However, this is not sufficient since $p_{2,4,6}=O(n^{-3}+n^{-2}m^{-1}+n^{-1}m^{-2}+m^{-3})$ while $\tilde p_{2,4,6}=O(n^{-2}+n^{-1}m^{-1}+m^{-2})$, and the same holds for $\tilde p_{4,2,6}$ and $\tilde p_{6,2,4}$. So we should treat $\tilde p_{2,4,6}$, $\tilde p_{4,2,6}$ and $\tilde p_{6,2,4}$  more thoroughly.\\

By conditioning on the length (which is denoted by $s$ in the formula below) of the part of the path $\pi_t$ between its earliest intersection with $\pi_{t+2}$ and $(u,t)$ and applying the law of total probabilities, we get
\begin{align*}
 \tilde p_{2,4,6} \,\, 
 &= \,\,\sum_{s\text{ even}}\frac{(1-\varepsilon)^{2s+2}}{n}\left(\frac{(n-1)(m-1)}{nm}\right)^{s/2}\times \\
 &\quad\quad\times 
 \biggl(F(s)+(1+\varepsilon)^{2s+10}F_1^{\text{even}}(s)+
  \sum_{\ell=1}^{\infty}(1+\varepsilon)^{2s+10+2\ell}F_2^{\text{even}}(s,\ell)\biggr)\\
& + \,\, \sum_{s\text{ odd}}\frac{(1-\varepsilon)^{2s+2}}{m}\left(\frac{(n-1)(m-1)}{nm}\right)^{s/2}\times \\
 &\quad\quad\times 
\biggl(F(s)+(1+\varepsilon)^{2s+10} F_1^{\text{odd}}(s)+
 \sum_{\ell=1}^{\infty}(1+\varepsilon)^{2s+10+2\ell}F_2^{\text{odd}}(s,\ell) \biggr),
\end{align*}
where 
\begin{align*}
F(s) \,\,
& = \,\, \frac{(1-\varepsilon)^2}{n}+\frac{(1-\varepsilon)^4}{m}\frac{n-1}{n}+
 \frac{n-1}{n}\frac{(1-\varepsilon)^6}{n}\frac{(n-1)(m-1)}{nm}+\\
&\quad\,\,\, \frac{(n-1)(m-1)}{nm}\frac{(1+\varepsilon)^8}{m}\frac{(n-1)(m-1)(n-2)}{n^2m}+I(s\geq 1)\sum_{\ell=0}^{s-1}(1-\varepsilon)^{10+2\ell}\tilde F(s,\ell),
\end{align*}
$$
\tilde F(s,\ell)= \frac{(m-1)^2(n-1)^2(n-2)^{2+\lfloor\ell/2\rfloor}(m-2)^{1+\lceil\ell/2\rceil}(n-3)^{\lceil\ell/2\rceil}(m-3)^{\lfloor\ell/2\rfloor}}{m^{3+2\lfloor(\ell+1)/2\rfloor}n^{5+2\lfloor\ell/2\rfloor}};
$$
\begin{align*}
 F_1^{\text{even}}(s) \,\,
 &=\,\,\frac{(n-1)^3(m-1)^2(n-2)^{1+s/2}(m-2)^{\frac{s+2}{2}}(n-3)^{s/2}(m-3)^{s/2}}{n^{5+s}m^{3+s}},\\
 F_1^{\text{odd}}(s) \,\,
 &=\,\, \frac{(n-1)^2(m-1)^3(n-2)^{\frac{s+3}{2}}(m-2)^{\frac{s+1}{2}}(n-3)^{\frac{s+1}{2}}(m-3)^{\frac{s-1}{2}}}{n^{4+s}m^{4+s}};
\end{align*}
\begin{align*}
F_2^{\text{even}}(s,\ell) \,\,
&=\,\,\frac{(n-1)^{3+\lfloor\frac{\ell-1}{2}\rfloor}(m-1)^{2+\lceil\frac{\ell-1}{2}\rceil}(n-2)^{2+\frac{s}{2}}(m-2)^{1+\frac{s}{2}}(n-3)^{\frac{s}{2}}(m-3)^{\frac{s}{2}}}{n^{4+s+\lceil\frac{\ell+1}{2}\rceil}m^{3+s+\lfloor\frac{\ell+1}{2}\rfloor}},\\
F_2^{\text{odd}}(s,\ell) \,\,
&=\,\,\frac{(n-1)^{2+\lceil\frac{\ell-1}{2}\rceil}(m-1)^{2+\lfloor\frac{\ell-1}{2}\rfloor}(n-2)^{\frac{s+3}{2}}(m-2)^{\frac{s+3}{2}}(n-3)^{\frac{s+1}{2}}(m-3)^{\frac{s-1}{2}}}{n^{4+s+\lceil\frac{\ell}{2}\rceil}m^{4+s+\lfloor\frac{\ell}{2}\rfloor}}.
\end{align*}

Omitting tedious computations, we arrive at the asymptotics of $\tilde p_{2,4,6}-p_2^2:$
$$
p_2^2-\tilde p_{2,4,6}=\Pi(n,m,\varepsilon)(1+o(1))+Q(n,m,\varepsilon),
$$
where
\begin{align*}
\Pi(n,m,\varepsilon) \,\,
&= \,\,
\frac{4m(1-\varepsilon)^{14}}{m^3n^3(1-(1-\varepsilon)^4)(1-(1-\varepsilon)^8)^2}\biggl(2(m+n)(m^2(1-\varepsilon)^2+n^2(1-\varepsilon)^6)+\\
&\quad\quad\quad\quad\quad\quad\quad\quad\quad\quad
+(1+(1-\varepsilon)^{4})(m^2n+2n^2m)+nm^2\biggr)\\
&+ \,\, \frac{(1-\varepsilon)^{14}}{m^3n^3(1-(1-\varepsilon)^4)(1-(1-\varepsilon)^8)}\biggl((1-\varepsilon)^{2}(4m^2n+8m^3)+\\
 &\quad\quad
 +(1-\varepsilon)^{4}(16m^2n+11n^2m)
+(1-\varepsilon)^6(13mn^2+8n^3)+12m^2n+5n^2m\biggr)\\
 & + \,\, \frac{3(1-\varepsilon)^{12}(nm^2+m^3+mn^2+n^3)}{m^3n^3(1-(1-\varepsilon)^8)^2}
 \end{align*}
and $Q$ is a rational function of $n,m,\varepsilon$ with a polynomial $q(n,m,\varepsilon)=q_1(n,m)q_2(\varepsilon)$ in the denominator such that the smallest non-trivial power in $q_2$ is 1.

In a similar way (we omit the routine and direct computations of the exact values of $\tilde p_{4,2,6}$ and $\tilde p_{6,2,4}$ --- they can be found in Appendix in~\cite{Appendix}), we get
$$
(p_4^2-\tilde p_{4,2,6})+(p_4p_6-\tilde p_{6,2,4})=\tilde\Pi(n,m,\varepsilon)(1+o(1))+\tilde Q(n,m,\varepsilon),
$$
where
\begin{align*}
\tilde\Pi(n,m,\varepsilon) \,\,
& = \,\,\frac{16(1-\varepsilon)^{14}(m+n)}{m^3n^3(1-(1-\varepsilon)^4)(1-(1-\varepsilon)^8)^2}\biggl(m^2(1-\varepsilon)^2+\\
&\quad\quad\quad\quad\quad\quad\quad\quad\quad\quad\quad\quad\quad
+nm(1+(1-\varepsilon)^4)+n^2(1-\varepsilon)^6\biggr)\\ 
& + \,\,\frac{4(1-\varepsilon)^{14}}{m^3n^3(1-(1-\varepsilon)^4)(1-(1-\varepsilon)^8)}\biggl((7m+5n)(nm(1-\varepsilon)^4+n^2(1-\varepsilon)^6)+\\
&\quad\quad\quad\quad\quad\quad\quad\quad\quad\quad\quad\quad\quad
+(5m+3n)(m^2(1-\varepsilon)^2+nm)\biggr)\\
& + \,\, \frac{4(1-\varepsilon)^{12}(n+m)(nm+2n^2(1-\varepsilon)^4+m^2) } {m^3n^3(1-(1-\varepsilon)^8)^2}
 \end{align*}
and $\tilde Q$ is a rational function of $n,m,\varepsilon$ with a polynomial $\tilde q(n,m,\varepsilon)=\tilde q_1(n,m)\tilde q_2(\varepsilon)$ in the denominator such that the smallest non-trivial power in $q_2$ is 1.\\

Asymptotics of $p_{2,4,6}$ (as $m\to\infty$ and $n\to\infty$) is computed in a direct routine way (see the computations in~\cite[Appendix]{Appendix}). We get that $p_{2,4,6}=[\Sigma_0+\Sigma_1+\Sigma^e_0+\Sigma^e_1+\Sigma^e_2+\Sigma^o_0+\Sigma^o_1+\Sigma^o_2](1+o(1))$, where the values of all $\Sigma$'s are given in~\cite[Appendix]{Appendix}. Note that, for $n=\Theta(m)$, we have $n^3[\Sigma_0+\Sigma_1+\Sigma^e_0+\Sigma^e_1+\Sigma^e_2+\Sigma^o_0+\Sigma^o_1+\Sigma^o_2]=\Theta(1)$.

Note that an $\varepsilon$ such that $S^{(2)}$ does not distinguish between $(G_1,u_1)$ and $(G_2,u_2)$ can be chosen arbitrarily close to $\alpha$. More formally, for every $\delta>0$, there exists $N$ such that, for all valid $m_1>N$ (i.e. $m_1/(1-\alpha)^3$ should be an integer), there exists $\varepsilon(m_1)\in(\alpha-\delta,\alpha+\delta)$ such that $S^{(2)}$ does not distinguish between $(G_1,u_1)$ and $(G_2,u_2)$. Assume that, for all such (or, equivalently, for all but finitely many) $m_1$ and $\varepsilon(m_1)$,  $S^{(2,4,6)}$ does not distinguish between these graphs as well. From the above, it implies that 
\begin{align}
 \sum_{j=0,1}\Sigma_j(n_1,(1-\alpha)^3n_1,\alpha) 
 &+\sum_{j=0,1,2}\sum_{*=o,e}\Sigma^*_j(n_1,(1-\alpha)^3n_1,\alpha)-\notag\\
 &-(\Pi+\tilde\Pi+Q+\tilde Q)(n_1,(1-\alpha)^3n_1,\alpha)\sim\notag\\
 \sum_{j=0,1}\Sigma_j((1-\alpha)n_1,(1-\alpha)^2n_1,\alpha)
 &+
  \sum_{j=0,1,2}\sum_{*=o,e}\Sigma^*_j((1-\alpha)n_1,(1-\alpha)^2n_1,\alpha)-
  \label{eq:sigmas_asympt_equality}\\
   &-(\Pi+\tilde\Pi+Q+\tilde Q)((1-\alpha)n_1,(1-\alpha)^2n_1,\alpha),\notag
  \label{eq:sigmas_asympt_equality}
\end{align}
where all $\Sigma$'s are considered as functions of $n,m$ and $\varepsilon$. After multiplying both parts of (\ref{eq:sigmas_asympt_equality}) by $n_1^3(1-(1-\alpha)^4)(1-(1-\alpha)^6)(1-(1-\alpha)^8)^2$, they become polynomials of $\alpha$ %of degree {\bf [30?]} 
that do not depend on $n_1$. Therefore  (\ref{eq:sigmas_asympt_equality}) is only possible when these polynomials are equal. But this is not the case since the coefficients in front of $\alpha^2$ of the left and right polynomials are equal 2384 and -2904 respectively (we do not need exact values of $Q$ and $\tilde Q$ since after the mentioned multiplication they become polynomials of $\alpha$ with the smallest non-trivial degree at least 3). $\Box$

\section{Remarks and further questions}
\label{sc:last}

In this paper we address the following question: given two non-isomorphic rooted graphs $(G_1,u_1)$ and $(G_2,u_2)$, is it possible to distinguish them by sequences of observations of the discrete time noisy voter model at vertices $u_1$ and $u_2$? We show that stars rooted at centers are not distinguishable, and this is the only example of non-distinguishable graphs that we know. On the other hand, the observations of the NVM determine the size of a complete graph as well as the size of a cycle; they also distinguish between non-isomorphic complete bipartite graphs (for almost all values of the noise), perfect trees, infinite graphs $\mathbb{Z}^d$. We conjecture that the stars are the only finite graphs that are not distinguishable. As an evidence we prove that almost all pairs of graphs of the same size are distinguishable for almost all $\varepsilon$. It is not hard to see that this result can be extended to pairs of random graphs of different sizes. 

\subsection{Weaker versions of the conjecture}

Some weaker versions of this conjecture can be also considered. In particular, one may try to verify it for trees, for complete multipartite graphs, for sparse binomial random graphs or for random regular graphs (both in sparse and dense regimes). On the other hand, assume that $G'$ is obtained from $G$ be deleting a single edge. It is not hard to see that $G$ and $G'$ are distinguishable for all values of $\varepsilon$. Is it true for all $\varepsilon\in(0,1)$? In Section~\ref{sc:bip}, we give infinitely many examples of graphs that can not be distinguished by statistics that we use everywhere in the paper (the number of repetitions $X_t=X_{t+d}$) for some specific $\varepsilon$. However, we show that this is not an issue of the model --- these graphs are actually distinguishable. A natural question to ask, is it true that distinguishability for some $\varepsilon$ implies distinguishability for all values of the noise? It is also interesting to prove or disprove the existence of a finite graph $G$ and an infinite graph $G'$ with bounded degrees that are indistinguishable. 

\subsection{Property testing}

Another possible application of the developed machinery is property testing~\cite{Testing}. In particular, it is not hard to see that, for infinite vertex transitive graphs, amenability (see, e.g.,~\cite{Networks}) is equivalent to the existence of a constant $c\in(0,1)$ and a constant $C>0$ such that, for all $d\in\mathbb{N}$, $0<\lim_{t\to\infty}\mathrm{cov}(X_t,X_{t+d})<C(c(1-\varepsilon))^d$. This follows directly from the Kesten's criterion~\cite[Theorem 6.7]{Networks} and may be considered as another criterion of amenability. A related problem for finite graphs is to test the expansion of the graph (see, e.g.,~\cite{Nachmias} for the recent progress in this problem). Can the local observations in NVM be used to distinguish between an $(n,d,\lambda)$-expander and a $d$-regular graph on $[n]$ which is not an $(n,d,\lambda)$-expander (or at least $\epsilon$-far from being an $(n,d,\lambda')$-expander)? It is easy to see that, for a vertex transitive graph on $n$ vertices, if the variation distance $\delta$-mixing time equals $d$, then $0<\lim_{t\to\infty}\mathrm{cov}(X_t,X_{t+d})\leq\frac{1}{1-(1-\varepsilon)^2}(1-\varepsilon)^d\left(\frac{1}{n}+\delta\right)$. It would be interesting to show an opposite implication. 

\subsection{Observing return times of a random walk}

The concept of reconstructing global properties of a graph $G$ from local observations of a random process on $G$ was introduced in~\cite{BL,BKL}. In~\cite{BL} the ``Noisy Circulator'' process on a graph $G$ embedded in an orientable surface was considered. It was shown that the genus of the surface can be determined by observing the process on the edges incident with some specific set $U\subseteq V(G)$. In the same paper, the problem of learning global graph properties from observing visits of a random walk at a single vertex was stated. This problem was further studied in~\cite{BKL}. It was shown that, by observing the return time sequence of a random walk at a connected graph $G$, it is possible to determine an eigenvalue $\lambda$ of $G$ under some additional restrictions on $G$. However, the multiplicity of $\lambda$ is not necessarily reconstructible. As an example, they considered two trees presented in Figure~\ref{fg:fg} and showed that the distribution of the return time to the root is the same in both trees, but the eigenvalues have different multiplicities. It is not hard to show that NVM (in particular, $S^{(2)}$) distinguishes between these 2 graphs. In particular, if $\varepsilon=0.01$, then for the tree on the left $p_2\geq 0.926$, while for the tree on the right $p_2\leq 0.917$. A natural question to ask, does NVM distinguish between any pair of non-isomorphic graphs that are not stars and have equal eigenvalues?

%\begin{wrapfigure}[9]{l}{260\unitlength}
\begin{figure}
  \includegraphics[width=330.0pt,height=130.0pt]{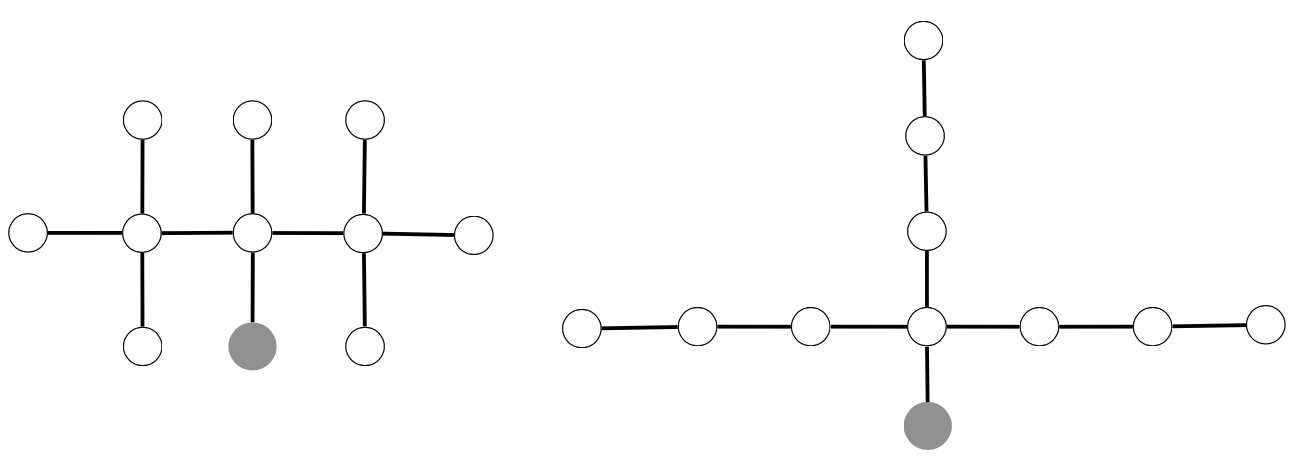}
%, bb=0 20 140 200]{spectral_trees.png}
  \caption{two trees with the same return times but different spectra}
  \label{fg:fg}
  \end{figure}
%\end{wrapfigure}

\subsection{Other models}

Although the voter model without noise ($\varepsilon=0$) reaches a consensus at some moment $t_0$, it may happen that some pairs of graphs can be distinguished (or some graph properties can be learned) with large enough probability in time less than $t_0$. Then, possibly, something can be learned even from local observations of the original voter model. Moreover, the reason of why stars are not distinguishable is that the decisions for choosing a new opinion are made by all vertices simultaneously which is not the case in the continuous time model. It would be interesting to know, are there two graphs that are not distinguishable by the continuous time noisy voter model at least for some $\varepsilon\in(0,1)$.

Let us conclude with asking the same question for other random processes on graphs modelling interactions between particles. What kind of global properties can be distilled from local observations of Glauber dynamics for Ising model (a standard model in statistical physics for the temporal evolution of spin systems) or, say, majority dynamics (which was of interest in biophysics~\cite{McCulloch} and psychology~\cite{Cartwright})?

\section*{Appendix}

Here we give the missing computations of $\tilde p_{4,2,6}$, $\tilde p_{6,2,4}$ and $p_{2,4,6}$ from the proof of the last part of Theorem~\ref{th:bip} in Section~\ref{sc:bip}.

  \subsection*{A non-intersecting pair of intersecting pairs of histories: $\tilde p_{4,2,6}$ and $\tilde p_{6,2,4}$}

By conditioning on the length (which is denoted by $s$ in the formula below) of the part of the path $\pi_t$ between its earliest intersection with $\pi_{t+4}$ and $(u,t)$ and applying the law of total probabilities, we get
\begin{multline*}
 \tilde p_{4,2,6}=\sum_{s\text{ even}}\frac{(1-\varepsilon)^{2s+4}}{n}\left(\frac{(n-1)(m-1)}{nm}\right)^{s/2}\left(1-\frac{1}{n}\right)\biggl(\tilde F_1(n,m,\varepsilon)+\\
 (1-\varepsilon)^{2s+8}\frac{(m-1)^2(n-1)^2(n-2)^{\frac{s}{2}+1}(m-2)^{\frac{s}{2}+1}(n-3)^{\frac{s}{2}}(m-3)^{\frac{s}{2}}}{n^{s+4}m^{s+3}}+
 \sum_{\ell=1}^{\infty}(1-\varepsilon)^{2s+2\ell+8}\times\\
 \frac{(m-1)^2(n-1)^2(n-2)^{\frac{s}{2}+2}(m-2)^{\frac{s}{2}+1}(n-3)^{\frac{s}{2}}(m-3)^{\frac{s}{2}}(m-1)^{\lceil\frac{\ell-1}{2}\rceil}(n-1)^{\lfloor\frac{\ell-1}{2}\rfloor}}{n^{s+\lfloor\frac{\ell}{2}\rfloor+4}m^{s+\lceil\frac{\ell}{2}\rceil+3}}
 \biggr)+\\
 \sum_{s\text{ odd}}\frac{(1-\varepsilon)^{2s+4}}{m}\left(\frac{n-1}{n}\right)^{\frac{s+1}{2}}\left(\frac{m-1)}{m}\right)^{\frac{s-1}{2}}\left(1-\frac{1}{n}\right)\biggl(\tilde F_1(n,m,\varepsilon)+\\
 (1-\varepsilon)^{2s+8}\frac{(m-1)^3(n-1)(n-2)^{\frac{s+3}{2}}(m-2)^{\frac{s+1}{2}}(n-3)^{\frac{s+1}{2}}(m-3)^{\frac{s-1}{2}}}{n^{s+3}m^{s+4}}+
 \sum_{\ell=1}^{\infty}(1-\varepsilon)^{2s+2\ell+8}\times\\
 \frac{(m-1)^3(n-1)(n-2)^{\frac{s+3}{2}}(m-2)^{\frac{s+3}{2}}(n-3)^{\frac{s+1}{2}}(m-3)^{\frac{s-1}{2}}(n-1)^{\lceil\frac{\ell-1}{2}\rceil}(m-1)^{\lfloor\frac{\ell-1}{2}\rfloor}}{n^{s+\lceil\frac{\ell}{2}\rceil+3}m^{s+\lfloor\frac{\ell}{2}\rfloor+4}}
 \biggr)
\end{multline*}
where
\begin{align*}
 \tilde F_1(n,m,\varepsilon)=
 &\Upsilon_1(n,m,\varepsilon)+\\
 &\sum_{\ell=0}^{s-1}(1-\varepsilon)^{2\ell+8}\frac{(n-1)(m-1)^2(n-2)^{\lceil\frac{\ell+3}{2}\rceil}(m-2)^{\lfloor\frac{\ell+3}{2}\rfloor}(n-3)^{\lceil\frac{\ell}{2}\rceil}(m-3)^{\lfloor\frac{\ell}{2}\rfloor}}{n^{2\lceil\frac{\ell+3}{2}\rceil}m^{2\lfloor\frac{\ell+3}{2}\rfloor}}
\end{align*}
and $\Upsilon_1(n,m,\varepsilon)$ is a rational function that does not depend on $s$ with a polynomial in the denominator that does not depend on $\varepsilon$.

In the same way,
\begin{multline*}
\tilde p_{6,2,4}=\sum_{s\text{ even}}\frac{(1-\varepsilon)^{2s+6}}{n}
 \left(\frac{(n-1)(m-1)}{nm}\right)^{s/2}
 \left(1-\frac{1}{n}\right)^2
 \biggl(\tilde F_2(n,m,\varepsilon)+\\
 (1-\varepsilon)^{2s+6}\frac{(m-1)^2(n-1)(n-2)^{\frac{s}{2}+1}(m-2)^{\frac{s}{2}+1}(n-3)^{\frac{s}{2}}(m-3)^{\frac{s}{2}}}{n^{s+3}m^{s+3}}+
 \sum_{\ell=1}^{\infty}(1-\varepsilon)^{2s+2\ell+6}\times\\
 \frac{(m-1)^2(n-1)(n-2)^{\frac{s}{2}+2}(m-2)^{\frac{s}{2}+1}(n-3)^{\frac{s}{2}}(m-3)^{\frac{s}{2}}(m-1)^{\lceil\frac{\ell-1}{2}\rceil}(n-1)^{\lfloor\frac{\ell-1}{2}\rfloor}}{n^{s+\lfloor\frac{\ell}{2}\rfloor+3}m^{s+\lceil\frac{\ell}{2}\rceil+3}}
 \biggr)+\\
 \sum_{s\text{ odd}}\frac{(1-\varepsilon)^{2s+6}}{m}
 \left(\frac{n-1}{n}\right)^{\frac{s+1}{2}}\left(\frac{m-1)}{m}\right)^{\frac{s-1}{2}}
 \left(1-\frac{1}{n}\right)^2
 \biggl(\tilde F_2(n,m,\varepsilon)+\\
 (1-\varepsilon)^{2s+6}\frac{(m-1)^3(n-2)^{\frac{s+3}{2}}(m-2)^{\frac{s+1}{2}}(n-3)^{\frac{s+1}{2}}(m-3)^{\frac{s-1}{2}}}{n^{s+2}m^{s+4}}+\\
 \sum_{\ell=1}^{\infty}(1-\varepsilon)^{2s+2\ell+6}\frac{(m-1)^3(n-2)^{\frac{s+3}{2}}(m-2)^{\frac{s+3}{2}}(n-3)^{\frac{s+1}{2}}(m-3)^{\frac{s-1}{2}}(n-1)^{\lceil\frac{\ell-1}{2}\rceil}(m-1)^{\lfloor\frac{\ell-1}{2}\rfloor}}{n^{s+\lceil\frac{\ell}{2}\rceil+2}m^{s+\lfloor\frac{\ell}{2}\rfloor+4}}
 \biggr)
\end{multline*}
where
$$
 \tilde F_2(n,m,\varepsilon)=
 \Upsilon_2(n,m,\varepsilon)+\sum_{\ell=0}^{s-1}(1-\varepsilon)^{2\ell+6}\frac{(m-1)^2(n-2)^{\lceil\frac{\ell+3}{2}\rceil}(m-2)^{\lfloor\frac{\ell+3}{2}\rfloor}(n-3)^{\lceil\frac{\ell}{2}\rceil}(m-3)^{\lfloor\frac{\ell}{2}\rfloor}}{n^{2\lfloor\frac{\ell}{2}\rfloor+3}m^{2\lceil\frac{\ell}{2}\rceil+3}}
$$
and $\Upsilon_2(n,m,\varepsilon)$ is a rational function that does not depend on $s$ with a polynomial in the denominator that does not depend on $\varepsilon$.

 \subsection*{Four intersecting histories}
 
Let us find an asymptotics of $p_{2,2,2}$. It consists of the following parts.\\

1) Here we find an asymptotics of the probability that $\pi_t$ meets $\pi_{t+2},\pi_{t+4},\pi_{t+6}$, and $\pi_{t+2}$ has length $2$ (clearly, it starts and finishes at $u$):
\begin{multline*} 
\frac{(1-\varepsilon)^2}{n}\biggl(\frac{(1-\varepsilon)^2}{n}\left[\sum_{j=0,2,4}\frac{(1-\varepsilon)^{2+j}}{n}+\sum_{j=0,2}\frac{(1-\varepsilon)^{3+j}}{m}\right]+\\
 \frac{(1-\varepsilon)^3}{m}\left[\sum_{j=0,4}\frac{(1-\varepsilon)^{2+j}}{n}
+\sum_{j=0,2}\frac{(1-\varepsilon)^{3+j}}{m}+  \frac{2(1-\varepsilon)^4}{n}\right]+\\
 \frac{(1-\varepsilon)^4}{n}\left[\sum_{j=0,4}\frac{(1-\varepsilon)^{2+j}}{n}
+\frac{(1-\varepsilon)^3}{m}+  \frac{2(1-\varepsilon)^4}{n}+\frac{2(1-\varepsilon)^5}{m}\right]\biggr)=
\end{multline*}
\begin{multline*}
\frac{(1-\varepsilon)^6}{n^3m^2}\biggl[m^2(1+2(1-\varepsilon)^2+3(1-\varepsilon)^4+(1-\varepsilon)^6)+\\
nm(2(1-\varepsilon)+4(1-\varepsilon)^3+3(1-\varepsilon)^5)+n^2((1-\varepsilon)^2+(1-\varepsilon)^4)\biggr]=:\Sigma_0.
\end{multline*}

%%%%%%%%%%%%%%%%%%%%%%%%%

2) Second, we find an asymptotics of the probability that $\pi_t$ meets $\pi_{t+2},\pi_{t+4},\pi_{t+6}$, and $\pi_{t+2}$ has length $3$:
\begin{multline*}
 \frac{(1-\varepsilon)^4}{m} \biggl(\frac{(1-\varepsilon)^2}{n}
 \left[\sum_{j=0,2}\frac{(1-\varepsilon)^{2+j}}{n}+\sum_{j=0,2,4}\frac{(1-\varepsilon)^{3+j}}{m}+\frac{2(1-\varepsilon)^6}{n}\right]+\\
 \frac{(1-\varepsilon)^3}{m} \left[\frac{(1-\varepsilon)^2}{n}
+\sum_{j=0,2,4}\frac{(1-\varepsilon)^{3+j}}{m}+  \sum_{j=2,4}\frac{2(1-\varepsilon)^{2+j}}{n}\right]+\\
 \frac{2(1-\varepsilon)^4}{n} \left[\frac{(1-\varepsilon)^2}{n}+\sum_{j=0,4}\frac{(1-\varepsilon)^{3+j}}{m}+ 
 \sum_{j=2,4}\frac{2(1-\varepsilon)^{2+j}}{n}+\frac{2(1-\varepsilon)^5}{m}\right]+\\
 \frac{(1-\varepsilon)^5}{m} \left[\frac{(1-\varepsilon)^2}{n}+\sum_{j=0,4}\frac{(1-\varepsilon)^{3+j}}{m}+ 
 \frac{2(1-\varepsilon)^{4}}{n}+\frac{2(1-\varepsilon)^5}{m}+\frac{3(1-\varepsilon)^{6}}{n}\right] \biggr)=\\
 \frac{(1-\varepsilon)^8}{m^3n^2}\biggl[m^2(1+3(1-\varepsilon)^2+6(1-\varepsilon)^4+4(1-\varepsilon)^6)+nm(2(1-\varepsilon)+6(1-\varepsilon)^3+9(1-\varepsilon)^5+5(1-\varepsilon)^7)+\\
 n^2((1-\varepsilon)^2+2(1-\varepsilon)^4+3(1-\varepsilon)^6+(1-\varepsilon)^8)\biggr]=:\Sigma_1.
\end{multline*}

3) Next, we find an asymptotics of the probability that $\pi_t$ meets $\pi_{t+2},\pi_{t+4},\pi_{t+6}$, and $\pi_t$ has an even length $s+2\geq 2$:
\begin{multline*}  
\sum_{s\text{ even}} \frac{(1-\varepsilon)^{2s+2}}{n}
  \biggl(\frac{(1-\varepsilon)^2}{n}\biggl[(1+(1-\varepsilon)^2)\left(\frac{(1-\varepsilon)^2}{n}+\frac{(1-\varepsilon)^3}{m}\right)+\\
  \sum_{r=0}^{s-2}\frac{2(1-\varepsilon)^{6+r}}{n}+\sum_{r=1}^{s-1}\frac{2(1-\varepsilon)^{6+r}}{m}+
  \frac{(1-\varepsilon)^{6+s}}{n}\biggr]+
\end{multline*}
\begin{multline*}  
  \frac{(1-\varepsilon)^3}{m}\biggl[\frac{(1-\varepsilon)^2}{n}+\frac{(1-\varepsilon)^3}{m}+\frac{2(1-\varepsilon)^4}{n}+\frac{(1-\varepsilon)^5}{m}+\\
  \sum_{r=0}^{s-2}\frac{2(1-\varepsilon)^{6+r}}{n}+\sum_{r=1}^{s-1}\frac{2(1-\varepsilon)^{6+r}}{m}+
  \frac{(1-\varepsilon)^{6+s}}{n}\biggr]+
\end{multline*}
\begin{multline*}  
  \sum_{\ell=0,2,\ldots,s-2}\frac{2(1-\varepsilon)^{4+\ell}}{n}\biggl[(1+2(1-\varepsilon)^2)\left(\frac{(1-\varepsilon)^2}{n}+\frac{(1-\varepsilon)^3}{m}\right)+I(\ell\geq 2)\sum_{r=0}^{\ell-2}\frac{3(1-\varepsilon)^{6+r}}{n}+\\
  I(\ell\geq 2)\sum_{r=1}^{\ell-1}\frac{3(1-\varepsilon)^{6+r}}{m}+
  \sum_{r=\ell}^{s-2}\frac{2(1-\varepsilon)^{6+r}}{n}+\sum_{r=\ell+1}^{s-1}\frac{2(1-\varepsilon)^{6+r}}{m}+\frac{(1-\varepsilon)^{6+s}}{n}\biggr]+
\end{multline*}
\begin{multline*}  
  \sum_{\ell=1,3,\ldots,s-1}\frac{2(1-\varepsilon)^{4+\ell}}{m}\biggl[(1+2(1-\varepsilon)^2)\left(\frac{(1-\varepsilon)^2}{n}+\frac{(1-\varepsilon)^3}{m}\right)+\sum_{r=0}^{\ell-1}\frac{3(1-\varepsilon)^{6+r}}{n}+\\
  I(\ell\geq 3)\sum_{r=1}^{\ell-2}\frac{3(1-\varepsilon)^{6+r}}{m}+
  I(\ell\leq s-3)\sum_{r=\ell+1}^{s-2}\frac{2(1-\varepsilon)^{6+r}}{n}+\sum_{r=\ell}^{s-1}\frac{2(1-\varepsilon)^{6+r}}{m}+\frac{(1-\varepsilon)^{6+s}}{n}\biggr]+
\end{multline*}
\begin{multline*}  
  \frac{(1-\varepsilon)^{4+s}}{n}\biggl[(1+2(1-\varepsilon)^2)\left(\frac{(1-\varepsilon)^2}{n}+\frac{(1-\varepsilon)^3}{m}\right)+\\
  \sum_{r=0}^{s-2}\frac{3(1-\varepsilon)^{6+r}}{n}+\sum_{r=1}^{s-1}\frac{3(1-\varepsilon)^{6+r}}{m}+
  \frac{(1-\varepsilon)^{6+s}}{n}\biggr]\biggr).
\end{multline*}

4) Finally, we find an asymptotics of the probability that $\pi_t$ meets $\pi_{t+2},\pi_{t+4},\pi_{t+6}$, and $\pi_t$ has an odd length $s+2\geq 3$:
\begin{multline*}
  \sum_{s\text{ odd}} \frac{(1-\varepsilon)^{2s+2}}{m}  \biggl(\frac{(1-\varepsilon)^2}{n}\biggl[(1+(1-\varepsilon)^2)\left(\frac{(1-\varepsilon)^2}{n}+\frac{(1-\varepsilon)^3}{m}\right)+\\
  \sum_{r=0}^{s-1}\frac{2(1-\varepsilon)^{6+r}}{n}+\sum_{r=1}^{s-2}\frac{2(1-\varepsilon)^{6+r}}{m}+
  \frac{(1-\varepsilon)^{6+s}}{m}\biggr]+
\end{multline*}
\begin{multline*}
  \frac{(1-\varepsilon)^3}{m}\biggl[\frac{(1-\varepsilon)^2}{n}+\frac{(1-\varepsilon)^3}{m}+\frac{2(1-\varepsilon)^4}{n}+\frac{(1-\varepsilon)^5}{m}+\\
  \sum_{r=0}^{s-1}\frac{2(1-\varepsilon)^{6+r}}{n}+\sum_{r=1}^{s-2}\frac{2(1-\varepsilon)^{6+r}}{m}+
  \frac{(1-\varepsilon)^{6+s}}{m}\biggr]+
\end{multline*}
\begin{multline*}
  \sum_{\ell=0,2,\ldots,s-1}\frac{2(1-\varepsilon)^{4+\ell}}{n}\biggl[(1+2(1-\varepsilon)^2)\left(\frac{(1-\varepsilon)^2}{n}+\frac{(1-\varepsilon)^3}{m}\right)+I(\ell\geq 2)\sum_{r=0}^{\ell-2}\frac{3(1-\varepsilon)^{6+r}}{n}+\\
  I(\ell\geq 2)\sum_{r=1}^{\ell-1}\frac{3(1-\varepsilon)^{6+r}}{m}+
  \sum_{r=\ell}^{s-1}\frac{2(1-\varepsilon)^{6+r}}{n}+I(\ell\leq s-3)\sum_{r=\ell+1}^{s-2}\frac{2(1-\varepsilon)^{6+r}}{m}+\frac{(1-\varepsilon)^{6+s}}{m}\biggr]+
\end{multline*}
\begin{multline*} 
 \sum_{\ell=1,3,\ldots,s-2}\frac{2(1-\varepsilon)^{4+\ell}}{m}\biggl[(1+2(1-\varepsilon)^2)\left(\frac{(1-\varepsilon)^2}{n}+\frac{(1-\varepsilon)^3}{m}\right)+\sum_{r=0}^{\ell-1}\frac{3(1-\varepsilon)^{6+r}}{n}+\\
  I(\ell\geq 3)\sum_{r=1}^{\ell-2}\frac{3(1-\varepsilon)^{6+r}}{m}+\sum_{r=\ell+1}^{s-1}\frac{2(1-\varepsilon)^{6+r}}{n}+\sum_{r=\ell}^{s-2}\frac{2(1-\varepsilon)^{6+r}}{m}+\frac{(1-\varepsilon)^{6+s}}{m}\biggr]+
\end{multline*} 
\begin{multline*}
  \frac{(1-\varepsilon)^{4+s}}{m}\biggl[(1+2(1-\varepsilon)^2)\left(\frac{(1-\varepsilon)^2}{n}+\frac{(1-\varepsilon)^3}{m}\right)+\\  
  \sum_{r=0}^{s-1}\frac{3(1-\varepsilon)^{6+r}}{n}+\sum_{r=1}^{s-2}\frac{3(1-\varepsilon)^{6+r}}{m}+
  \frac{(1-\varepsilon)^{6+s}}{m}\biggr]\biggr)
\end{multline*}

Now, let us sum up series in the third summand of the limit probability (related to an even length at least 2). It comprises the following three summands:\\

\begin{itemize}
\item no nested summations:
\begin{multline*}
  \frac{(1-\varepsilon)^{6}}{n(1-(1-\varepsilon)^4)}\times\\
  \biggl(\frac{m^2(1+(1-\varepsilon)^2)+(1-\varepsilon)(2+2(1-\varepsilon)+(1-\varepsilon)^2)mn+(1-\varepsilon)^2(1+(1-\varepsilon)^2)n^2)}{m^2n^2}\biggr)+\\
 \frac{(1-\varepsilon)^{8}(m(1+3(1-\varepsilon)^2)+n((1-\varepsilon)+3(1-\varepsilon)^3))}{n^3m(1-(1-\varepsilon)^6)}+\frac{(1-\varepsilon)^{12}}{n^3(1-(1-\varepsilon)^8)}=:\Sigma^e_0;
\end{multline*}

\item a single pair of nested summations:
\begin{multline*}
  \sum_{s\text{ even}} \frac{(1-\varepsilon)^{2s+2}}{n}\times
  \biggl(\frac{(1-\varepsilon)^2}{n}\left[\sum_{r=0}^{s-2}\frac{2(1-\varepsilon)^{6+r}}{n}+\sum_{r=1}^{s-1}\frac{2(1-\varepsilon)^{6+r}}{m}\right]+\\
  \frac{(1-\varepsilon)^3}{m}\left[\sum_{r=0}^{s-2}\frac{2(1-\varepsilon)^{6+r}}{n}+\sum_{r=1}^{s-1}\frac{2(1-\varepsilon)^{6+r}}{m}\right]+\\
  \sum_{\ell=0,2,\ldots,s-2}\frac{2(1-\varepsilon)^{4+\ell}}{n}\biggl[(1+2(1-\varepsilon)^2)\left(\frac{(1-\varepsilon)^2}{n}+\frac{(1-\varepsilon)^3}{m}\right)+\frac{(1-\varepsilon)^{6+s}}{n}\biggr]+\\
  \sum_{\ell=1,3,\ldots,s-1}\frac{2(1-\varepsilon)^{4+\ell}}{m}\biggl[(1+2(1-\varepsilon)^2)\left(\frac{(1-\varepsilon)^2}{n}+\frac{(1-\varepsilon)^3}{m}\right)+\frac{(1-\varepsilon)^{6+s}}{n}\biggr]+\\
  \frac{(1-\varepsilon)^{4+s}}{n}\left[\sum_{r=0}^{s-2}\frac{3(1-\varepsilon)^{6+r}}{n}+\sum_{r=1}^{s-1}\frac{3(1-\varepsilon)^{6+r}}{m}\right]\biggr)=
  \end{multline*}
\begin{multline*}
 \sum_{r\text{ even}}\sum_{s=r+2}^{\infty}\biggl[\frac{(1-\varepsilon)^{2s+8}}{n}\biggl(\frac{6(1-\varepsilon)^{2+r}+2(1-\varepsilon)^r}{n^2}+\frac{6(1-\varepsilon)^{3+r}+2(1-\varepsilon)^{1+r}}{mn}\biggr)+\\
 \frac{5(1-\varepsilon)^{12+r+3s}}{n^3}\biggr]+
\end{multline*}
\begin{multline*} 
 \sum_{r\text{ odd}}\sum_{s=r+1}^{\infty}\biggl[\frac{(1-\varepsilon)^{2s+8}}{n}\biggl(\frac{6(1-\varepsilon)^{2+r}+2(1-\varepsilon)^r}{nm}+\frac{6(1-\varepsilon)^{3+r}+2(1-\varepsilon)^{1+r}}{m^2}\biggr)+\\
 \frac{5(1-\varepsilon)^{12+r+3s}}{n^2m}\biggr]=
\end{multline*}
\begin{multline*}
 \sum_{r\text{ even}}\left[\frac{(1-\varepsilon)^{3r+12}(m(6(1-\varepsilon)^2+2)+n(6(1-\varepsilon)^3+2(1-\varepsilon)))}{mn^3(1-(1-\varepsilon)^4)}+\frac{5(1-\varepsilon)^{18+4r}}{n^3(1-(1-\varepsilon)^6)}\right]+\\
 \sum_{r\text{ odd}}\left[\frac{(1-\varepsilon)^{3r+10}(m(6(1-\varepsilon)^2+2)+n(6(1-\varepsilon)^3+2(1-\varepsilon)))}{m^2n^2(1-(1-\varepsilon)^4)}+\frac{5(1-\varepsilon)^{15+4r}}{n^2m(1-(1-\varepsilon)^6)}\right]=
\end{multline*}
\begin{multline*}
 \frac{(1-\varepsilon)^{12}(m(6(1-\varepsilon)^2+2)+n(6(1-\varepsilon)^3+2(1-\varepsilon)))}{mn^3(1-(1-\varepsilon)^4)(1-(1-\varepsilon)^6)}+\frac{5(1-\varepsilon)^{18}}{n^3(1-(1-\varepsilon)^6)(1-(1-\varepsilon)^8)}+\\
 \frac{(1-\varepsilon)^{13}(m(6(1-\varepsilon)^2+2)+n(6(1-\varepsilon)^3+2(1-\varepsilon)))}{m^2n^2(1-(1-\varepsilon)^4)(1-(1-\varepsilon)^6)}+\frac{5(1-\varepsilon)^{19}}{n^2m(1-(1-\varepsilon)^6)(1-(1-\varepsilon)^8)}=:\Sigma^e_1.
\end{multline*}

\item a 3-tuple of nested summations:
$$
 \sum_{\ell\text{ even}}\sum_{r=\ell}^{\infty}\sum_{s=r+2}^{\infty}\frac{4(1-\varepsilon)^{2s+\ell+r+12}}{n^3}+
 \sum_{\ell\text{ even}}\sum_{r=\ell+1}^{\infty}\sum_{s=r+1}^{\infty}\frac{4(1-\varepsilon)^{2s+\ell+r+12}}{n^2m}+
$$
$$
 \sum_{r\text{ even}}\sum_{\ell=r+2}^{\infty}\sum_{s=\ell+2}^{\infty}\frac{6(1-\varepsilon)^{2s+\ell+r+12}}{n^3}+
 \sum_{r\text{ odd}}\sum_{\ell=r+1}^{\infty}\sum_{s=\ell+2}^{\infty}\frac{6(1-\varepsilon)^{2s+\ell+r+12}}{n^2m}+
$$
$$
 \sum_{r\text{ even}}\sum_{\ell=r+1}^{\infty}\sum_{s=\ell+1}^{\infty}\frac{6(1-\varepsilon)^{2s+\ell+r+12}}{mn^2}+
 \sum_{r\text{ odd}}\sum_{\ell=r+2}^{\infty}\sum_{s=\ell+1}^{\infty}\frac{6(1-\varepsilon)^{2s+\ell+r+12}}{m^2n}+
$$
$$
 \sum_{\ell\text{ odd}} \sum_{r=\ell+1}^{\infty}\sum_{s=r+2}^{\infty}\frac{4(1-\varepsilon)^{2s+\ell+r+12}}{mn^2}+
 \sum_{\ell\text{ odd}}\sum_{r=\ell}^{\infty}\sum_{s=r+1}^{\infty}\frac{4(1-\varepsilon)^{2s+\ell+r+12}}{m^2n}=
$$
\begin{multline*}
 \frac{(4(1-\varepsilon)^{16}+6(1-\varepsilon)^{22})(m+n(1-\varepsilon))}{n^3m(1-(1-\varepsilon)^4)(1-(1-\varepsilon)^6)(1-(1-\varepsilon)^8)}+\\
 \frac{(1-\varepsilon)^{17}(n(4(1-\varepsilon)+6(1-\varepsilon)^3)+m(6+4(1-\varepsilon)^6))}{m^2n^2(1-(1-\varepsilon)^4)(1-(1-\varepsilon)^6)(1-(1-\varepsilon)^8)}=:\Sigma^e_2.
\end{multline*}
\end{itemize}

In the same way, the fourth summand of the limit probability (related to an odd length at least 2) consists of the following summands:\\

\begin{itemize}

\item no nested summations:

\begin{multline*}
  \frac{(1-\varepsilon)^{8}}{m(1-(1-\varepsilon)^4)}\times\\
  \biggl(\frac{m^2(1+(1-\varepsilon)^2)+(1-\varepsilon)(2+2(1-\varepsilon)+(1-\varepsilon)^2)mn+(1-\varepsilon)^2(1+(1-\varepsilon)^2)n^2)}{m^2n^2}\biggr)+\\
 \frac{(1-\varepsilon)^{11}(m(1+3(1-\varepsilon)^2)+n((1-\varepsilon)+3(1-\varepsilon)^3))}{nm^3(1-(1-\varepsilon)^6)}+\frac{(1-\varepsilon)^{16}}{m^3(1-(1-\varepsilon)^8)}=:\Sigma^o_0;
\end{multline*}

\item a pair of nested summations:

%$$
%  \sum_{s\text{ odd}} \frac{(1-\varepsilon)^{2s+2}}{m}\times
% $$
% $$ 
%  \biggl(\frac{(1-\varepsilon)^2}{n}\left[\sum_{r=0}^{s-1}\frac{2(1-\varepsilon)^{6+r}}{n}+\sum_{r=1}^{s-2}\frac{2(1-\varepsilon)^{6+r}}{m}\right]+
%  \frac{(1-\varepsilon)^3}{m}\left[\sum_{r=0}^{s-1}\frac{2(1-\varepsilon)^{6+r}}{n}+\sum_{r=1}^{s-2}\frac{2(1-\varepsilon)^{6+r}}{m}\right]+
%$$
%$$  
%  \sum_{\ell=0,2,\ldots,s-1}\frac{2(1-\varepsilon)^{4+\ell}}{n}\biggl[(1+2(1-\varepsilon)^2)\left(\frac{(1-\varepsilon)^2}{n}+\frac{(1-\varepsilon)^3}{m}\right)+\frac{(1-\varepsilon)^{6+s}}{m}\biggr]+
%$$  
%$$  
%  \sum_{\ell=1,3,\ldots,s-2}\frac{2(1-\varepsilon)^{4+\ell}}{m}\biggl[(1+2(1-\varepsilon)^2)\left(\frac{(1-\varepsilon)^2}{n}+\frac{(1-\varepsilon)^3}{m}\right)+\frac{(1-\varepsilon)^{6+s}}{m}\biggr]+
 %$$ 
 %$$ 
 % \frac{(1-\varepsilon)^{4+s}}{m}\left[\sum_{r=0}^{s-1}\frac{3(1-\varepsilon)^{6+r}}{n}+\sum_{r=1}^{s-2}\frac{3(1-\varepsilon)^{6+r}}{m}\right]\biggr)=
% $$
\begin{multline*}
 \sum_{r\text{ even}}\sum_{s=r+1}^{\infty}\biggl[\frac{(1-\varepsilon)^{2s+8}}{m}\biggl(\frac{6(1-\varepsilon)^{2+r}+2(1-\varepsilon)^r}{n^2}+\frac{6(1-\varepsilon)^{3+r}+2(1-\varepsilon)^{1+r}}{mn}\biggr)+\\
 \frac{5(1-\varepsilon)^{12+r+3s}}{m^2n}\biggr]+
\end{multline*}
\begin{multline*}
 \sum_{r\text{ odd}}\sum_{s=r+2}^{\infty}\biggl[\frac{(1-\varepsilon)^{2s+8}}{m}\biggl(\frac{6(1-\varepsilon)^{2+r}+2(1-\varepsilon)^r}{nm}+\frac{6(1-\varepsilon)^{3+r}+2(1-\varepsilon)^{1+r}}{m^2}\biggr)+\\
 \frac{5(1-\varepsilon)^{12+r+3s}}{m^3}\biggr]=
\end{multline*}
%$$
% \sum_{r\text{ even}}\left[\frac{(1-\varepsilon)^{3r+10}(m(6(1-\varepsilon)^2+2)+n(6(1-\varepsilon)^3+2(1-\varepsilon)))}{m^2n^2(1-(1-\varepsilon)^4)}+\frac{5(1-\varepsilon)^{15+4r}}{m^2n(1-(1-\varepsilon)^6)}\right]+
%$$
%$$
% \sum_{r\text{ odd}}\left[\frac{(1-\varepsilon)^{3r+12}(m(6(1-\varepsilon)^2+2)+n(6(1-\varepsilon)^3+2(1-\varepsilon)))}{nm^3(1-(1-\varepsilon)^4)}+\frac{5(1-\varepsilon)^{18+4r}}{m^3(1-(1-\varepsilon)^6)}\right]=
%$$
$$
 \frac{(1-\varepsilon)^{10}(m(6(1-\varepsilon)^2+2)+n(6(1-\varepsilon)^3+2(1-\varepsilon)))}{m^2n^2(1-(1-\varepsilon)^4)(1-(1-\varepsilon)^6)}+\frac{5(1-\varepsilon)^{15}}{m^2n(1-(1-\varepsilon)^6)(1-(1-\varepsilon)^8)}+
$$
$$
 \frac{(1-\varepsilon)^{15}(m(6(1-\varepsilon)^2+2)+n(6(1-\varepsilon)^3+2(1-\varepsilon)))}{nm^3(1-(1-\varepsilon)^4)(1-(1-\varepsilon)^6)}+\frac{5(1-\varepsilon)^{22}}{m^3(1-(1-\varepsilon)^6)(1-(1-\varepsilon)^8)}=:\Sigma^o_1;
$$

\item a 3-tuple of summations:
  $$
  \sum_{r\text{ even}}\sum_{\ell=r+2}^{\infty}\sum_{s=\ell+1}^{\infty} 
  \frac{6(1-\varepsilon)^{2s+\ell+r+12}}{mn^2}+
   \sum_{r\text{ odd}}\sum_{\ell=r+1}^{\infty}\sum_{s=\ell+1}^{\infty} 
  \frac{6(1-\varepsilon)^{2s+\ell+r+12}}{m^2n}+
 $$
 $$
 \sum_{\ell\text{ even}}\sum_{r=\ell}^{\infty}\sum_{s=r+1}^{\infty}
  \frac{4(1-\varepsilon)^{2s+\ell+r+12}}{mn^2}+
 \sum_{\ell\text{ even}}\sum_{r=\ell+1}^{\infty}\sum_{s=r+2}^{\infty}
  \frac{4(1-\varepsilon)^{2s+\ell+r+12}}{m^2n}+
 $$
 $$
  \sum_{r\text{ even}}\sum_{\ell=r+1}^{\infty}\sum_{s=\ell+2}^{\infty} 
  \frac{6(1-\varepsilon)^{2s+\ell+r+12}}{m^2n}+
  \sum_{r\text{ odd}}\sum_{\ell=r+2}^{\infty}\sum_{s=\ell+2}^{\infty} 
  \frac{6(1-\varepsilon)^{2s+\ell+r+12}}{m^3}+
 $$ 
 $$
 \sum_{\ell\text{ odd}}\sum_{r=\ell+1}^{\infty}\sum_{s=r+1}^{\infty}
  \frac{4(1-\varepsilon)^{2s+\ell+r+12}}{m^2n}+
 \sum_{\ell\text{ odd}}\sum_{r=\ell}^{\infty}\sum_{s=r+2}^{\infty}
  \frac{4(1-\varepsilon)^{2s+\ell+r+12}}{m^3}=
 $$
\begin{multline*}
  \frac{(1-\varepsilon)^{14}(m(4+6(1-\varepsilon)^6)+n(4(1-\varepsilon)^{5}+6(1-\varepsilon)^{7}))}{m^2n^2(1-(1-\varepsilon)^4)(1-(1-\varepsilon)^6)(1-(1-\varepsilon)^8)}+\\
  \frac{(1-\varepsilon)^{19}(m(6+4(1-\varepsilon)^2)+n(4(1-\varepsilon)+6(1-\varepsilon)^7))}{m^3n(1-(1-\varepsilon)^4)(1-(1-\varepsilon)^6)(1-(1-\varepsilon)^8)}=:\Sigma^0_2.
\end{multline*}
 
 \end{itemize}

\end{document}